\documentclass[reqno]{amsart}

\usepackage{graphicx}
\usepackage[usenames,dvipsnames,svgnames,table]{xcolor}
\usepackage{amssymb}
\usepackage[T1]{fontenc}
\usepackage{tgschola}
\usepackage{mathtools}
\colorlet{symbols}{blue}     

\def\bluesymbol#1{\textcolor{symbols}{#1}}
\def\1{\mathbf{\bluesymbol{1}}}
\def\Xk{\bluesymbol{X^k}}
\def\bf{\bluesymbol{f}}
\def\bg{\bluesymbol{g}}
\def\bpi{\bluesymbol{\pi}}
\def\fbar{\bluesymbol{\bar{f}}}
\def\gbar{\bluesymbol{\bar{g}}}

\def\bTau{\bluesymbol{\tau}}
\def\Taubar{\bluesymbol{\bar{\tau}}}

\def\T{\mathcal{T}}
\def\Tbar{\bar{\T}}
\def\A{\mathcal{A}}
\def\Abar{\bar{\A}}
\def\G{\mathcal{G}}
\def\Gbar{\bar{\G}}

\def\K{\mathcal{K}}

\def\Q{\mathcal{Q}}
\def\E{\mathcal{E}}
\def\Recon{\mathcal{R}}
\def\D{\mathcal{D}}
\def\Fbar{\bar{\mathcal{F}}}
\def\s{\mathfrak{s}}

\newtheorem{theorem}{Theorem}[section]
\newtheorem{lemma}[theorem]{Lemma}
\newtheorem{proposition}[theorem]{Proposition}
\newtheorem{corollary}[theorem]{Corollary}
\newtheorem{assump}[theorem]{Assumption}
\theoremstyle{definition}
\newtheorem{definition}[theorem]{Definition}

\newtheorem{remark}[theorem]{Remark}

\numberwithin{equation}{section}
\allowdisplaybreaks

\DeclareFontFamily{U}{matha}{\hyphenchar\font45}
\DeclareFontShape{U}{matha}{m}{n}{
      <5> <6> <7> <8> <9> <10> gen * matha
      <10.95> matha10 <12> <14.4> <17.28> <20.74> <24.88> matha12
      }{}
\DeclareSymbolFont{matha}{U}{matha}{m}{n}
\DeclareFontSubstitution{U}{matha}{m}{n}

\DeclareFontFamily{U}{mathx}{\hyphenchar\font45}
\DeclareFontShape{U}{mathx}{m}{n}{
      <5> <6> <7> <8> <9> <10>
      <10.95> <12> <14.4> <17.28> <20.74> <24.88>
      mathx10
      }{}
\DeclareSymbolFont{mathx}{U}{mathx}{m}{n}
\DeclareFontSubstitution{U}{mathx}{m}{n}

\DeclareMathDelimiter{\vvvert}{0}{matha}{"7E}{mathx}{"17}

\makeatletter
\newcommand*{\mint}[1]{%
  \mint@l{#1}{}%
}
\newcommand*{\mint@l}[2]{%
  \@ifnextchar\limits{%
    \mint@l{#1}%
  }{%
    \@ifnextchar\nolimits{%
      \mint@l{#1}%
    }{%
      \@ifnextchar\displaylimits{%
        \mint@l{#1}%
      }{%
        \mint@s{#2}{#1}%
      }%
    }%
  }%
}
\newcommand*{\mint@s}[2]{%
  \@ifnextchar_{%
    \mint@sub{#1}{#2}%
  }{%
    \@ifnextchar^{%
      \mint@sup{#1}{#2}%
    }{%
      \mint@{#1}{#2}{}{}%
    }%
  }%
}
\def\mint@sub#1#2_#3{%
  \@ifnextchar^{%
    \mint@sub@sup{#1}{#2}{#3}%
  }{%
    \mint@{#1}{#2}{#3}{}%
  }%
}
\def\mint@sup#1#2^#3{%
  \@ifnextchar_{%
    \mint@sup@sub{#1}{#2}{#3}%
  }{%
    \mint@{#1}{#2}{}{#3}%
  }%
}
\def\mint@sub@sup#1#2#3^#4{%
  \mint@{#1}{#2}{#3}{#4}%
}
\def\mint@sup@sub#1#2#3_#4{%
  \mint@{#1}{#2}{#4}{#3}%
}
\newcommand*{\mint@}[4]{%
  \mathop{}%
  \mkern-\thinmuskip
  \mathchoice{%
    \mint@@{#1}{#2}{#3}{#4}%
        \displaystyle\textstyle\scriptstyle
  }{%
    \mint@@{#1}{#2}{#3}{#4}%
        \textstyle\scriptstyle\scriptstyle
  }{%
    \mint@@{#1}{#2}{#3}{#4}%
        \scriptstyle\scriptscriptstyle\scriptscriptstyle
  }{%
    \mint@@{#1}{#2}{#3}{#4}%
        \scriptscriptstyle\scriptscriptstyle\scriptscriptstyle
  }%
  \mkern-\thinmuskip
  \int#1%
  \ifx\\#3\\\else_{#3}\fi
  \ifx\\#4\\\else^{#4}\fi  
}
\newcommand*{\mint@@}[7]{%
  \begingroup
    \sbox0{$#5\int\m@th$}%
    \sbox2{$#5\int_{}\m@th$}%
    \dimen2=\wd0 %
    \let\mint@limits=#1\relax
    \ifx\mint@limits\relax
      \sbox4{$#5\int_{\kern1sp}^{\kern1sp}\m@th$}%
      \ifdim\wd4>\wd2 %
        \let\mint@limits=\nolimits
      \else
        \let\mint@limits=\limits
      \fi
    \fi
    \ifx\mint@limits\displaylimits
      \ifx#5\displaystyle
        \let\mint@limits=\limits
      \fi
    \fi
    \ifx\mint@limits\limits
      \sbox0{$#7#3\m@th$}%
      \sbox2{$#7#4\m@th$}%
      \ifdim\wd0>\dimen2 %
        \dimen2=\wd0 %
      \fi
      \ifdim\wd2>\dimen2 %
        \dimen2=\wd2 %
      \fi
    \fi
    \rlap{%
      $#5%
        \vcenter{%
          \hbox to\dimen2{%
            \hss
            $#6{#2}\m@th$%
            \hss
          }%
        }%
      $%
    }%
  \endgroup
}

\usepackage{accents}
\newcommand\ubar[1]{\underaccent{\bar}{#1}}

\newcommand{\R}{\mathbf{R}}
\newcommand{\N}{\mathbf{N}}
\newcommand{\Z}{\mathbf{Z}}
\newcommand{\F}{\mathcal{F}}
\newcommand{\B}{\mathfrak{B}}
\newcommand{\BIGOP}[1]{\mathop{\mathchoice%
{\raise-0.22em\hbox{\huge $#1$}}%
{\raise-0.05em\hbox{\Large $#1$}}{\hbox{\large $#1$}}{#1}}}

\newcommand{\BIGboxplus}{\mathop{\mathchoice%
{\raise-0.35em\hbox{\huge $\boxplus$}}%
{\raise-0.15em\hbox{\Large $\boxplus$}}{\hbox{\large $\boxplus$}}{\boxplus}}}

\usepackage[colorlinks,linkcolor=black,citecolor=black,urlcolor=blue]{hyperref}

\begin{document}

\title{Modelled distributions of Triebel--Lizorkin type}
\author[S.\ Hensel]{Sebastian Hensel}
\address{Humboldt-Universitaet zu Berlin, Institut fuer Mathematik, Unter den Linden 6, 10099 Berlin}
\curraddr{Institute of Science and Technology Austria (IST), Am Campus 1, AT-3400 Klosterneuburg}
\email{sebastian.hensel@ist.ac.at}

\author[T.\ Rosati]{Tommaso Rosati}
\address{Humboldt-Universitaet zu Berlin, Institut fuer Mathematik, Unter den Linden 6, 10099 Berlin}
\email{rosatito@hu-berlin.de}



\date{\today}


\keywords{Regularity structures, Triebel--Lizorkin spaces, Reconstruction theorem, 
Embeddings, Schauder estimates}

\begin{abstract}
In order to provide a local description of a regular function in a small neighbourhood of
a point $x$, it is sufficient by Taylor's theorem to know the value of the function
as well as all of its derivatives up to the required order at the point $x$ itself. In other words,
one could say that a regular function is locally modelled by the set of polynomials. The theory
of regularity structures due to Hairer 
generalizes this observation
and provides an abstract setup, which in the application of singular SPDE extends the set
of polynomials by functionals constructed from, e.g., white noise. In this context, the 
notion of Taylor polynomials is lifted to the notion of so-called modelled distributions.
The celebrated reconstruction theorem, which in turn was inspired by Gubinelli's \textit{sewing lemma}, 
is of paramount importance for the theory. It enables to reconstruct
a modelled distribution as a true distribution on $\R^d$ which is locally approximated by this 
extended set of models or ``monomials''.   
In the original work of Hairer, the error is measured based on H\"{o}lder norms.
This was then generalized to the whole scale of Besov spaces by Hairer and Labb\'{e} in
subsequent papers. It is the aim of this work to adapt the analytic part of the theory 
of regularity structures to the scale of Triebel--Lizorkin spaces. 
\end{abstract}

\maketitle

\setcounter{tocdepth}{1}
\tableofcontents

\section{Introduction}

The theory of function spaces represents undoubtedly a vast subject within
the mathematical landscape. One common theme however is given by the
question of how to measure the ``regularity'' or ``smoothness'' of a function. 
Unsurprisingly, an appropriate answer often depends 
on the given context or application. 
From a historical point of view, the most basic notion of regularity is
encoded in terms of the classical spaces of differentiable functions which
come with integer regularity index. By Taylor's theorem, a function is
differentiable up to some fixed order if the small-scale fluctuations of the
function are given, up to some error term, 
by a polynomial with degree of required order.
In this sense, ``regular'' functions are exactly 
those which locally look like polynomials.

Now, in the context of differential equations there are lots of model problems where this point of view 
is simply not appropriate, i.e.\ where it is not expected that the
small-scale fluctuations of the solution are those of polynomials.
This is for example the case for controlled ODEs with rough driving noise, or
for singular stochastic PDEs with white noise as the driving force of the equation.
In the latter example, the deep insight comes from the observation
that one would rather expect the solution to locally look like functionals which are build
from the driving noise. In the last years, two solution theories were developed to formalize
this observation and give for a first time a direct meaning to interesting SPDEs like the KPZ equation, the
2D parabolic Anderson model and generalizations thereof, or the 3D stochastic quantization equation
for the $(\Phi)^4_3$ euclidean quantum field.

On one hand, there is the theory of paracontrolled distributions due to Gubinelli, Imkeller and Perkowski
\cite{Gubinelli:2015}. In this theory, ideas from paradifferential calculus as well as the theory of
controlled rough paths are combined in order to give a rigorous treatment of ill-posed stochastic PDEs.
Another approach, though related to the paracontrolled approach, was developed by Hairer \cite{Hairer:2014}
in his theory of regularity structures. In the language of his theory, 
one would refer to the set of polynomials as a ``model'' for (in the classical sense) differentiable functions.
As we already alluded to above, this is not the right formalization of ``regularity''
in the context of many interesting SPDEs. Instead, the theory of regularity structures develops a framework
which extends the set of ``models'' beyond the set of polynomials in order to provide a useful notion of ``regularity''.
As a consequence, the notion of Taylor polynomials is generalized to this enlarged setting and the
corresponding ``function space'' is given by so-called modelled distributions. It is a key result
of the theory that all modelled distributions can be reconstructed as genuine distributions
which then locally look like the given fixed set of models. This result is referred to
as the \textit{reconstruction theorem}.

In the original work on regularity structures \cite{Hairer:2014}, the space of
modelled distributions under consideration was set up in direct analogy to H\"{o}lder spaces.
The theory in particular allows for potential blow-up on the $t=0$ hyperplane
in order to treat a large class of initial data. Boundary conditions entered the
analysis in form of the initial data, since the work in \cite{Hairer:2014} concentrated specifically on
spatially periodic problems. Thus, we would like to refer the interested reader to
the very recent work of Gerencs\'{e}r and Hairer \cite{Gerencser:2017}, who extended
the original framework to also allow for singularities near the boundary of domains
in the space variable. This generalization then enables the authors to study singular SPDE
with certain boundary conditions, e.g.\ the KPZ equation with Dirichlet and Neumann
conditions or the generalized 2D parabolic Anderson model with Dirichlet conditions.

A first step in the direction of the full Besov scale appeared in the work
of Hairer and Labb\'{e} \cite{Hairer:2015} on multiplicative stochastic heat equations.
The motivation for this generalization developed from the desire to start the equation
from a Dirac mass. One key observation here is that the Dirac mass
possesses improved regularity when considered as an element of the Besov scale
in contrast to H\"{o}lder spaces. This then led to a generalization of the original framework
to analogues of the Besov spaces of type $B^\alpha_{p,\infty}$. We also want to mention at this point the recent work of 
Pr\"{o}mel and Teichmann \cite{Promel:2016} which study Sobolev--Slobodeckij type of
modelled distributions. Recall here that the classical Sobolev--Slobodeckij spaces are
norm equivalent to the Besov spaces $B^\alpha_{p,p}$. The whole scale of Besov spaces
was then eventually treated in a recent work of Hairer and Labb\'{e} \cite{Hairer:2015}.

In this paper, we adapt the analytic part of the theory of regularity
structures to a space of modelled distributions which mimics classical
Triebel--Lizorkin distributions. To this end, let us use the remainder of this introduction
to first describe the key notions from the theory of regularity structures, i.e.\ the notion 
of a regularity structure itself as well as the notion of a model. 
Based on that, we will briefly discuss the results on the Besov scale
of modelled distributions as obtained by Hairer and Labb\'{e} in \cite{Hairer:2017}. 
We eventually conclude the introduction with an outline of the paper.

\subsection*{Regularity structures} We recall the basic notions from the theory
of regularity structures as introduced by Hairer in \cite{Hairer:2014}. 
We also use the opportunity to clarify various notation.
First and foremost, a \textit{regularity structure} consists of a triple $(\A,\T,\G)$
where we call, following Hairer, $\A$ the set of \textit{homogeneities}, $\T$
the \textit{model space} and $\G$ the \textit{structure group}. To form a
regularity structure, the set of homogeneities $\A$ is required to be a subset of $\R$, which
is bounded from below and locally finite. Furthermore, the model space $\T$ is required 
to be a graded vector space of type $\bigoplus_{\zeta\in\A}\T_\zeta$, where each $\T_\zeta$
itself is a Banach space. Finally, the structure group $\G$ is a group of linear maps
$\Gamma\colon\T\to\T$, with the property that for all $\zeta\in\A$, all $\bTau\in\T_\zeta$
as well as all $\Gamma\in\G$ it holds $\Gamma\bTau-\bTau\in\bigoplus_{\beta\in\A\cap(-\infty,\zeta)}\T_{\beta}$.

In the following, we will denote by $\Q_\zeta$ the projection from $\T$ onto $\T_\zeta$.
Furthermore, we define $|\bTau|_\zeta := \|\Q_\zeta\bTau\|_{\T_\zeta}$ for all $\bTau\in\T$.
For notational convenience, let us also set $\A_\gamma:=\A\cap(-\infty,\gamma)$ as well as
$\T_\gamma^-:=\bigoplus_{\beta\in\A_\gamma}\T_\beta$, where $\gamma\in\R$. Then, we
denote by $\Q_{<\gamma}$ the projection of $\T$ onto $\T_\gamma^-$. For all what follows
in this work, given a fixed regularity structure $(\A,\T,\G)$ and a fixed $\gamma>0$,
the integer $r\in\N$ is always assumed to be the smallest positive integer such that
$r>|\min\A|\vee|\max\A_\gamma|$. Finally, once and for all, let us also fix a scaling $\s=(\s_1,\ldots,\s_d)\in\N^d$. 
Given this scaling, we write $\|x\|_\s = \sup |x_i|^{1/{\s_i}}$ for $x\in\R^d$, i.e.\ we consider the $\s$-scaled
``supremum norm'' on $\R^d$. Given $x\in\R^d$ and $R>0$, we then denote by $Q(x,R)$ the balls
centred at $x$ and radius $R$ with respect to this ``norm''. Furthermore, let $A(x,R)$ 
denote the annuli $\{z\in\R^d\colon R/2\leq \|z\|_\s\leq R\}$.

Apart from the algebraic setup represented by the triple $(\A,\T,\G)$, the theory of regularity structures provides
the elements of the model space with some analytical structure. This is the
content of the key notion of a \textit{model}, which is a pair
$(\Pi,\Gamma)$ obeying the following requirements. The object $\Pi$ denotes
a family $(\Pi_x)_{x\in\R^d}$ which itself consists of continuous linear
maps from $\T$ into the space of Schwartz distributions $\D'(\R^d)$
such that, for every $\gamma>0$, the following bound holds true
\begin{align*}
    \|\Pi\|_x := \sup_{\eta\in\B^r}\sup_{\lambda\in(0,1]}\sup_{\zeta\in\A_\gamma}
    \sup_{\bTau\in\T_\zeta}\frac{|\langle\Pi_x\bTau,\eta^\lambda_x\rangle|}{|\bTau|\lambda^\zeta}
    \lesssim 1,
\end{align*}
uniformly over all $x\in\R^d$. Here, we still have to clarify some notation.
First, we denote by $\mathcal{C}^r(\R^d)$ the space of functions $\eta\colon\R^d\to\R$
such that $\eta$ is continuously differentiable for all orders $k\in\N^d$ with scaled
degree $|k|_\s:=k_1\s_1+\cdots+k_d\s_d\leq r$. Then, $\B^r$ denotes the space of all functions
in $\mathcal{C}^\infty(\R^d)$, which are supported in $Q(0,1)$ and with $\mathcal{C}^r(\R^d)$-norm
bounded by 1. In addition, we made use of
\begin{align*}
    \eta^\lambda_x(z) := \lambda^{-|\s|}\eta\big(\lambda^{-\s_1}(z_1{-}x_1),\ldots,\lambda^{-\s_d}(z_d{-}x_d)\big)
\end{align*}
for all $\eta\colon\R^d\to\R$, all $\lambda\in (0,1]$ and all $x,z\in\R^d$. 
Now, we proceed with the object $\Gamma$, which is a map $\R^d\times\R^d\to\G$ such that
\begin{itemize}\itemsep3pt
    \item[i)] $\Gamma_{x,x}=1,\,\Gamma_{x,z}\Gamma_{z,y}=\Gamma_{x,y}$ for all 
              $x,y,z\in\R^d$,
    \item[ii)] $\Pi_y=\Pi_x\Gamma_{x,y}$ for all $x,y\in\R^d$, and
    \item[iii)] the following bound is satisfied
                \begin{align*}
                    \|\Gamma\|_{x,y} := \sup_{\zeta\in\A_\gamma}\sup_{\beta\in\A\cap(-\infty,\zeta]}
                    \sup_{\bTau\in\T_\zeta}\frac{|\Gamma_{x,y}\bTau|_\beta}{|\bTau|\|x{-}y\|_\s^{\zeta-\beta}}
                    \lesssim 1,
                \end{align*}
                uniformly over all $x\in\R^d$ and all $y\in Q(x,1)$.
\end{itemize}
Finally, we will write $$\|\Pi\| := \sup_{x\in\R^d}\|\Pi\|_x,\quad\|\Gamma\| := \sup_{x,y\in\R^d}\|\Gamma\|_{x,y}.$$

An elementary example of a regularity structure is given by the \textit{polynomial regularity structure}
which in the following will always be denoted by $(\Abar, \Tbar, \Gbar)$. The associated set of
homogeneities is simply $\Abar=\N_0$. For $\zeta\in\Abar$, we let $\T_\zeta$ be the linear span
generated by monomials $\Xk=\textcolor{blue}{X_1^{k_1}\cdots X_d^{k_d}}$ with scaled degree $|k|_\s=\zeta$,
i.e.\ in particular $\Tbar=\R[\textcolor{blue}{X_1},\ldots,\textcolor{blue}{X_d}]$. Finally, 
the structure group $\Gbar\sim(\R^d,+)$ acts on polynomials $Q\in\Tbar$ as the group of translations
on $\R^d$, i.e.\
\begin{align*}
    (\Gamma_hQ)(\textcolor{blue}{X}) = Q(\textcolor{blue}{X}+h\1), \quad h\in\R^d.
\end{align*}
Finally, the polynomial structure $(\Abar,\Tbar,\Gbar)$ comes with a canonical model.
This \textit{canonical polynomial model} is given by
\begin{align*}
    (\Pi_x\Xk)(y) = (y-x)^k,\quad \Gamma_{x,y}=\Gamma_{x-y}.
\end{align*}

\subsection*{Besov scale of modelled distributions} Let us now recall from
the work of Hairer and Labb\'{e} \cite{Hairer:2017} the Besov scale of modelled distributions,
which resembles the classical Besov scale on the level of a regularity structure. 
To this end (and for all that follows), $L^p$ will always refer to the Lebesgue space 
$L^p(\R^d,\mathrm{d}x)$ with $x$ denoting the corresponding integration variable.

\begin{definition}
Let $(\A,\T,\mathcal{G})$ be a regularity structure, and let $(\Pi,\Gamma)$ be a model.
Consider $1\leq p \leq\infty$, $1\leq q \leq \infty$ as well as $\gamma > 0$. 
Then, we let $\mathcal{B}_{p,q}^\gamma$ be the (Banach) space of all functions 
$\bf\colon\R^d\to\T_{\gamma}^-$ such that
\begin{itemize}
\item[i)] $\displaystyle\sup_{\zeta\in\A_{\gamma}}\big\||\bf(x)|_\zeta\big\|_{L^p} < \infty$,
\item[ii)] $\displaystyle\sup_{\zeta\in\A_{\gamma}}\bigg(
\int_{Q(0,1)}\bigg\|\frac{|\bf(x{+}h)-\Gamma_{x+h,x}\bf(x)|_\zeta}
{\|h\|_\s^{\gamma-\zeta}}\bigg\|_{L^p}^q\,\frac{\mathrm{d}h}{\|h\|_\s^{|\s|}}\bigg)^\frac{1}{q}<\infty$.
\end{itemize}
The associated norm for $\bf\in\mathcal{B}_{p,q}^\gamma$ is denoted by 
$\vvvert\bf\vvvert_{\mathcal{B}_{p,q}^\gamma}$.
\end{definition}

One of the key results in \cite{Hairer:2017} was to obtain a ``countable description''
of this space. The idea is to work only with local volume means of a modelled distribution.
This leads to a scale of spaces denoted by $\bar{\mathcal{B}}$, cf.\ Definition \ref{def:discreteBesov},
which turns out to be equivalent to the scale $\mathcal{B}$, cf.\ \cite[Thm 2.15]{Hairer:2017}.
The idea to work with local volume means in order to construct the reconstruction operator already appeared in 
the work of Hairer and Labb\'{e} \cite{Hairer:2015} on multiplicative stochastic heat equations.

With this ``countable description'' at their hands, Hairer and Labb\'{e} can proof on one side
the reconstruction theorem for the full scale of Besov type modelled distributions (with non-integer regularity index), 
and on the other side various continuous embeddings which are already well-known in the classical setting of Besov spaces.
They further obtain Schauder type estimates for their spaces of modelled distributions.

The main purpose of the present work is to adapt this theory to spaces of modelled distributions,
which shall resemble the scale of Triebel--Lizorkin distributions in the framework of regularity structures.
To this end, we start in Section \ref{sec:wavelets} with providing the necessary results 
for the classical Triebel--Lizorkin spaces, i.e.\ a wavelet characterization and a convergence criterion.
In a next step, we introduce in Section \ref{sec:modelled} the corresponding space of modelled
distributions. Here, we opt for a definition which is in analogy to the classical characterization
of Triebel--Lizorkin spaces in terms of volume means of differences, cf.\ for
more on characterizations in terms of differences the book of Triebel
\cite[Sec.\ 2.5.11]{Triebel:1983}. We then provide, following the ideas discussed above,
results which show that it is again sufficient to work only with a discrete set of volume
means, cf.\ Proposition \ref{prop:forth} as well as Proposition \ref{prop:back}. 

In Section \ref{sec:embedding}, we make use of this characterization in order to proof embeddings for the spaces
considered in this work. We also provide embeddings which incorporate the Besov scale
of modelled distributions. Section \ref{sec:reconstruction} is then devoted to the proof
of the reconstruction theorem, and in Section \ref{sec:schauder} we show that it is still possible
to obtain Schauder estimates in the framework of the Triebel--Lizorkin scale of modelled distributions.
We also included a short section on products of modelled distributions, cf.\ Section \ref{sec:products}.
Finally, in Section \ref{sec:Besov} we shift our discussion back to the Besov scale.
The motivation comes from the work in \cite{Hairer:2015}, where the authors use volume
means of differences instead of differences itself (as in \cite{Hairer:2017}) for their respective 
space of modelled distributions. In the classical setting of Besov spaces, it is a well-known
fact that this actually makes no difference. We show that this is still true in the framework
of regularity structures.

\section{Some harmonic analysis} \label{sec:wavelets}
The following definition introduces the precise space of distributions which we 
are going to use in this work. The notation $L^q_\lambda$ appearing in the definition
of this space is used as a shortcut for the space $L^q((0,1],\lambda^{-1}\mathrm{d}\lambda)$.
Furthermore, for $\beta\in\N_0$, we denote by $\B^r_\beta(\R^d)$ the subspace of functions
in $\B^r(\R^d)$ which annihilate polynomials with scaled degree at most $\beta$.

\begin{definition}
Let $1\leq p < \infty$, $1\leq q \leq \infty$ and $\alpha\in\R$. Furthermore, fix $r\in\N$
such that $r > |\alpha|$. In the case of $\alpha < 0$, we let $F_{p,q}^\alpha$ be
the (Banach) space of Schwartz distributions $\xi$ on $\R^d$ with
\begin{align}\label{TL1}
    \bigg\|\Big\|\sup_{\eta\in\B^r(\R^d)}
    \frac{|\langle\xi,\eta^{\lambda}_x\rangle|}{\lambda^\alpha}\Big\|_{L^q_\lambda}\bigg\|_{L^p} < \infty.
\end{align}
On the other side, if $\alpha\geq 0$, we require that
\begin{align}\label{TL2}
    \Big\|\sup_{\eta\in\B^r(\R^d)}|\langle\xi,\eta_x\rangle|\Big\|_{L^p} < \infty,\quad
   \bigg\|\Big\|\sup_{\eta\in\B^r_{\lfloor\alpha\rfloor}(\R^d)}
    \frac{|\langle\xi,\eta^{\lambda}_x\rangle|}{\lambda^\alpha}\Big\|_{L^q_\lambda}\bigg\|_{L^p} < \infty. 
\end{align}
No matter which case, we denote by $\|\xi\|_{F^\alpha_{p,q}}$ the corresponding norm.
\end{definition}

Note that this scale of spaces depends on the chosen scaling $\s$ through the definition
of the spaces $\B^r(\R^d)$ and $\B^r_{\lfloor\alpha\rfloor}(\R^d)$. Next, we aim for
a characterization of the spaces $F^\alpha_{p,q}$ in terms of a wavelet analysis. 
Such a characterization is of course not new, cf.\ the book of Triebel \cite{Triebel:2006}, 
but we still prefer to give a proof since we allow for non-trivial scalings of $\R^d$. 
To this end, let us recall the necessary ingredients. For more on wavelets with
compact support and certain regularity, we refer to the work of Daubechies \cite{Daubechies:1988}.

For $r>0$, there is a compactly supported scaling function $\varphi\in\mathcal{C}^r(\R)$ with 
the following properties:
\begin{itemize}\itemsep3pt
\item[i)]   $\int_{\R^d}\mathrm{d}x\,\varphi(x)\varphi(x{-}k) = \delta_{k,0}$
            for every $k\in\Z$,
\item[ii)]  there exists a finite family of constants $(a_k)_{k\in K}$ with $K\subset\Z$
            such that $\varphi(x)=\sum_{k\in K}a_k\varphi(2x{-}k)$, and
\item[iii)] for every $x\in\R^d$ and every polynomial $Q$ of scaled degree at most $r$, we have
            $\sum_{k\in\Z}\int_{\R^d}Q(z)\varphi(z{-}k)\varphi(x{-}k)\,\mathrm{d}z = Q(x).$
\end{itemize}
Now, given such a scaling function $\varphi$ we let $V_n$ be the linear subspace of $L^2(\R^d)$
generated by functions of the form 
\begin{align*}
    \varphi^n_y(x) := 2^{-n\frac{|\s|}{2}}\prod_{i=1}^d\varphi^{2^{-n\s_i}}_{y_i}(x_i),
\end{align*}
where $y\in\Lambda_n$ and
\begin{align*}
    \Lambda_n = \big\{(y_1,\ldots,y_d)\in\R^d\colon y_i=2^{-n\s_i}k_i,\,k_i\in\Z,\,i=1,\ldots,d\big\}.
\end{align*}
The properties of $\varphi$ ensure that $V_n\subset V_{n+1}$. Wavelet theory then also
guarantees the existence of a finite set $\Psi\subset\mathcal{C}^r(\R^d)$ of 
compactly supported functions such that
\begin{itemize}\itemsep3pt
\item[i)] the linear subspace generated by functions of type
          $\psi^n_y:=2^{-n|\s|/2}\psi^{2^{-n}}_y$, where $\psi\in\Psi$ and $y\in\Lambda_n$,
          equals the orthogonal complement of $V_n$ in $V_{n+1}$, for all $n\geq 0$,
\item[ii)] we have $\int_{\R^d}\mathrm{d}x\,x^k\psi(x)=0$
           for every $k\in\N_0^d$ with $|k|_\s\leq r$ and every $\psi\in\Psi$, and
\item[iii)] for all $n\geq 0$, the set $\{\varphi^n_y\colon y\in\Lambda_n\}\cup
            \{\psi^m_y\colon m\geq n,\,\psi\in\Psi,\,y\in\Lambda_m\}$ yields an
            orthonormal basis for $L^2(\R^d)$.
\end{itemize}
Finally, for every $n\geq 0$ and $y\in\Lambda_n$, define the cubes
\begin{align*}
    Q_y^n = [y_1,y_1+2^{-n\s_1})\times\cdots\times[y_d,y_d+2^{-n\s_d})
\end{align*}
and let $\chi_y^n$ denote the associated characteristic functions. Now, the announced wavelet description reads as follows.  

\begin{proposition}\label{prop:waveletBounds}
Let $1\leq p < \infty,\,1\leq q \leq \infty$ and $\alpha\in\R$. In addition, choose $r\in\N$
such that $r>|\alpha|$ and consider a Schwartz distribution $\xi$ on $\R^d$. Then, one has $\xi\in F^\alpha_{p,q}$
if, and only if, the following two bounds hold:
\begin{equation}
\begin{aligned}\label{waveletBounds}
    \Big(\sum_{y\in\Lambda_0} |\langle\xi,\varphi_y^0\rangle|^p\Big)^\frac{1}{p}&<\infty, \\
    \sup_{\psi\in\Psi}\bigg\|\bigg(\sum_{n\geq 0}\sum_{y\in\Lambda_n}
    \frac{|\langle\xi,\psi^n_y\rangle|^q}{2^{-n(\frac{|\s|}{2}+\alpha)q}}
    \chi_y^n(x)\bigg)^\frac{1}{q}\bigg\|_{L^p} &< \infty.
\end{aligned}
\end{equation}
In particular, this gives an equivalent norm for the space $F^\alpha_{p,q}$ which
is encoded in terms of the magnitude of wavelet coefficients.
\end{proposition}

\begin{proof}
We first remark that the ``only if'' part of the statement follows from
the same type of considerations which are already presented in \cite{Hairer:2017}.
Therefore, let us only discuss the ``if'' assertion. In other words,
we assume that the bounds in \eqref{waveletBounds} hold, and we have to verify that these
imply that $\xi\in F^\alpha_{p,q}$. To this end, we distinguish between the cases $\alpha\geq 0$
and $\alpha<0$. Let us begin with the latter one for the case $q<\infty$. 

Following the argument in \cite{Hairer:2017}, we show that the series 
\begin{align}\label{wavDecomp}
\sum_{y\in\Lambda_0}\langle\varphi^0_y,\eta^{\lambda}_x\rangle\langle\xi,\varphi^0_y\rangle +
\sum_{\psi\in\Psi}\sum_{n\geq 0}\sum_{y\in\Lambda_n}
\langle\psi^n_y,\eta^{\lambda}_x\rangle\langle\xi,\psi^n_y\rangle
\end{align}
converges absolutely, and by bounding the two contributions separately, we will
eventually obtain the required bound for $\langle\xi,\eta^{\lambda}_x\rangle$.
To this end, the first term incorporating the scaling function $\varphi$ is treated exactly as in \cite{Hairer:2017},
which also explains the occurrence of the first bound in \eqref{waveletBounds}.
In particular, we only discuss the second term in detail. Using the triangle inequality for 
the sum over $\psi\in\Psi$ this amounts to bound the term
\begin{align}\label{aux1}
    \bigg\|\bigg(\sum_{n_0=0}^\infty\bigg\|\sup_{\eta\in\mathfrak{B}^r(\R^d)}
    \frac{\sum_{n\geq 0}\sum_{y\in\Lambda_n}|\langle\psi^n_y,\eta^{\lambda}_x\rangle||\langle\xi,\psi^n_y\rangle|}
    {\lambda^\alpha}\bigg\|_{L^q_{\lambda,n_0}}^q\bigg)^\frac{1}{q}\bigg\|_{L^p}
\end{align}
where we introduced the shortcut $L^q_{\lambda,n_0}$ to denote $L^q((2^{-(n_0+1)},2^{-n_0}],\lambda^{-1}\mathrm{d}\lambda)$.
In addition, let $M$ denote the smallest integer such that $M$ is greater or equal to the maximum sizes
of the supports of $\varphi$ and $\psi\in\Psi$. For fixed $n_0\in\N_0$, we separate the sum over $n\geq 0$ 
into terms corresponding to $n<n_0$ and $n\geq n_0$, respectively. For such fixed $n_0\in\N_0$, we then note that
(cf.\ \cite{Hairer:2017})
\begin{align}\label{largeScales}
    |\langle\psi^n_y,\eta^{\lambda}_x\rangle| \lesssim 2^{n\frac{|\s|}{2}}
    \chi_{Q(x,\lambda+M2^{-n})}(y),
\end{align}
being valid uniformly over all $\psi\in\Psi$, all $x\in\R^d$, all $n<n_0$,
all $y\in\Lambda_n$, all $\lambda\in (2^{-(n_0+1)},2^{-n_0}]$
and all $\eta\in\mathfrak{B}^r(\R^d)$. Furthermore, it holds
\begin{align}\label{smallScales}
    |\langle\psi^n_y,\eta^{\lambda}_x\rangle| \lesssim 
    2^{-(n-n_0)(\frac{|\s|}{2}+r)}2^{n_0\frac{|\s|}{2}}\chi_{Q(x,\lambda+M2^{-n})}(y),
\end{align}
uniformly over all $\psi\in\Psi$, all $x\in\R^d$, all $n\geq n_0$,
all $y\in\Lambda_n$, all $\lambda\in (2^{-(n_0+1)},2^{-n_0}]$
and all $\eta\in\mathfrak{B}^r(\R^d)$. 

Now, by means of \eqref{largeScales} we obtain
\begin{align*}
    \bigg\|\sup_{\eta\in\mathfrak{B}^r(\R^d)}\sum_{n<n_0}&\sum_{y\in\Lambda_n}
    \frac{|\langle\psi^n_y,\eta^{\lambda}_x\rangle||\langle\xi,\psi^n_y\rangle|}{\lambda^\alpha}\bigg\|_{L^q_{\lambda,n_0}}^q \\
    &\lesssim\sum_{n<n_0}2^{\alpha(n_0-n)}\bigg|\sum_{\substack{y\in\Lambda_n \\ \|x-y\|_\s \leq 2^{-n}(M+1)}}
    \frac{|\langle\xi,\psi^n_y\rangle|}{2^{-n(\frac{|\s|}{2}+\alpha)}}\bigg|^q.
\end{align*}
Here, we used Jensen's inequality on the sum over $n<n_0$. In addition, using \eqref{smallScales} yields the bound
\begin{align*}
    \bigg\|\sup_{\eta\in\mathfrak{B}^r(\R^d)}&\sum_{n\geq n_0}\sum_{y\in\Lambda_n}
    \frac{|\langle\psi^n_y,\eta^{\lambda}_x\rangle||\langle\xi,\psi^n_y\rangle|}{\lambda^\alpha}\bigg\|_{L^q_{\lambda,n_0}}^q \\
    &\lesssim\sum_{n\geq n_0}2^{-(n-n_0)(r+\alpha)}\bigg|
    \sum_{\substack{y\in\Lambda_n \\ \|x-y\|_\s\leq 2^{-{n_0}}(M+1)}}2^{(n_0-n)|\s|}
    \frac{|\langle\xi,\psi^n_y\rangle|}{ 2^{-n(\frac{|\s|}{2}+\alpha)}}\bigg|^q.
\end{align*}
This time, we exploited Jensen's inequality on the sum over $n\geq n_0$.
The next step now consists of appropriately estimating the restricted 
sums over $y\in\Lambda_n$ in this last two bounds.

We begin with the large scales regime $n<n_0$. Recall the notation $Q^n_y$. Let $\Lambda_n^{M,x}$
denote the smallest subset of $\Lambda_n$ such that the cube $Q(x,2^{-n}(M{+}2))$ is covered
by the disjoint union of cubes $Q^n_y$ with $y\in\Lambda_n^{M,x}$. Note that $\Lambda_n^{M,x}$
contains the set of all $y\in\Lambda_n$ such that $\|y-x\|_\s\leq 2^{-n}(M{+}1)$. Hence,
for any $\eta\in (0,1)$ we obtain due to the monotonicity of $\ell^s$-norms that
\begin{align*}
    \bigg|\sum_{\substack{y\in\Lambda_n \\ \|x-y\|_\s \leq 2^{-n}(M+1)}}
    \frac{|\langle\xi,\psi^n_y\rangle|}{2^{-n(\frac{|\s|}{2}+\alpha)}}\bigg|^\eta &\leq
    \sum_{\substack{y\in\Lambda_n \\ \|x-y\|_\s \leq 2^{-n}(M+1)}}
    \frac{|\langle\xi,\psi^n_y\rangle|^\eta}{2^{-n(\frac{|\s|}{2}+\alpha)\eta}} \\
    &\lesssim \mint{-}_{Q(x,2^{-n}(M{+}2))}\bigg|\sum_{y\in\Lambda_n^{M,x}}
    \frac{|\langle\xi,\psi^n_y\rangle|}{2^{-n(\frac{|\s|}{2}+\alpha)}}\chi^n_y(z)\bigg|^\eta\,\mathrm{d}z \\
    &\leq\mathcal{M}\Big(\Big|\sum_{y\in\Lambda_n^{M,x}}
    \frac{|\langle\xi,\psi^n_y\rangle|}{2^{-n(\frac{|\s|}{2}+\alpha)}}\chi^n_y\Big|^\eta\Big)(x),
\end{align*}
where $\mathcal{M}f$ denotes the Hardy--Littlewood maximal function for a 
locally integrable function $f\in L^1_{\mathrm{loc}}(\R^d)$.
Due to the additional factor $2^{(n_0-n)|\s|}$, an analogous type of argument shows for
the small scales regime $n\geq n_0$ that
\begin{align*}
    \sum_{\substack{y\in\Lambda_n \\ \|x-y\|_\s\leq 2^{-{n_0}}(M+1)}}&2^{(n_0-n)|\s|}
    \frac{|\langle\xi,\psi^n_y\rangle|}{ 2^{-n(\frac{|\s|}{2}+\alpha)}} \\
    &\lesssim2^{-(n-n_0)\frac{\eta-1}{\eta}|\s|}\bigg\{\mathcal{M}\Big(\Big|\sum_{y\in\Lambda_n^{M,x}}
    \frac{|\langle\xi,\psi^n_y\rangle|}{2^{-n(\frac{|\s|}{2}+\alpha)}}\chi^n_y\Big|^\eta\Big)(x)\bigg\}^\frac{1}{\eta},
\end{align*}
which again holds true for every $\eta\in (0,1)$. Of course, this time one writes $\Lambda_n^{M,x}$
to denote the smallest subset of $\Lambda_n$ such that the cube $Q(x,2^{-n_0}(M{+}2))$ is covered
by the disjoint union of cubes $Q^n_y$ with $y\in\Lambda_n^{M,x}$.

Let us summarize what we have achieved so far. Taking together all of our bounds, we see that
the term in \eqref{aux1} is bounded by
\begin{align}\label{aux2}
    \bigg\|\bigg(\sum_{n_0=0}^\infty\sum_{n=0}^\infty\theta(n{-}n_0)
    \bigg\{\mathcal{M}\Big(\Big|\sum_{y\in\Lambda_n}\frac{|\langle\xi,\psi^n_y\rangle|}{2^{-n(\frac{|\s|}{2}+\alpha)}}
    \chi^n_y\Big|^\eta\Big)(x)\bigg\}^\frac{q}{\eta}\bigg)^\frac{1}{q}\bigg\|_{L^p},
\end{align}
where we introduced the sequence $(\theta(z))_{z\in\Z}$ by
\begin{align*}
    \theta(z) = \begin{cases}
                2^{-z(r+\alpha+\frac{\eta-1}{\eta}q|\s|)}, & z\geq 0, \\
                2^{-z\alpha},     & z < 0.
                \end{cases}
\end{align*}
Choosing $0<\eta<1$ such that $r+\alpha+\frac{\eta{-}1}{\eta}q|\s|>0$, it follows
$\|\theta\|_{\ell^1(\Z)}<\infty$. In particular, due to Young's inequality for convolutions
the quantity from \eqref{aux2} can be estimated by
\begin{align*}
    \bigg\|\bigg(\sum_{n_0=0}^\infty\bigg\{\mathcal{M}\Big(\Big|\sum_{y\in\Lambda_{n_0}}
    \frac{|\langle\xi,\psi^{n_0}_y\rangle|}{2^{-n_0(\frac{|\s|}{2}+\alpha)}}
    \chi^{n_0}_y\Big|^\eta\Big)(x)\bigg\}^\frac{q}{\eta}\bigg)^\frac{1}{q}\bigg\|_{L^p}.
\end{align*}
Next, since we have $\eta < p\wedge q$ we can also apply the 
Fefferman--Stein maximal inequality (cf.\ \cite{Fefferman:1971})
to bound this by
\begin{align*}
    \bigg\|\bigg(\sum_{n\geq 0}\sum_{y\in\Lambda_n}
    \frac{|\langle\xi,\psi^n_y\rangle|^q}{2^{-n(\frac{|\s|}{2}+\alpha)q}}
    \chi_y^n(x)\bigg)^\frac{1}{q}\bigg\|_{L^p},
\end{align*}
which is finite by assumption, i.e.\ $\xi\in F^\alpha_{p,q}$ as asserted.

Let us discuss briefly the case $q=\infty$. We also still assume that $\alpha < 0$. Then, a similar
argument as above shows that
\begin{align*}
    \bigg\|\sup_{\lambda\in (0,1]}\sup_{\eta\in\mathfrak{B}^r(\R^d)}
    \frac{|\langle\xi,\eta_x^{\lambda}\rangle|}{\lambda^\alpha}\bigg\|_{L^p}
    &\lesssim\bigg\|\sup_{n\in\N_0}\bigg\{\mathcal{M}\Big(\Big|\sum_{y\in\Lambda_n}
    \frac{|\langle\xi,\psi^n_y\rangle|}{2^{-n(\frac{|\s|}{2}+\alpha)}}
    \chi^n_y\Big|^\eta\Big)(x)\bigg\}^\frac{1}{\eta}\bigg\|_{L^p},
\end{align*}
where $\eta\in (0,1)$. Again, it remains to make use of the vector-valued maximal inequality, i.e.\ we obtain
\begin{align*}
    \bigg\|\sup_{\lambda\in (0,1]}\sup_{\eta\in\mathfrak{B}^r(\R^d)}
    \frac{|\langle\xi,\eta_x^{\lambda}\rangle|}{\lambda^\alpha}\bigg\|_{L^p}
    &\lesssim\bigg\|\sup_{n\in\N_0}\sum_{y\in\Lambda_n}\frac{|\langle\xi,\psi^n_y\rangle|}
    {2^{-n(\frac{|\s|}{2}+\alpha)}}\chi^n_y(x)\bigg\|_{L^p},
\end{align*}
which concludes the proof for the case $q=\infty$.

Now, we move on to the discussion of the case $\alpha\geq 0$. Recall from the definition of the
spaces $F^\alpha_{p,q}$ that this amounts to bound the two quantities from \eqref{TL2}. With
respect to the latter one, the argument works verbatim as the one presented above except that
one has to replace the bound in \eqref{largeScales} with
\begin{align}\label{LargeScales}
    |\langle\psi^n_y,\eta^{\lambda}_x\rangle| \lesssim 
    2^{n\frac{|\s|}{2}+\lfloor\alpha\rfloor+1}\lambda^{\lfloor\alpha\rfloor+1}
    \chi_{Q(x,\lambda+M2^{-n})}(y),
\end{align}
which holds uniformly over all $\psi\in\Psi$, all $x\in\R^d$, all $n<n_0$,
all $y\in\Lambda_n$, all $\lambda\in (2^{-(n_0+1)},2^{-n_0}]$
and all $\eta\in\mathfrak{B}^r_{\lfloor\alpha\rfloor}(\R^d)$. Therefore,
we only deal with the first quantity from \eqref{TL2}. To this end,
we again resort to the wavelet decomposition from \eqref{wavDecomp}---this time
with $\lambda=1$. The sum involving the scaling
function $\varphi$ can be bounded by the same arguments presented in \cite{Hairer:2017}.
Hence, we briefly discuss the sum involving the functions $\psi\in\Psi$.

Note that this time there is no integration in $\lambda\in (0,1]$.
In other words, we can use the arguments presented above, formally substituting
$\lambda = 1$, i.e.\ $n_0=0$. In particular, it is readily seen
that the following bound holds (with the obvious modification in the case $q=\infty$):
\begin{align*}
    \bigg\|\sup_{\eta\in\mathfrak{B}^r(\R^d)}&\sum_{n\geq 0}\sum_{y\in\Lambda_n}
    |\langle\xi,\psi^n_y\rangle||\langle\psi^n_y,\eta_x\rangle|\bigg\|_{L^p} \\
    &\lesssim\bigg\|\sum_{n\geq 0}2^{-n(r+\alpha+\frac{\eta{-}1}{\eta}|\s|)}
    \bigg\{\mathcal{M}\Big(\Big|\sum_{y\in\Lambda_n}\frac{|\langle\xi,\psi^n_y\rangle|}{2^{-n(\frac{|\s|}{2}+\alpha)}}
    \chi^n_y\Big|^\eta\Big)(x)\bigg\}^\frac{1}{\eta}\bigg\|_{L^p} \\
    &\lesssim\bigg\|\bigg(\sum_{n\geq 0}
    \bigg\{\mathcal{M}\Big(\Big|\sum_{y\in\Lambda_n}\frac{|\langle\xi,\psi^n_y\rangle|}{2^{-n(\frac{|\s|}{2}+\alpha)}}
    \chi^n_y\Big|^\eta\Big)(x)\bigg\}^\frac{q}{\eta}\bigg)^\frac{1}{q}\bigg\|_{L^p}.
\end{align*}
In the last step, we made use of H\"{o}lder's inequality, and that one can choose $0<\eta<1$ such that $r+\alpha+\frac{\eta{-}1}{\eta}|\s|>0$
Hence, we can conclude the proof of the ``if'' assertion by another application of the Fefferman--Stein maximal inequality.
\end{proof}

Consider now a sequence $(\xi_n)_{n\in\N_0}$ given by
\begin{align*}
    \xi_n = \sum_{y\in\Lambda_n} A^n_y\varphi^n_y,
\end{align*}
and define the quantities $\delta A^n_y = \langle\xi_{n+1}-\xi_n,\varphi^n_y\rangle$.
For the proof of the reconstruction theorem, the following convergence criterion in $F^\alpha_{p,q}$
for the sequence $(\xi_n)_{n\in\N}$ will be of importance. Of course, this criterion is the 
counterpart to the corresponding assertion in \cite[Prop 3.7]{Hairer:2017}. 

\begin{proposition}\label{proposition:characterisation_convergence}
Let $\gamma>0$, $\alpha < 0$, $1\leq p <\infty$ and $1\leq q \leq \infty$. 
Furthermore, we assume that the following two bounds hold:
\begin{align}\label{convCrit1}
\bigg\| \sup_{n\geq 0} \sum_{y\in\Lambda_n} \frac{|A^n_y|} 
{2^{-n(\alpha+\frac{|\s|}{2})}}\chi^n_y(x) \bigg\|_{L^p} &< \infty, \\\label{convCrit2}
\bigg\|\bigg(\sum_{n\geq 0}\Big|\sum_{y\in\Lambda_n}
\frac{|\delta A^n_y|}{2^{-n(\gamma+\frac{|\s|}{2})}}\chi^n_y(x)
\Big|^q\bigg)^\frac{1}{q}\bigg\|_{L^p} &< \infty.
\end{align}
Then, for every $\bar{\alpha}<\alpha$, the sequence $(\xi_n)_{n\in\N}$ converges in $F^{\bar{\alpha}}_{p,q}$.
Moreover, if $q=\infty$ the limit distribution is actually an element of $F^\alpha_{p,q}$.
\end{proposition}

\begin{proof}
We follow the strategy of Hairer and Labb\'{e} \cite{Hairer:2017}, adapting some arguments as already seen
in the wavelet characterization of the space $F^\alpha_{p,q}$. In particular, we will again
make use of the Fefferman--Stein maximal inequality. But first, we decompose
\begin{align*}
    \xi_{n+1}-\xi_n = g_n + \delta\xi_n \in V_n\oplus V_n^\perp = V_{n+1},
\end{align*}
where
\begin{align*}
    g_n = \sum_{z\in\Lambda_n}\delta A^n_z\varphi^n_z,\quad
    \delta\xi_n = \sum_{z\in\Lambda_n}\Big(\sum_{u\in\Lambda_{n+1}}
    A^{n+1}_u\langle\varphi^{n+1}_u,\psi^n_z\rangle\Big)\,\psi^n_z.
\end{align*}
Now, fix positive integers $k\leq K$. We aim to bound the wavelet norm of the sums $\sum_{n=k}^K g_n$
and $\sum_{n=k}^K \delta\xi_n$ in $F^{\bar{\alpha}}_{p,q}$, for every $\bar{\alpha}<\alpha$.
Let us treat first the contributions coming from $\delta\xi_n$ (with an obvious modification
of the argument for the case $q=\infty$). To this end, note first that
\begin{align*}
    \bigg\|\bigg(\sum_{m\geq 0}&\sum_{y\in\Lambda_m}
    \frac{|\langle\sum_{n=k}^K\delta\xi_n,\psi^m_y\rangle|^q}
    {2^{-m(\bar{\alpha}+\frac{|\s|}{2})q}}\chi^m_y(x)\bigg)^\frac{1}{q}\bigg\|_{L^p} \\
    &=\bigg\|\bigg(\sum_{k\leq n\leq K}\sum_{y\in\Lambda_n}
    \frac{|\sum_{z\in\Lambda_{n+1}}A^{n+1}_z\langle\varphi^{n+1}_z,\psi^n_y\rangle|^q}
    {2^{-n(\bar{\alpha}+\frac{|\s|}{2})q}}\chi^n_y(x)\bigg)^\frac{1}{q}\bigg\|_{L^p}.
\end{align*}
Then, we use that $|\langle\varphi^{n+1}_z,\psi^n_y\rangle|\lesssim 2^{-|\s|/2}$
uniformly over the respective parameters. Actually, the left hand side vanishes whenever
$\|y-z\|_\s>(2M)2^{-n}$. Hence, we deduce that
\begin{align*}
    \bigg\|&\bigg(\sum_{k\leq n\leq K}\sum_{y\in\Lambda_n}
    \frac{|\sum_{z\in\Lambda_{n+1}}A^{n+1}_z\langle\varphi^{n+1}_z,\psi^n_y\rangle|^q}
    {2^{-n(\bar{\alpha}+\frac{|\s|}{2})q}}\chi^n_y(x)\bigg)^\frac{1}{q}\bigg\|_{L^p} \\
    &\hspace{1cm}\lesssim\bigg\|\bigg(\sum_{k\leq n\leq K}\sum_{y\in\Lambda_n}
    \Big|\sum_{\substack{z\in\Lambda_{n+1} \\ \|z-x\|_\s\leq C2^{-n}}}
    2^{-|\s|}\frac{|A^{n+1}_z|}{2^{-(n+1)(\bar{\alpha}+\frac{|\s|}{2})}}
    \Big|^q\chi^n_y(x)\bigg)^\frac{1}{q}\bigg\|_{L^p},
\end{align*}
with $C=2M+1$. Now, due to the factor $2^{-|\s|}$ we can proceed similarly as in the proof
of the wavelet characterization to obtain (with $0<\eta<1$)
\begin{align*}
    \bigg\|\bigg(\sum_{k\leq n\leq K}\sum_{y\in\Lambda_n}&
    \Big|\sum_{\substack{z\in\Lambda_{n+1} \\ \|z-x\|_\s\leq C2^{-n}}}
    2^{-|\s|}\frac{|A^{n+1}_z|}{2^{-(n+1)(\bar{\alpha}+\frac{|\s|}{2})}}
    \Big|^q\chi^n_y(x)\bigg)^\frac{1}{q}\bigg\|_{L^p} \\
    &\lesssim\bigg\|\bigg(\sum_{k\leq n\leq K}\bigg\{\mathcal{M}\Big(\Big|\sum_{z\in\Lambda_{n+1}}
    \frac{|A^{n+1}_z|}{2^{-(n+1)(\bar{\alpha}+\frac{|\s|}{2})}}\chi^n_z\Big|^\eta\Big)(x)
    \bigg\}^\frac{q}{\eta}\bigg)^\frac{1}{q}\bigg\|_{L^p} \\
    &\lesssim\bigg\|\bigg(\sum_{k\leq n\leq K}\Big|\sum_{z\in\Lambda_{n+1}}
    \frac{|A^{n+1}_z|}{2^{-(n+1)(\bar{\alpha}+\frac{|\s|}{2})}}\chi^n_z(x)\Big|^q
    \bigg)^\frac{1}{q}\bigg\|_{L^p},
\end{align*}
where the last step, of course, follows from the Fefferman--Stein maximal inequality.
From this, we immediately infer that
\begin{align*}
    \bigg\|\bigg(\sum_{m\geq 0}\sum_{y\in\Lambda_m}&
    \frac{|\langle\sum_{n=k}^K\delta\xi_n,\psi^m_y\rangle|^q}
    {2^{-m(\bar{\alpha}+\frac{|\s|}{2})q}}\chi^m_y(x)\bigg)^\frac{1}{q}\bigg\|_{L^p} \\ 
    &\lesssim 2^{-k(\alpha-\bar{\alpha})}\bigg\| \sup_{n\geq 0} \sum_{y\in\Lambda_n} 
    \frac{|A^n_y|} {2^{-n(\alpha+\frac{|\s|}{2})}} \chi^n_y(x) \bigg\|_{L^p}.
\end{align*}

In the second step we treat the contributions from the $g_n$. To this end, we only discuss
the bound related to the second quantity in \eqref{waveletBounds}. As a matter of fact,
the bound related to the first quantity in \eqref{waveletBounds} follows from similar
considerations, which is why we refrain from providing details.
Moreover, the case $q = \infty$ again follows from an obvious modification of the argument 
for the case $q < \infty$. For this reason, we will only treat the latter one in detail.
First of all, we have
\begin{align*}
    \bigg\|\bigg(\sum_{m\geq 0}&\sum_{y\in\Lambda_m}
    \frac{|\langle\sum_{n=k}^K g_n,\psi^m_y\rangle|^q}{2^{-m(\bar{\alpha}+\frac{|\s|}{2})q}}
    \chi^m_y(x)\bigg)^\frac{1}{q}\bigg\|_{L^p} \\
    &\leq\bigg\|\bigg(\sum_{m\geq 0}\sum_{y\in\Lambda_m}\bigg(\sum_{(m{+}1)\vee k\leq n\leq K}\sum_{z\in\Lambda_n}
    \frac{|\delta A^n_z||\langle\varphi^n_z,\psi^m_y\rangle|}{2^{-m(\bar{\alpha}+\frac{|\s|}{2})}}
    \bigg)^q\chi^m_y(x)\bigg)^\frac{1}{q}\bigg\|_{L^p}.
\end{align*}
Then, observe that $|\langle\varphi^n_z,\psi^m_y\rangle|\lesssim 2^{-(n-m)\frac{|\s|}{2}}$
which holds uniformly over all the respective parameters---in particular for all $m<n$. 
This time, the left hand side vanishes as soon as $\|z-y\|_\s > (2M)2^{-m}$. 
In addition, let us choose $0<\eta<1$ such that one has $\bar{\gamma}:=\gamma+\frac{\eta{-}1}{\eta}|\s|>0$.
We also put $\alpha^\prime := \bar{\alpha}+\frac{\eta{-}1}{\eta}|\s|<0$. Then, we obtain the bound
\begin{align*}
    &\bigg(\sum_{(m{+}1)\vee k\leq n\leq K}\sum_{z\in\Lambda_n}
    \frac{|\delta A^n_z||\langle\varphi^n_z,\psi^m_y\rangle|}
    {2^{-m(\bar{\alpha}+\frac{|\s|}{2})}}\bigg)^q\chi^m_y(x) \\
    &\hspace{0.5cm}\lesssim 2^{\bar{\alpha}qm}\bigg(
    \sum_{(m{+}1)\vee k\leq n\leq K}2^{-n\gamma}
    \sum_{\substack{z\in\Lambda_n \\ \|z-y\|_\s\leq (2M)2^{-m}}}
    2^{(m-n)|\s|}\frac{|\delta A^n_z|}{2^{-n(\gamma+\frac{|\s|}{2})}}\bigg)^q\chi_y^m(x) \\
    &\hspace{0.5cm}\lesssim 2^{\alpha^\prime qm}
    \sum_{(m{+}1)\vee k\leq n\leq K}2^{-n\bar{\gamma}}\bigg(
    \sum_{\substack{z\in\Lambda_n \\ \|z-x\|_\s\leq C2^{-m}}}
    2^{(m-n)|\s|}\frac{|\delta A^n_z|^\eta}{2^{-n(\gamma+\frac{|\s|}{2})\eta}}
    \bigg)^\frac{q}{\eta}\chi_y^m(x),
\end{align*}
where the last line is a consequence of the monotonicity of $\ell^s$-norms followed by 
an application of Jensen's inequality. From this, we can infer that
\begin{align*}
     \bigg\|\bigg(&\sum_{m\geq 0}\sum_{y\in\Lambda_m}
    \frac{|\langle\sum_{n=k}^K g_n,\psi^m_y\rangle|^q}{2^{-m(\bar{\alpha}+\frac{|\s|}{2})q}}
    \chi^m_y(x)\bigg)^\frac{1}{q}\bigg\|_{L^p} \\
    &\lesssim 2^{-k\bar{\gamma}}\bigg\|\bigg(\sum_{m\geq 0}\sum_{n\geq m}2^{\alpha^\prime qm}
    \bigg|\mint{-}_{Q(x,C^\prime 2^{-m})}\Big|\sum_{z\in\Lambda_n}
    \frac{|\delta A^n_z|^\eta}{2^{-n(\gamma+\frac{|\s|}{2})\eta}}\chi^n_z(u)
    \Big|^\eta\,\mathrm{d}u\bigg|^\frac{q}{\eta}\bigg)^\frac{1}{q}\bigg\|_{L^p} \\
    &\lesssim 2^{-k\bar{\gamma}}\bigg\|\bigg(\sum_{n\geq 0}
    \bigg\{\mathcal{M}\Big(\Big|\sum_{z\in\Lambda_n}
    \frac{|\delta A^n_z|^\eta}{2^{-n(\gamma+\frac{|\s|}{2})\eta}}\chi^n_z
    \Big|^\eta\Big)(x)\bigg\}^\frac{q}{\eta}\bigg)^\frac{1}{q}\bigg\|_{L^p}.
\end{align*}
Here, we in particular made use of $\sum_{m\geq 0}2^{\alpha^\prime qm} \sim 1$ due to $\alpha^\prime < 0$.
As it is by now routine, we can proceed from this point with the help of the Fefferman--Stein maximal 
inequality to deduce that
\begin{align*}
    \bigg\|\bigg(\sum_{m\geq 0}\sum_{y\in\Lambda_m}&
    \frac{|\langle\sum_{n=k}^K g_n,\psi^m_y\rangle|^q}{2^{-m(\bar{\alpha}+\frac{|\s|}{2})q}}
    \chi^m_y(x)\bigg)^\frac{1}{q}\bigg\|_{L^p} \\
    &\lesssim 2^{-k\bar{\gamma}}\bigg\|\bigg(\sum_{n\geq 0}\Big|\sum_{y\in\Lambda_n}
    \frac{|\delta A^n_y|}{2^{-n(\gamma+\frac{|\s|}{2})}}\chi^n_y(x)
    \Big|^q\bigg)^\frac{1}{q}\bigg\|_{L^p}.
\end{align*}
All in all, we obtain the desired convergence by making use of our assumptions \eqref{convCrit1} 
and \eqref{convCrit2}, respectively. Furthermore, a second view on our argument 
for the contributions due to the $\delta\xi_n$ reveals that we also have the bound
\begin{align*}
    \bigg\|\sup_{m\geq 0}\sum_{y\in\Lambda_m}
    \frac{|\langle\sum_{n=k}^K\delta\xi_n,\psi^m_y\rangle|}
    {2^{-m(\alpha+\frac{|\s|}{2})}}\chi^m_y(x)\bigg\|_{L^p} 
    \lesssim \bigg\| \sup_{n\geq 0} \sum_{y\in\Lambda_n} 
    \frac{|A^n_y|} {2^{-n(\alpha+\frac{|\s|}{2})}} \chi^n_y(x) \bigg\|_{L^p},
\end{align*}
uniformly over all $k\geq 0$ and all $K\geq k$. As the argument for the contributions due to
the $g_n$ works verbatim for the case $\bar{\alpha} = \alpha$, we deduce that in case of $q=\infty$
the limit distribution actually lives in $F^\alpha_{p,q}$.
\end{proof}

\section{Spaces of modelled distributions} \label{sec:modelled}
In this section, we introduce the corresponding scale of spaces of modelled distributions 
which shall resemble the scale $F_{p,q}^\alpha$ on the level of a regularity structure. 
To this end, we want to follow as close as possible the characterization of Triebel--Lizorkin spaces
by volume means of differences, cf.\ the book of Triebel \cite[Sec.\ 2.5.11]{Triebel:1983}. 
Therefore, we opt for the following definition.

\begin{definition}
Let $(\A,\T,\mathcal{G})$ be a regularity structure, and let $(\Pi,\Gamma)$ be a model.
Consider $1\leq p <\infty$, $1\leq q \leq \infty$ as well as $\gamma > 0$. 
Then, we let $\F_{p,q}^\gamma$ be the (Banach) space of all functions 
$\bf\colon\R^d\to\T_{\gamma}^-$ such that
\begin{itemize}
\item[i)] $\displaystyle\sup_{\zeta\in\A_{\gamma}}\big\||\bf(x)|_\zeta\big\|_{L^p} < \infty$,
\item[ii)] $\displaystyle\sup_{\zeta\in\A_{\gamma}}\bigg\|\Big\|
\mint{-}_{Q(0,4\lambda)}\frac{|\bf(x{+}h)-\Gamma_{x+h,x}\bf(x)|_\zeta}
{\lambda^{\gamma-\zeta}}\,\mathrm{d}h\Big\|_{L^q_\lambda}\bigg\|_{L^p}<\infty$.
\end{itemize}
In the sequel, we will often make use of the difference operator
\begin{align*}
    \Delta_{\Gamma,h}\bf(x) = \bf(x{+}h)-\Gamma_{x+h,x}\bf(x).
\end{align*}
Finally, the associated norm for $\bf\in\F_{p,q}^\gamma$ is denoted by 
$\vvvert\bf\vvvert_{\F_{p,q}^\gamma}$.
\end{definition}

Recall that one particular useful fact of having a wavelet characterization for the space
of distributions $F^\alpha_{p,q}$ is that norms of sequence spaces are typically more
easily analyzed than norms of function spaces, say $L^p$-norms. For the same reason, it turns out that
one should look for a suitable ``discretized'' characterization for the Triebel--Lizorkin scale of modelled distributions
$\F_{p,q}^\gamma$. This was actually one of the key technical achievements in the work of Hairer and Labb\'{e} \cite{Hairer:2017}
for the Besov scale of modelled distributions. Hence, it is not too much of a surprise
that we will follow their ideas and constructions; and adapt them in a suitable way to our setting.

\begin{definition}\label{def:fbar}
Let $(\A,\T,\mathcal{G})$ be a regularity structure, and let $(\Pi,\Gamma)$ be a model.
Consider $1\leq p <\infty$, $1\leq q \leq \infty$ as well as $\gamma > 0$. 
Furthermore, let $\E_n=Q(0,2^{-n})\cap\Lambda_n\setminus\{0\}$. We denote by $\bar{\mathcal{F}}^\gamma_{p,q}$ 
the (Banach) space of all sequences of maps 
\begin{align*}
\fbar^{(n)}\colon\Lambda_n\to\T_{\gamma}^-,\quad n\geq 0,
\end{align*}
such that, uniformly over all $\zeta\in\A_{\gamma}$,
\begin{itemize}
\item[i)] $\displaystyle\Big(\sum_{y\in\Lambda_0}
\big|\fbar^{(0)}(y)\big|_\zeta^p\Big)^\frac{1}{p} < \infty$,
\item[ii)] $\displaystyle\bigg\|\bigg(\sum_{n\geq 0}\Big|\sum_{y\in\Lambda_n}\sum_{h\in\E_{n}}
\frac{|\fbar^{(n)}(y{+}h)-\Gamma_{y{+}h,y}\fbar^{(n)}(y)|_\zeta}{2^{-n(\gamma-\zeta)}}
\chi^n_y(x)\Big|^q\bigg)^\frac{1}{q}\bigg\|_{L^p}<\infty$,
\item[iii)] $\displaystyle\bigg\|\bigg(\sum_{n\geq 0}
\Big|\sum_{y\in\Lambda_n}\frac{|\fbar^{(n)}(y)-\fbar^{(n+1)}(y)|_\zeta}
{2^{-n(\gamma-\zeta)}}\chi^n_y(x)\Big|^q\bigg)^\frac{1}{q}\bigg\|_{L^p}<\infty$.
\end{itemize}
The associated norm for an element $\fbar\in\Fbar_{p,q}^\gamma$ will be denoted by 
$\vvvert\fbar\vvvert_{\Fbar^\gamma_{p,q}}$.
\end{definition}

We define for $M\in\N$ the sets $\E_n^{M,0}=Q(0,M\cdot 2^{-n})\cap \Lambda_n$ and
$\E_n^M = \E_n^{M,0}\setminus\{0\}$. We then have the following additional bounds
which will be of great use for us in the sequel.

\begin{remark}\label{rem:consistency1}
For every $\fbar\in\Fbar^\gamma_{p,q}$ and every $M\in\N$, we have
\begin{align*}
\bigg\|\bigg(\sum_{n\geq 0}\Big|\sum_{y\in\Lambda_n}\sum_{h\in\E_{n+1}^{M,0}}
\frac{|\fbar^{(n)}(y)-\Gamma_{y,y+h}\fbar^{(n{+}1)}(y{+}h)|_\zeta}
{2^{-n(\gamma-\zeta)}}\chi^n_y(x)\Big|^q\bigg)^\frac{1}{q}\bigg\|_{L^p} < \infty.
\end{align*}
In fact, this quantity is bounded by $\vvvert\fbar\vvvert_{\Fbar^\gamma_{p,q}}$ up to some
proportionality constant. For a proof, one can use induction over $M\in\N$ and reduce everything to
the consistency and translation bound.
\end{remark}

\begin{remark}\label{rem:consistency2}
Let $\fbar\in\Fbar^\gamma_{p,q}$ and $M\in\N$. Then
\begin{align*}
    \bigg\|\bigg(\sum_{n\geq 0}\Big|\sum_{y\in\Lambda_n}\sum_{h\in\E_n^M}
    \frac{|\fbar^{(n)}(y{+}h)-\Gamma_{y+h,y}\fbar^{(n)}(y)|_\zeta}{2^{-n(\gamma-\zeta)}}
    \chi^n_y(x)\Big|^q\bigg)^\frac{1}{q}\bigg\|_{L^p} \lesssim \vvvert\fbar\vvvert_{\Fbar^\gamma_{p,q}}.
\end{align*}
This follows immediately from the translation bound, e.g.\ by induction over $M\in\N$.
\end{remark}

\begin{remark}\label{embeddingsSeqSpaces}
Consider parameters $1 \leq q \leq \bar{q} \leq \infty$, $1 \leq p <\infty$ and $\gamma > 0$. 
Then, the following elementary (continuous) embeddings hold true (cf.\ \cite[Sec.\ 2.3.2]{Triebel:1983}):
\begin{itemize}\itemsep3pt
    \item[i)]   $\Fbar^\gamma_{p,q}\subset \Fbar^\gamma_{p,\bar{q}},$
    \item[ii)]  $\bar{\mathcal{B}}^{\gamma}_{p,q\wedge p}\subset
                \Fbar^\gamma_{p,q}\subset\bar{\mathcal{B}}^{\gamma}_{p,q\vee p}.$
\end{itemize}
For the scale of spaces $\bar{\mathcal{B}}$, we refer to Definition \ref{def:discreteBesov}.
In this context, we also use for $n\geq 0$ and maps $u\colon\Lambda_n\to\R$ the notation
\begin{align*}
    \|u\|_{\ell^p_n} := \bigg(\sum_{y\in\Lambda_n}2^{-n|\s|}|u(y)|^p\bigg)^\frac{1}{p}.
\end{align*}
\textit{Indeed:} The first assertion follows directly from $\ell^q\subset\ell^{\bar{q}}$.
Therefore, we immediately move on to the second assertion. Of course, 
the respective local bounds are immediate. Furthermore, the argumentation
for the respective translation and consistency bounds is almost identical. Thus, we only
discuss the former. For this, fix $n\geq 0$, $h\in\E_n$ and $\zeta\in\A_\gamma$. 
Due to the continuous embedding $\ell^1\subset\ell^p$ we obtain the bound
\begin{align*}
    \bigg\|&\frac{|\fbar^{(n)}(y{+}h)-\Gamma_{y{+}h,y}\fbar^{(n)}(y)|_\zeta}
    {2^{-n(\gamma-\zeta)}}\bigg\|_{\ell^p_n}^{q\vee p} \\
    &\hspace{1cm}\leq\bigg(\int_{\R^d}\Big(\sum_{y\in\Lambda_n}
    \frac{|\fbar^{(n)}(y{+}h)-\Gamma_{y{+}h,y}\fbar^{(n)}(y)|_\zeta}{2^{-n(\gamma-\zeta)}}
    \chi^n_y(x)\Big)^{(q\vee p)\frac{p}{q\vee p}}\,\mathrm{d}x\bigg)^\frac{q\vee p}{p}.
\end{align*}
Since $p/(q\vee p)\leq 1$, recall that the space $L^{p/(q\vee p)}(\R^d)$ admits an inverse
triangle inequality for \textit{non-negative} functions. Combining this fact with the former
bound immediately yields
\begin{align*}
    \bigg(\sum_{n\geq 0}&\sum_{h\in\E_n}\left\|
    \frac{|\fbar^{(n)}(y{+}h)-\Gamma_{y{+}h,y}\fbar^{(n)}(y)|_\zeta}{2^{-n(\gamma-\zeta)}}
    \right\|_{\ell^p_n}^{q\vee p}\bigg)^\frac{1}{q\vee p} \\ &\leq
    \bigg\|\bigg(\sum_{n\geq 0}\sum_{h\in\E_n}\Big|\sum_{y\in\Lambda_n}
    \frac{|\fbar^{(n)}(y{+}h)-\Gamma_{y{+}h,y}\fbar^{(n)}(y)|_\zeta}{2^{-n(\gamma-\zeta)}}
    \chi^n_y(x)\Big|^{q\vee p}\bigg)^\frac{1}{q\vee p}\bigg\|_{L^p}.
\end{align*}
At this point, it suffices to exploit the continuous embedding $\ell^q\subset\ell^{q\vee p}$
and the fact that the set $\E_n$ contains finitely many points, independent of $n\geq 0$.
This proves the embedding $\Fbar^\gamma_{p,q}\subset\bar{\mathcal{B}}^{\gamma}_{p,q\vee p}$.
The remaining one is an immediate consequence of the triangle inequality in $L^{p/(q\wedge p)}(\R^d)$.
\end{remark}

Apart from these elementary embeddings, we also expect to have $\F^\gamma_{p,q}\subset\F^{\gamma'}_{p,q}$ 
whenever $\gamma'<\gamma$. In general, this may not be true if $q<\infty$ since the theory imposes
H\"{o}lder type bounds on models $(\Pi,\Gamma)$. Therefore, we make the following standing assumption.

\begin{assump}\label{gamma}
    For a given set of homogeneities $\A$, we assume that $\gamma\notin\A$.
\end{assump}

\begin{remark}\label{rem:sup_Lp}
A noteworthy consequence of the embedding 
$\Fbar^\gamma_{p,q}\subset\bar{\mathcal{B}}^{\gamma}_{p,q\vee p}$
is that we can immediately propagate the local bound to smaller scales, i.e.\
\begin{align*}
    \sup_{n\geq 0} \big\||\fbar^{(n)}(y)|_\zeta\big\|_{\ell^p_n} \lesssim 
    \vvvert\fbar\vvvert_{\Fbar^\gamma_{p,q}},
\end{align*}
which holds uniformly over all $\fbar\in\Fbar^\gamma_{p,q}$ and all $\zeta\in\A_{\gamma}$, 
cf.\ \cite[Lemma 2.14]{Hairer:2017}. On the other side, adapting the 
arguments developed in the proof of \cite[Lemma 2.14]{Hairer:2017} to our setup
and working directly with the scale $\Fbar^\gamma_{p,q}$ shows that we 
actually have the following improved version
\begin{align}\label{localBound}
    \bigg\|\sup_{n\geq 0}\sum_{y\in\Lambda_n}\sum_{h\in\mathcal{E}_n}
		|\fbar^{(n)}(y{+}h)|_\zeta\chi^n_y(x)\bigg\|_{L^p} 
    \lesssim \vvvert\fbar\vvvert_{\Fbar^\gamma_{p,q}},
\end{align}
which again is valid uniformly over all $\fbar\in\Fbar^\gamma_{p,q}$ and all $\zeta\in\A_{\gamma}$.
\end{remark}

\begin{proof}[Proof of \eqref{localBound}]
First of all, due to Assumption \ref{gamma} and the argument in the proof of \cite[Lemma 2.14]{Hairer:2017}
it suffices to prove the desired bound for the case $\zeta=\max\A_\gamma$. We fix $n\geq 0$,
$y'\in\Lambda_n$ and $x\in Q^n_{y'}$. We then have
\begin{align*}
\sum_{y\in\Lambda_{n+1}}\sum_{h\in\mathcal{E}_{n+1}}|\fbar^{(n+1)}(y{+}h)|_\zeta\chi^{n+1}_y(x)
=\sum_{\substack{y\in\Lambda_{n+1}\\ y\in Q^n_{y'}}}\sum_{h\in\mathcal{E}_{n+1}}
|\fbar^{(n+1)}(y{+}h)|_\zeta\chi^{n+1}_y(x).
\end{align*}
Recall the action of the scaling map
\begin{align*}
\mathcal{S}^\lambda_\s x' := (\lambda^{-\s_1}x'_1,\ldots,\lambda^{-\s_d}x'_d),\quad x'\in\R^d,\,\lambda>0.
\end{align*}
Furthermore, for any given $x'\in\R^d$ and $m\in\N$, we denote by $z^m_{x'}\in\Lambda_m$ the unique lattice point
with the property that $x'\in Q^m_{z^m_{x'}}$. Now, we bound
\begin{align*}
\sum_{\substack{y\in\Lambda_{n+1}\\ y\in Q^n_{y'}}}\sum_{h\in\mathcal{E}_{n+1}}
|\fbar^{(n+1)}(y{+}h)|_\zeta\chi^{n+1}_y(x) &\leq
\sum_{\substack{y\in\Lambda_{n+1}\\ y\in Q^n_{y'}}}\sum_{h\in\mathcal{E}_{n+1}}
\big|\fbar^{(n)}\big(z^n_{y+\mathcal{S}_\s^{2^{-1}}h}\big)\big|_\zeta\chi^{n+1}_y(x) \\
&\hspace{-2.5cm}+\sum_{\substack{y\in\Lambda_{n+1}\\ y\in Q^n_{y'}}}\sum_{h\in\mathcal{E}_{n+1}}
|\fbar^{(n+1)}(y{+}h)-\fbar^{(n)}(z_y^n)|_\zeta\chi^{n+1}_y(x) \\
&\hspace{-2.5cm}+\sum_{\substack{y\in\Lambda_{n+1}\\ y\in Q^n_{y'}}}\sum_{h\in\mathcal{E}_{n+1}}
\big|\fbar^{(n)}(z_y^n)-\fbar^{(n)}\big(z^n_{y+\mathcal{S}_\s^{2^{-1}}h}\big)\big|_\zeta\chi^{n+1}_y(x).
\end{align*}
Observe that for $y\in\Lambda_{n+1}\cap Q^n_{y'}$ and $h\in\mathcal{E}_{n+1}$ 
we have $z^n_y=y'$ as well as $z^n_{y+\mathcal{S}_\s^{2^{-1}}h}=y'+\mathcal{S}_\s^{2^{-1}}h$.
Note also that the map $h\mapsto y'+\mathcal{S}_\s^{2^{-1}}h$ defines a bijection
$\mathcal{E}_{n+1}\to\{y'\}+\mathcal{E}_{n}$. Now since $\zeta=\max\A_\gamma$ we have
$\Q_\zeta\Gamma\bTau = \Q_\zeta\bTau$ for all $\bTau\in\T^-_\gamma$ and all $\Gamma\in\G$.
Hence,
\begin{align*}
\sum_{\substack{y\in\Lambda_{n+1}\\ y\in Q^n_{y'}}}\sum_{h\in\mathcal{E}_{n+1}}
|\fbar^{(n+1)}(y{+}h)|_\zeta\chi^{n+1}_y(x) &\leq \sum_{h\in\mathcal{E}_{n}}
|\fbar^{(n)}(y'{+}h)|_\zeta\chi^{n}_{y'}(x) \\
&\hspace{-2.5cm}+\sum_{h\in\mathcal{E}^{2,0}_{n+1}}
|\Gamma_{y',y'+h}\fbar^{(n+1)}(y'{+}h)-\fbar^{(n)}(y')|_\zeta\chi^{n}_{y'}(x) \\
&\hspace{-2.5cm}+\sum_{h\in\mathcal{E}_{n}}
|\Gamma_{y'+h,y'}\fbar^{(n)}(y')-\fbar^{(n)}(y'{+}h)|_\zeta\chi^{n}_{y'}(x).
\end{align*}
Since this bound holds true uniformly over all $n\geq 0$,
all $y'\in\Lambda_n$ and all $x\in Q^n_{y'}$ we infer with 
the help of Remark \ref{rem:consistency1} that
\begin{align*}
\bigg\|\sup_{n\geq 0}&\sum_{y\in\Lambda_n}\sum_{h\in\mathcal{E}_n}
|\fbar^{(n)}(y{+}h)|_\zeta\chi^n_y(x)\bigg\|_{L^p} \\ &\leq
\bigg\|\sum_{y\in\Lambda_0}\sum_{h\in\mathcal{E}_0}
|\fbar^{(0)}(y{+}h)|_\zeta\chi^0_y(x)\bigg\|_{L^p} +
C\|(2^{-n(\gamma-\zeta)})_n\|_{\ell^{\frac{q}{q-1}}}\vvvert\fbar\vvvert_{\Fbar^\gamma_{p,q}},
\end{align*}
for some absolute constant $C>0$. This concludes the proof.
\end{proof}

The remainder of this section is devoted to the discussion of how these spaces,
i.e.\ $\F^\gamma_{p,q}$ and $\Fbar^\gamma_{p,q}$, are related to each other.
We begin with the task of how to obtain from a given modelled distribution $\bf\in\F^\gamma_{p,q}$
an associated element $\fbar$ in the space $\Fbar^\gamma_{p,q}$. This objective is
settled in the first part of the following statement. On the other hand, if we want
to go the other way around later on, we would like to ensure that both of the resulting maps
between $\F^\gamma_{p,q}$ and $\Fbar^\gamma_{p,q}$ are consistent. We take care
of this aim in the second part of the following statement.

\begin{proposition}\label{prop:forth} 
Let $1\leq p <\infty$, $1\leq q\leq \infty$ and $\gamma > 0$.

(i) Given a modelled distribution $\bf\in\F^\gamma_{p,q}$, we define for all $n\geq 0$
and $y\in\Lambda_n$
\begin{align}\label{average}
\fbar^{(n)}(y) = \mint{-}_{Q(y,2^{-n})}\Gamma_{y,z}\bf(z)\,\mathrm{d}z.  
\end{align}
With this definition, we indeed obtain $\fbar\in\Fbar^\gamma_{p,q}$ and
$\vvvert\fbar\vvvert_{\Fbar^\gamma_{p,q}}\lesssim\vvvert\bf\vvvert_{\F^\gamma_{p,q}}$.

(ii) Let again $\bf\in\F^\gamma_{p,q}$ be a modelled distribution and define, 
for $x\in\R^d$ and $n\in\N_0$, the map
\begin{align}
\bf_n(x) = \Gamma_{x,x_n}\fbar^{(n)}(x_n).
\end{align}
Here, $\fbar^{(n)}$ is defined as in the first part of the statement.
Then, we have convergence $\Q_\zeta\bf_n\to\Q_\zeta\bf$ in $L^p(\R^d)$ for all $\zeta\in\A_\gamma$.
\end{proposition}

\begin{proof}
(i) The local bound holds immediately. Thus, let us start with the translation bound. For, we
fix $x\in\R^d$, $n\geq 0$ and $y=x_n\in\Lambda_n$. We then have the estimate
\begin{align*}
    \frac{|\fbar^{(n)}(y{+}h)-\Gamma_{y+h,y}\fbar^{(n)}(y)|_\zeta}{2^{-n(\gamma-\zeta)}} 
    &\lesssim \sum_{\beta\geq\zeta}\mint{-}_{Q(y,2^{-n})}
    \frac{|\bf(z{+}h)-\Gamma_{z+h,z}\bf(z)|_\beta}{2^{-n(\gamma-\beta)}}\,\mathrm{d}z \\
    &\hspace{-3cm}\lesssim \sum_{\beta\geq\zeta}\mint{-}_{Q(x,2\cdot 2^{-n})}
    \frac{|\bf(z{+}h)-\Gamma_{z+h,z}\bf(z)|_\beta}{2^{-n(\gamma-\beta)}}\,\mathrm{d}z \\
    &\hspace{-3cm}\lesssim \sum_{\beta\geq\zeta}\mint{-}_{Q(0,2\cdot 2^{-n})}
    \frac{|\bf(x{+}z{+}h)-\Gamma_{x{+}z{+}h,x}\bf(x)|_\beta}{2^{-n(\gamma-\beta)}}\,\mathrm{d}z \\
    &\hspace{-1.5cm}+\sum_{\beta\geq\zeta}\mint{-}_{Q(0,2\cdot 2^{-n})}
    \frac{|\Gamma_{x+z+h,x+z}(\bf(x{+}z)-\Gamma_{x+z,x}\bf(x))|_\beta}{2^{-n(\gamma-\beta)}}\,\mathrm{d}z \\
    &\hspace{-3cm}\lesssim\sum_{\delta\geq\beta}\sum_{\beta\geq\zeta}\mint{-}_{Q(0,3\cdot 2^{-n})}
    \frac{|\bf(x{+}h)-\Gamma_{x+h,x}\bf(x)|_\delta}{2^{-n(\gamma-\delta)}}\,\mathrm{d}h.
\end{align*}
In particular, this immediately yields the required bound.
When it comes to the consistency bound, it is straightforward to verify that
\begin{align*}
    |\fbar^{(n)}(y)&-\fbar^{(n+1)}(y)|_\zeta \lesssim
    \sum_{\beta\geq\zeta}2^{-n(\beta-\zeta)}
    \mint{-}_{Q(0,2^{-n})}|\Delta_{\Gamma,h}\bf(x)|_\beta\,\mathrm{d}h,
\end{align*}
from which the desired bound again follows at once. This concludes the proof of the first assertion.

(ii) This is an immediate consequence of the fact that for a.e.\ $x\in\R^d$
\begin{align*}
    |\bf(x)-\bf_n(x)|_\zeta \lesssim \sum_{\beta\geq\zeta}2^{-n(\gamma-\zeta)}
    \mint{-}_{Q(0,2\cdot2^{-n})}\frac{|\Delta_{\Gamma,h}\bf(x)|_\beta}{2^{-n(\gamma-\beta)}}\,\mathrm{d}h.
\end{align*} 
Hence, we can conclude the proof.
\end{proof}

In a second step, we prove how to get back a modelled distribution from
an element of the discrete counterpart $\Fbar^\gamma_{p,q}$.

\begin{proposition}\label{prop:back}
We consider $1\leq p <\infty$, $1\leq q\leq \infty$ and $\gamma > 0$.
Let $\fbar\in\Fbar^\gamma_{p,q}$ and define, for $x\in\R^d$ and $n\in\N_0$,
\begin{align*}
\bf_n(x) = \Gamma_{x,x_n}\fbar^{(n)}(x_n).
\end{align*}
Then, for all $\zeta\in\A_\gamma$, the sequence
$(\Q_\zeta\bf_n)_{n\in\N}$ is Cauchy in $L^p(\R^d)$. Furthermore, the limit 
$\bf\colon\R^d\to\T_\gamma^-$ is a modelled distribution in $\F^\gamma_{p,q}$
with $\vvvert\bf\vvvert_{\F^\gamma_{p,q}}\lesssim\vvvert\fbar\vvvert_{\Fbar^\gamma_{p,q}}$.
\end{proposition}

\begin{proof}
We first prove the convergence assertion. To this end, fix $x\in\R^d$ and $n\geq 0$.
We then bound $|\bf_n(x)-\bf_{n+1}(x)|_\zeta$ by
\begin{equation}
\begin{aligned}\label{LpConv}
    \big|\Gamma_{x,x_n}\big(\fbar^{(n)}&(x_n)
    -\Gamma_{x_n,x_{n+1}}\fbar^{(n+1)}(x_{n+1})\big)\big|_\zeta \\
    &\lesssim \sum_{\beta\geq\zeta}2^{-n(\beta-\zeta)}
    |\fbar^{(n)}(x_n)-\Gamma_{x_n,x_{n+1}}\fbar^{(n+1)}(x_{n+1})|_\beta.
\end{aligned}
\end{equation}
Next, we make use of H\"older's inequality and 
Remark \ref{rem:consistency1} to deduce that
\begin{align*}
    \bigg\|\sum_{n\geq n_0}\frac{|\fbar^{(n)}(x_n)-\Gamma_{x_n,x_{n+1}}\fbar^{(n+1)}(x_{n+1})|_\beta}
    {2^{n(\beta-\zeta)}}\bigg\|_{L^p} \lesssim 2^{-n_0(\gamma-\zeta)}
    \vvvert\fbar\vvvert_{\Fbar^\gamma_{p,q}}.
\end{align*}
Thus, for all $\zeta\in\A_\gamma$, the sequence
$(\Q_\zeta\bf_n)_{n\in\N}$ is indeed Cauchy in $L^p(\R^d)$. Let us denote
the limit by $\bf\colon\R^d\to\T^-_\gamma$. It remains to show that
$\bf\in\F^\gamma_{p,q}$, and that the asserted bound does hold.

We first note that the local bound is already part of the argument above.
Hence, let us concentrate on the translation bound. For this, we start with
\begin{align*}
    \bigg\|\Big\|\mint{-}_{Q(0,4\lambda)}\frac{|\Delta_{\Gamma,h}\bf(x)|_\zeta}
    {\lambda^{\gamma-\zeta}}\,\mathrm{d}h\Big\|_{L^q_\lambda}\bigg\|_{L^p} \lesssim
    \bigg\|\bigg(\sum_{n\geq 0}\Big|\mint{-}_{Q(0,4\cdot 2^{-n})}
    \frac{|\Delta_{\Gamma,h}\bf(x)|_\zeta}{2^{-n(\gamma-\zeta)}}\,\mathrm{d}h\Big|^q\bigg)^\frac{1}{q}\bigg\|_{L^p}.
\end{align*}
For the time being, we fix $x\in\R^d$, $n\geq 0$ as well as $h\in Q(0,4\cdot 2^{-n})$.
Following Hairer and Labb\'{e} \cite{Hairer:2017}, 
we then employ the following decomposition of $\Delta_{\Gamma,h}\bf(x)$:
\begin{align}\label{decomp}
    \big(\bf(x{+}h)-\bf_n(x{+}h)\big) + \Delta_{\Gamma,h}\bf_n(x)
    +\Gamma_{x+h,x}\big(\bf_n(x)-\bf(x)\big).
\end{align}
We treat each term separately in the following. Let us begin with the second term in \eqref{decomp}.
We find $h'\in\E_n^{M,0}$, $M=4$, such that $(x{+}h)_n=x_n+h'$. Observe next that
\begin{align*}
    \Delta_{\Gamma,h}\bf_n(x) = \Gamma_{x+h,x_n+h'}
    \big(\fbar^{(n)}(x_n{+}h')-\Gamma_{x_n+h',x_n}\fbar^{(n)}(x_n)\big).
\end{align*}
If $h'=0$ the term on the right hand side of this identity vanishes, i.e.\ we obtain the bound
\begin{align*}
    \Big|\mint{-}_{Q(0,4\cdot 2^{-n})}&
    \frac{|\Delta_{\Gamma,h}\bf_n(x)|_\zeta}{2^{-n(\gamma-\zeta)}}\,\mathrm{d}h\Big|
    \\ &\lesssim \sum_{\beta\geq\zeta}\sum_{y\in\Lambda_n}\sum_{h\in\E_n^M}
    \frac{|\fbar^{(n)}(y{+}h')-\Gamma_{y+h',y}\fbar^{(n)}(y)|_\beta}
    {2^{-n(\gamma-\beta)}}\chi^n_y(x).
\end{align*}
Therefore, it remains to make use of Remark \ref{rem:consistency2} which yields a bound
of requested order. We turn to the third term in \eqref{decomp}. In this case, we can estimate
\begin{align*}
    \mint{-}_{Q(0,4\cdot 2^{-n})}&
    \frac{|\Gamma_{x+h,x}(\bf_n(x)-\bf(x))|_\zeta}{2^{-n(\gamma-\zeta)}}\,\mathrm{d}h
    \lesssim\sum_{\beta\geq\zeta}\sum_{m\geq n}\frac{|\bf_m(x)-\bf_{m+1}(x)|_\beta}{2^{-n(\gamma-\beta)}} \\
    &\hspace{-1cm}\lesssim\sum_{\delta\geq\beta}\sum_{\beta\geq\zeta}\sum_{m\geq n}
    \sum_{y\in\Lambda_m}\sum_{h\in\E_{m+1}^0}\frac{|\fbar^{(m)}(y)-\Gamma_{y,y+h}\fbar^{(m+1)}(y{+}h)|_\delta}
    {2^{-(n-m)(\gamma-\beta)}2^{-m(\gamma-\delta)}}\chi^m_y(x).
\end{align*}
Exploiting Young's inequality for convolutions therefore shows that
\begin{align*}
    \bigg\|\bigg(&\sum_{n\geq 0}\Big|\mint{-}_{Q(0,4\cdot 2^{-n})}
    \frac{|\Gamma_{x+h,x}(\bf_n(x)-\bf(x))|_\zeta}
    {2^{-n(\gamma-\zeta)}}\,\mathrm{d}h\Big|^q\bigg)^\frac{1}{q}\bigg\|_{L^p} \\ &\lesssim
    \sum_{\delta\geq\zeta}\bigg\|\bigg(\sum_{n\geq 0}\Big|\sum_{y\in\Lambda_n}\sum_{h\in\E_{n+1}^0}
    \frac{|\fbar^{(n)}(y)-\Gamma_{y,y+h}\fbar^{(n+1)}(y{+}h)|_\delta}
    {2^{-n(\gamma-\delta)}}\chi^n_y(x)\Big|^q\bigg)^\frac{1}{q}\bigg\|_{L^p}.
\end{align*}
Thus, Remark \ref{rem:consistency1} finally ensures that a bound of required type holds.

Let us finally discuss the remaining term from \eqref{decomp}. To this end, choose $0<\eta<1$
such that $\gamma':= \gamma-\zeta+|\s|(\eta-1)/\eta>0$. We then proceed as follows:
\begin{align*}
    \mint{-}_{Q(0,4\cdot 2^{-n})}&
    \frac{|\bf_n(x{+}h)-\bf(x{+}h)|_\zeta}{2^{-n(\gamma-\zeta)}}\,\mathrm{d}h \\
    &\lesssim\sum_{m\geq n}\mint{-}_{Q(0,4\cdot 2^{-n})}
    \frac{|\bf_m(x{+}h)-\bf_{m+1}(x{+}h)|_\beta}{2^{-n(\gamma-\zeta)}}\,\mathrm{d}h \\
    &\lesssim\sum_{\beta\geq\zeta}\sum_{m\geq n}
    \sum_{\substack{y\in\Lambda_m \\ \|y{-}x\|_\s\leq C2^{-n}}}\sum_{h\in\E^0_{m+1}}
    \frac{|\fbar^{(m)}(y)-\Gamma_{y,y+h}\fbar^{(m+1)}(y{+}h)|_\beta}
    {2^{-(n-m)|\s|}2^{-(n-m)(\gamma-\zeta)}2^{-m(\gamma-\beta)}}.
\end{align*}
Having arrived at this point, note that each term in the sum over $m\geq n$ can itself be bounded by
\begin{align*}
    2^{(n-m)\gamma'}\Big\{\mathcal{M}\Big(\Big|\sum_{y\in\Lambda_m}\sum_{h\in\E^0_{m+1}}
    \frac{|\fbar^{(m)}(y)-\Gamma_{y,y+h}\fbar^{(m+1)}(y{+}h)|_\beta}
    {2^{-m(\gamma-\beta)}}\chi^m_y\Big|^\eta\Big)(x)\Big\}^\frac{1}{\eta}.
\end{align*}
Now, apply first Young's inequality for convolutions and then
the vector-valued maximal inequality. Having done this, it only remains 
to resort another time to the bound of Remark \ref{rem:consistency1} in order
to conclude the argument.
\end{proof}

\begin{remark}
As a consequence of the bounds contained in Proposition \ref{prop:forth} and Proposition \ref{prop:back},
the elementary embeddings presented in Remark \ref{embeddingsSeqSpaces} carry over immediately
to the spaces $\F^\gamma_{p,q}$ and $\mathcal{B}^\gamma_{p,q}$.
More elaborate embeddings follow in the next section.
\end{remark}

\section{Embedding theorems} \label{sec:embedding}
It is the aim of this section to provide embeddings for (and between) the two scales $\F^{\gamma}_{p,q}$
and $\mathcal{B}^\gamma_{p,q}$. Embeddings with respect to the scale of Besov spaces were already
proven by Hairer and Labb\'{e} in \cite[Sec.\ 4]{Hairer:2017}. Therefore, we focus on statements which always include
at least once the scale of Triebel--Lizorkin modelled distributions. Recall that 
in the course of the last section we have already proven the following embeddings
\begin{align*}
    \F^\gamma_{p,q}\subset \F^\gamma_{p,\bar{q}},\quad
    \mathcal{B}^{\gamma}_{p,q\wedge p}\subset\F^\gamma_{p,q}\subset\mathcal{B}^{\gamma}_{p,q\vee p}.
\end{align*}
Here, $1 \leq q \leq \bar{q} \leq \infty$, $1 \leq p <\infty$ and $\gamma > 0$.
The main result of this section provides further embeddings and reads as follows.
We also want to remind the reader at this point of Assumption \ref{gamma}.

\begin{proposition}\label{embedding}
    Consider $0<\gamma\leq\bar{\gamma}$, $1\leq q,\bar{q},r\leq \infty$ as well as $1\leq p,\bar{p}<\infty$.
    Then, it holds
    \begin{itemize}\itemsep3pt
    \item[i)]   $\F^{\bar{\gamma}}_{p,\bar{q}}\subset \F^\gamma_{p,q},\quad$ if \quad $q<\bar{q}$ 
                \quad and \quad $\gamma<\bar{\gamma}$,
    \item[ii)]  $\F^{\bar{\gamma}}_{\bar{p},q}\subset \F^\gamma_{p,r},\quad$ if \quad $\bar{p}<p$ \quad and 
                \quad $\gamma<\bar{\gamma}-|\s|(1/\bar{p}-1/p)$,
    \item[iii)] $\F^{\bar{\gamma}}_{\bar{p},q}\subset \mathcal{B}^\gamma_{p,\bar{p}},\quad$ if \quad $\bar{p}<p$ 
                \quad and \quad $\gamma<\bar{\gamma}-|\s|(1/\bar{p}-1/p)$.
\end{itemize}
\end{proposition}

\begin{remark}
The first embedding is the pendant of the corresponding embedding for the Besov scale of modelled
distributions from \cite{Hairer:2017}. The second and final one represent analogues of well-known 
embeddings for the classical Triebel--Lizorkin scale due to Jawerth \cite{Jawerth:1977}.
\end{remark}

Before we dive into the proof of these embeddings, we provide some helpful results. 
In order to obtain the embedding $\F^{\bar{\gamma}}_{\bar{p},q}\subset \F^\gamma_{p,r}$,
we will rely several times on the following bound. 

\begin{lemma}\label{lemmaEmb}
    Consider $1\leq p_1<p<\infty$, $1\leq q,r\leq \infty$ and $0<\delta\leq\delta_1$ 
    such that $\delta\leq\delta_1-|\s|(1/p_1{-}1/p)$. Furthermore, let 
    $u_n\colon\Lambda_n\to\R$, $n\geq 0$, be a sequence of functions. 
    Then, the following bound holds
    \begin{align*}
        \bigg\|\Big\|\sum_{y\in\Lambda_n}\frac{u_n(y)}{2^{-n\delta}}\chi^n_y(x)\Big\|_{\ell^r}\bigg\|_{L^p}
        \lesssim \bigg\|\Big\|\sum_{y\in\Lambda_n}\frac{u_n(y)}{2^{-n\delta_1}}
        \chi^n_y(x)\Big\|_{\ell^q}\bigg\|_{L^{p_1}}.
    \end{align*}
\end{lemma}

\begin{proof}
    We use an interpolation trick from Jawerth, cf.\ also the book of Triebel 
    \cite[Sec. 2.7.1]{Triebel:1983}. For, we first remark that we may assume 
    without loss of generality that the right hand side equals one. Furthermore,
    let us also recall some notation. For a measurable set $M\subset\R^d$ we denote by 
    $|M|$ its Lebesgue measure. If $h$ is a measurable function, its associated
    distribution function is given by $\lambda_h(t)=|\{x\in\R^d\colon h(x)>t\}|$, where $t>0$.
    
    We observe first that
    \begin{align*}
        \sup_{y\in\Lambda_n}\frac{|u_n(y)|}{2^{-n\delta}} \leq 2^{n\frac{|\s|}{p}}
        \Big\|\frac{u_n(y)}{2^{-n\delta}}\Big\|_{\ell^{p}_n} \leq 2^{n\frac{|\s|}{p}},
    \end{align*}
    uniformly over $n\geq 0$. In particular, there exists $K>0$ such that 
    \begin{align}\label{lower}
        \sum_{n=0}^N\Big|\sum_{y\in\Lambda_n}\frac{u_n(y)}{2^{-n\delta}}\chi^n_y(x)\Big|^r
        \leq K2^{N\frac{|\s|}{p}r}, 
    \end{align}
    as well as
    \begin{align}\label{upper}
        \sum_{n\geq N+1}\Big|\sum_{y\in\Lambda_n}\frac{u_n(y)}{2^{-n\delta}}
        \chi^n_y(x)\Big|^r\leq K2^{N(\delta-\delta_1)r}\sup_{n\geq 0}
        \Big|\sum_{y\in\Lambda_n}\frac{u_n(y)}{2^{-n\delta_1}}\chi^n_y(x)\Big|^r,
    \end{align}
    uniformly over all $N\geq 0$ and all $x\in\R^d$. Next, we bound
    \begin{align*}
        \bigg\|\Big\|\sum_{y\in\Lambda_n}\frac{u_n(y)}{2^{-n\delta}}\chi^n_y(x)\Big\|_{\ell^r}\bigg\|_{L^p}
        \lesssim \sum_{\substack{k\in\Z \\ 2^k\leq (2K)^\frac{1}{r}}}& 2^{kp}
        \lambda_{\big\|\sum\limits_{y\in\Lambda_n}\frac{u_n(y)}{2^{-n\delta}}\chi^n_y\big\|_{\ell^r}}(2^k) \\
        &+\sum_{\substack{k\in\Z \\ 2^k> (2K)^\frac{1}{r}}} 2^{kp}
        \lambda_{\big\|\sum\limits_{y\in\Lambda_n}\frac{u_n(y)}{2^{-n\delta}}\chi^n_y\big\|_{\ell^r}}(2^k).
    \end{align*}
    In the following, we will bound each of the two sums separately. Beginning with the first term,
    we note that due to \eqref{upper} with $N=-1$ we have
    \begin{align*}
        \sum_{\substack{k\in\Z \\ 2^k\leq (2K)^\frac{1}{r}}}
        2^{kp}\lambda_{\big\|\sum\limits_{y\in\Lambda_n}\frac{u_n(y)}{2^{-n\delta}}\chi^n_y\big\|_{\ell^r}}(2^k)
        &\lesssim \sum_{\substack{k\in\Z \\ 2^k\leq (2K)^\frac{1}{r}}}2^{kp_1}
        \lambda_{\sup\limits_{n\geq 0}\big|\sum\limits_{y\in\Lambda_n}\frac{u_n(y)}{2^{-n\delta_1}}\chi^n_y\big|^r}(2^{kr}) \\
        &\lesssim \sum_{k\in\Z} 2^{kp_1}\lambda_{\big\|\sum\limits_{y\in\Lambda_n}\frac{u_n(y)}
        {2^{-n\delta_1}}\chi^n_y\big\|_{\ell^q}}(2^k) \lesssim 1,
    \end{align*}
    which is a bound of required order. With regard to the second term, we first choose for each
    $k\in\Z$ with the property $2^k>(2K)^{1/r}$ the uniquely determined integer $N=N(k)\in\N$
    which is maximal with respect to $K2^{Nr|\s|/p}\leq 2^{-1}2^{kr}$. Observe that then
    $2^{kr}2^{N|\s|(\delta-\delta')r} \sim 2^{kr\frac{p}{p_1}}$ uniformly over all $k\in\Z$
    with the property $2^k>(2K)^{1/r}$. In particular, we obtain due to \eqref{lower} and \eqref{upper}
    the bound
    \begin{align*}
        \sum_{\substack{k\in\Z \\ 2^k> (2K)^\frac{1}{r}}} 2^{kp}
        \lambda_{\big\|\sum\limits_{y\in\Lambda_n}\frac{u_n(y)}{2^{-n\delta}}\chi^n_y\big\|_{\ell^r}}(2^k)
        &\lesssim \sum_{\substack{k\in\Z}} \big(2^{\frac{p}{p_1}}\big)^{kp_1}
        \lambda_{\sup\limits_{n\geq 0}\big|\sum\limits_{y\in\Lambda_n}\frac{u_n(y)}{2^{-n\delta_1}}
        \chi^n_y\big|}(2^{\frac{p}{p_1}k}),
    \end{align*}
    which is again of required order. This concludes the proof.
\end{proof}

Secondly, with respect to the embedding $\F^{\bar{\gamma}}_{\bar{p},q}\subset \mathcal{B}^\gamma_{p,\bar{p}}$
it turns out to be useful to make use of the real interpolation method based on Peetre's $K$-functional.
Therefore, let us introduce some notation. For two real Banach spaces $X$ and $Y$, which are both
continuously embedded into another real Banach space $Z$ (in our case of interest, $Z=L^p$),
we define the functional
\begin{align*}
    K(t,z) := \inf_{\substack{z=x+y \\ x\in X,\, y\in Y}}\big(\|x\|_X + t\|y\|_X\big).
\end{align*}
For $0<\theta<1$ and $1\leq q\leq \infty$, we then denote by $(X,Y)_{\theta,q}$ the subset of
all elements $z\in X+Y$ such that
\begin{align*}
    \|z\|_{(X,Y)_{\theta,q}} := \bigg(\int_{0}^\infty\frac{\mathrm{d}t}{t}\,
    \Big|\frac{K(t,z)}{t^\theta}\Big|^q\bigg)^\frac{1}{q} < \infty.
\end{align*}
It is a well known fact that the space $(X,Y)_{\theta,q}$ with the norm $\|\cdot\|_{(X,Y)_{\theta,q}}$
itself constitutes a real Banach space. Now, we eventually have all ingredients in place in order
to prove the asserted embedding relations.

\begin{proof}[Proof (of Proposition \ref{embedding})]
    Due to Proposition \ref{prop:forth} and Proposition \ref{prop:back} it suffices to establish
    the asserted embeddings on the level of the sequence spaces $\Fbar$ and $\bar{\mathcal{B}}$.
    
    \textit{i)} Here, a straightforward modification of the proof of the second case of 
    \cite[Thm 4.1]{Hairer:2017} also works in the case of the
    Triebel--Lizorkin scale of modelled distributions. We leave the details to the
    interested reader.
    
    \textit{ii)} Put $\bar{\zeta}=\max\A_{\bar{\gamma}}$ and let us first consider the case 
    $\gamma\in (\bar{\zeta},\bar{\gamma})$. For $\zeta\in\A_\gamma$, we define the functions
    \begin{align*}
        u_n^\zeta(y) = \sum_{h\in\E_n}|\fbar^{(n)}(y{+}h)-\Gamma_{y+h,y}\fbar^{(n)}(y)|_\zeta,
        \quad n\geq 0,\quad y\in\Lambda_n.
    \end{align*}
    As the local bound is an immediate consequence of $\ell^{\bar{p}}\subset\ell^p$,
    the asserted embedding relation then follows directly from Lemma \ref{lemmaEmb}.
    For what follows, it is crucial that this result holds true even if 
    $\gamma=\bar{\gamma}-|\s|(1/\bar{p}-1/p)$.
    
    Now, let us turn to the case $\gamma\in(\ubar{\zeta},\bar{\zeta})$, where
    $\ubar{\zeta}=\max\,(\A_{\bar{\gamma}}\setminus\{\bar{\zeta}\})$.
    Before we give the proof of the required bounds, let us remark that after having
    established this particular case, one obtains the general statement by a recursion over 
    the (finite) set of homogeneities $\A_{\bar{\gamma}}$. 
    
    Now, in order to verify the asserted embedding in the case $\gamma\in(\ubar{\zeta},\bar{\zeta})$,
    we follow the argument in \cite{Hairer:2017}. We consider
    the case where $\bar{\gamma}-|\s|(1/\bar{p}-1/p)\leq\bar{\zeta}$. Then,
    for every $\varepsilon\geq 0$ such that $\bar{\zeta}+\varepsilon < \bar{\gamma}$
    we denote by $p_{\bar{\zeta}^\varepsilon}\in[\bar{p},p]$
    the uniquely determined number with the property
    \begin{align*}
        \bar{\zeta}+\varepsilon = \bar{\gamma}-|\s|(1/\bar{p}-1/p_{\bar{\zeta}^\varepsilon}).
    \end{align*}
    In a first step, we remark that it suffices to show that 
    \begin{align}\label{1st}
    \Q_{<\gamma}\fbar\in\Fbar^{\bar{\zeta}}_{p_{\bar{\zeta}},r}, \quad
    \vvvert\Q_{<\gamma}\fbar\vvvert_{\Fbar^{\bar{\zeta}}_{p_{\bar{\zeta}},r}}
    \lesssim \vvvert\fbar\vvvert_{\Fbar^{\bar{\gamma}}_{\bar{p},q}}.
    \end{align}
    Indeed, we observe that $\A_\gamma=\A_{\bar{\zeta}}$ and $\gamma<\bar{\zeta}$.
    Now, if $p=p_{\bar{\zeta}}$ then $\Fbar^{\bar{\zeta}}_{p,r}\subset\Fbar^\gamma_{p,r}$
    is nothing else than an already established embedding. On the other side, if we have $p_{\bar{\zeta}}<p$
    then because of $\gamma<\bar{\gamma}-|\s|(1/\bar{p}-1/p) = \bar{\zeta}-|\s|(1/p_{\bar{\zeta}}-1/p)$
    we can employ the bound from the first case and obtain 
    $\Fbar^{\bar{\zeta}}_{p_{\bar{\zeta}},r}\subset\Fbar^\gamma_{p,r}$.
    In order to obtain \eqref{1st}, it in turn suffices to show that
    \begin{align}\label{2nd}
        \vvvert\Q_{<\gamma}\fbar\vvvert_{\Fbar^{\gamma'}_{p_{\bar{\zeta}},r}}
        \lesssim\vvvert\fbar\vvvert_{\Fbar^{\bar{\gamma}}_{\bar{p},q}}
    \end{align}
    uniformly over all $\gamma'\in(\ubar{\zeta},\bar{\zeta})$. Indeed, \eqref{1st}
    follows immediately from \eqref{2nd} by an application of Fatou's Lemma. On the 
    other side, for a proof of \eqref{2nd} it suffices to bound
    \begin{align}\label{3rd}
         \vvvert\Q_{<\gamma}\fbar\vvvert_{\Fbar^{\gamma'}_{p_{\bar{\zeta}^\varepsilon},r}}
        \lesssim\vvvert\fbar\vvvert_{\Fbar^{\bar{\gamma}}_{\bar{p},q}},
    \end{align}
    uniformly over all $\varepsilon>0$ such that $\bar{\zeta}+\varepsilon<\bar{\gamma}$
    and all $\gamma'\in (\ubar{\zeta},\bar{\zeta})$. Indeed, let 
    $\gamma'_\varepsilon=\gamma'-|\s|(1/p_{\bar{\zeta}^\varepsilon}-1/p_{\bar{\zeta}})$.
    Then, for sufficiently small $\varepsilon>0$ we have $\A_{\gamma'_\varepsilon}=\A_{\gamma'}$
    and by an application of Fatou's Lemma as well as the first case we have
    \begin{align*}
        \vvvert\Q_{<\gamma}\fbar\vvvert_{\Fbar^{\gamma'}_{p_{\bar{\zeta}},r}}
        \leq\varliminf_{\varepsilon \to 0}
        \vvvert\Q_{<\gamma}\fbar\vvvert_{\Fbar^{\gamma'_\varepsilon}_{p_{\bar{\zeta}},r}}
        \lesssim\varliminf_{\varepsilon \to 0}
        \vvvert\Q_{<\gamma}\fbar\vvvert_{\Fbar^{\gamma'}_{p_{\bar{\zeta}^\varepsilon},r}},
    \end{align*}
    with a proportionality constant of required order, i.e.\ the bound in \eqref{2nd} follows. 
    Thus, we are finally left with a proof of \eqref{3rd}. To this end, we first make use of the fact that
    \begin{equation}
    \begin{aligned}\label{decomp2}
        |\Q_{<\gamma}&(\fbar^{(n)}(y{+}h)-\Gamma_{y+h,h}\fbar^{(n)}(y))|_\zeta
        \\ &\hspace{1cm}\leq|\fbar^{(n)}(y{+}h)-\Gamma_{y+h,h}\fbar^{(n)}(y)|_\zeta +
        |\Gamma_{y+h,y}\Q_{\bar{\zeta}}\fbar^{(n)}(y)|_\zeta,
    \end{aligned}
    \end{equation}
    which holds for all $\zeta\in\A_{\gamma'}$. Of course, for the first term on the 
    right hand side of this bound we again simply apply Lemma \ref{lemmaEmb}. With respect to
    the second quantity we have
    \begin{align*}
        \bigg\|\bigg(\sum_{n\geq 0}\sum_{h\in\E_n}&\Big|
        \sum_{y\in\Lambda_n}\frac{|\Gamma_{y+h,y}\Q_{\bar{\zeta}}\fbar^{(n)}(y)|_\zeta}
        {2^{-n(\gamma-\zeta)}}\chi^n_y(x)\Big|^r\bigg)^\frac{1}{r}\bigg\|_{L^{p_{\bar{\zeta}^\varepsilon}}} \\
        &\lesssim\Big\|\sup_{n\geq 0}\sum_{y\in\Lambda_n}
        |\Q_{\bar{\zeta}}\fbar^{(n)}(y)|_{\bar{\zeta}}\chi^n_y(x)\Big\|_{L^{p_{\bar{\zeta}^\varepsilon}}}
        \big\|\big(2^{n(\gamma-\bar{\zeta})}\big)_n\big\|_{\ell^r} \\
        &\lesssim \Big\|\sup_{n\geq 0}\sum_{y\in\Lambda_n}
        |\Q_{\bar{\zeta}}\fbar^{(n)}(y)|_{\bar{\zeta}}\chi^n_y(x)\Big\|_{L^{p_{\bar{\zeta}^\varepsilon}}}.
    \end{align*}
    Now, as in the proof of Remark \ref{rem:sup_Lp} one establishes the existence
    of a constant $K>0$ such that
    \begin{align*}
        \Big\|\sup_{0\leq n\leq N+1}\sum_{y\in\Lambda_n}
        |\Q_{\bar{\zeta}}\fbar^{(n)}(y)|_{\bar{\zeta}}\chi^n_y(x)\Big\|_{L^{p_{\bar{\zeta}^\varepsilon}}}
        &\leq \Big\|\sup_{0\leq n\leq N}\sum_{y\in\Lambda_n}
        |\Q_{\bar{\zeta}}\fbar^{(n)}(y)|_{\bar{\zeta}}\chi^n_y(x)\Big\|_{L^{p_{\bar{\zeta}^\varepsilon}}} \\
        &\hspace{-5.5cm}+K\bigg\|\bigg(\sum_{n\geq 0}\sum_{h\in\E_{n+1}^0}\Big|
        \sum_{y\in\Lambda_n}\frac{|\fbar^{(n)}(y{+}h)-\Gamma_{y+h,y}\fbar^{(n)}(y)|_\zeta}
        {2^{-n\varepsilon}}\chi^n_y(x)\Big|^r\bigg)^\frac{1}{r}\bigg\|_{L^{p_{\bar{\zeta}^\varepsilon}}},
    \end{align*}
    uniformly over all $N\geq 0$. Hence, by induction over $N\geq 0$ and with the help of
    Remark \ref{rem:consistency1} as well as Lemma \ref{lemmaEmb} we obtain
    \begin{align*}
        \Big\|\sup_{n\geq 0}\sum_{y\in\Lambda_n}
        |\Q_{\bar{\zeta}}\fbar^{(n)}(y)|_{\bar{\zeta}}\chi^n_y(x)\Big\|_{L^{p_{\bar{\zeta}^\varepsilon}}}
        \lesssim \vvvert\fbar\vvvert_{\Fbar^{\bar{\gamma}}_{\bar{p},q}},
    \end{align*}
    uniformly over the required data.
    
    It finally remains to consider the case $\bar{\gamma}-|\s|(1/\bar{p}-1/p)>\bar{\zeta}$,
    where still $\gamma\in(\ubar{\zeta},\bar{\zeta})$. Here, we find $\varepsilon>0$
    such that $\bar{\zeta}+\varepsilon<\bar{\gamma}-|\s|(1/\bar{p}-1/p)$.
    Then, again by making use of the bound \eqref{decomp2} and slightly modifying
    the following argumentation shows that
    \begin{align*}
        \vvvert\Q_{<\gamma}\fbar\vvvert_{\Fbar^\gamma_{p,r}} \lesssim
        \vvvert\fbar\vvvert_{\Fbar^{\bar{\gamma}}_{\bar{p},q}}.
    \end{align*}
    This concludes the discussion of the second embedding.
    
    \textit{iii)} First, choose $\varepsilon > 0$ such that $\gamma\pm\varepsilon$
    satisfy Assumption \ref{gamma}, $\A_\gamma=\A_{\gamma\pm\varepsilon}$ as well as 
    $\gamma\pm\varepsilon < \bar{\gamma} - |\s|(1/\bar{p}{-}1/p)$. We already know that
    $\F^{\bar{\gamma}}_{\bar{p},\infty} \subset \F^{\gamma\pm\varepsilon}_{p,\infty}
    \subset\mathcal{B}^{\gamma\pm\varepsilon}_{p,\infty}$, i.e.\ it follows from
    the theory of real interpolation that
    \begin{align*}
        \big(\F^{\bar{\gamma}}_{\bar{p},\infty},\,\F^{\bar{\gamma}}_{\bar{p},\infty}\big)_{1/2,\bar{p}} 
        \subset \big(\mathcal{B}^{\gamma+\varepsilon}_{p,\infty},\,
        \mathcal{B}^{\gamma-\varepsilon}_{p,\infty}\big)_{1/2,\bar{p}}.
    \end{align*}
    Furthermore, it is trivial to see that $\F^{\bar{\gamma}}_{\bar{p},\infty}\subset
    \big(\F^{\bar{\gamma}}_{\bar{p},\infty},\,\F^{\bar{\gamma}}_{\bar{p},\infty}\big)_{1/2,\bar{p}}$.
    Thus, the asserted relation follows at once if we have shown that
    \begin{align*}
        \big(\bar{\mathcal{B}}^{\gamma+\varepsilon}_{p,\infty},\,
        \bar{\mathcal{B}}^{\gamma-\varepsilon}_{p,\infty}\big)_{1/2,\bar{p}}
        \subset \bar{\mathcal{B}}^\gamma_{p,\bar{p}}.
    \end{align*}
    The corresponding local bound is trivial. Hence, we only concentrate on the
    translation bound in detail. To this end, we follow the argument in the
    book of Triebel \cite[Sec.\ 2.4.2]{Triebel:1983}.
    Consider $\bf\in\big(\bar{\mathcal{B}}^{\gamma+\varepsilon}_{p,\infty},\,
    \bar{\mathcal{B}}^{\gamma-\varepsilon}_{p,\infty}\big)_{1/2,\bar{p}}$. In addition, 
    we consider $\fbar_1\in\bar{\mathcal{B}}^{\gamma+\varepsilon}_{p,\infty}$ and 
    $\fbar_2\in\bar{\mathcal{B}}^{\gamma-\varepsilon}_{p,\infty}$ such that 
    $\fbar = \fbar_1 + \fbar_2$. Then, we have the bound
    \begin{align*}
        \sum_{h\in\E_n}\Big\|\frac{|\fbar^{(n)}(y{+}h)-\Gamma_{y+h,y}\fbar^{(n)}(y)|_\zeta}
        {2^{-n(\gamma-\zeta)}}\Big\|^{\bar{p}}_{\ell^p_n}
        \lesssim 2^{-n\varepsilon\bar{p}}\big|\vvvert\fbar_1\vvvert_{\bar{\mathcal{B}}^{\gamma+\varepsilon}_{p,\infty}}
        +2^{2n\varepsilon}\vvvert\fbar_2\vvvert_{\bar{\mathcal{B}}^{\gamma-\varepsilon}_{p,\infty}}\big|^{\bar{p}},
    \end{align*}
    uniformly over all $n\geq 0$. Thus, by taking the infimum we deduce that
    \begin{align*}
        \bigg(\sum_{n\geq 0}\Big\|\frac{|\fbar^{(n)}(y{+}h)-\Gamma_{y+h,y}\fbar^{(n)}(y)|_\zeta}
        {2^{-n(\gamma-\zeta)}}\Big\|^{\bar{p}}_{\ell^p_n}\bigg)^{1/\bar{p}} 
        &\lesssim \bigg(\sum_{n\geq 0}\big|2^{-n\varepsilon}
        K(2^{2n\varepsilon},\fbar)\big|^{\bar{p}}\bigg)^{1/\bar{p}} 
        \\ &\lesssim\|\fbar\|_{(\bar{\mathcal{B}}^{\gamma+\varepsilon}_{p,\infty},\,
        \bar{\mathcal{B}}^{\gamma-\varepsilon}_{p,\infty})_{1/2,\bar{p}}}.
    \end{align*}
    As this is the asserted bound, we can conclude the argument.
\end{proof}

\section{The reconstruction operator} \label{sec:reconstruction}\newcommand{\bnorm}[1]{\bigg\| #1 \bigg\|}
\newcommand{\norm}[1]{\left\lVert #1 \right\rVert}
\newcommand{\nnorm}[1]{\vvvert  #1 \vvvert}
\newcommand{\bp}[1]{\left( #1 \right)}
\newcommand{\bv}[1]{\left\vert #1 \right\vert}

This section is devoted to the proof of the reconstruction theorem. The setting we work with consists of a fixed regularity structure $(\A, \T, \G)$ 
which we will always ask to include the polynomial regularity structure $(\Abar, \Tbar, \Gbar).$ Furthermore, models
for such a regularity structure are always understood to act canonically on the polynomial structure.

In addition, it is a powerful ingredient of the theory of regularity structures that it
comes with an appropriate notion of distance between two modelled distributions. We want to emphasize the fact
that the theory in particular allows for the situation, where two such modelled distributions are modelled with
respect to two \textit{different} models. Of course, we want to translate the corresponding notion
from \cite{Hairer:2014} to the scale $\F^\gamma_{p,q}$ and therefore define
\begin{align*}
    & \vvvert\bf; \fbar\vvvert  = \sup_{\zeta} \norm{|\bf(x) - \fbar(x)|_{\zeta}}_{L^p}  \\ 
    & +  \sup_{\zeta} \bigg\|\Big\|\mint{-}_{Q(0, 4 \lambda)} 
    \frac{|\bf (x{+}h) - \fbar(x{+}h) - \Gamma_{x+h,x} \bf (x) + \bar{\Gamma}_{x+h, x}\fbar(x)|_{\zeta}}
    {\lambda^{\gamma - \zeta}}\,\mathrm{d}h\Big\|_{L^q_{\lambda}}\bigg\|_{L^p}.
\end{align*}
The reconstruction theorem then reads as follows.
\begin{theorem}\label{reconstruction_theorem}
Let  $(\mathcal{A}, \mathcal{T}, \mathcal{G})$ be a regularity structure and $(\Pi, \Gamma)$ 
be a model as above. Furthermore, consider $\gamma \in \R_+ \setminus \N$ and
$1\leq p <\infty$ as well as $1\leq q \leq \infty$. 
We set $\alpha = \min (\mathcal{A} \setminus \N)\wedge \gamma. $ If $q<\infty$, choose $\bar{\alpha} < \alpha.$ 
In the case $q=\infty$, we let $\bar{\alpha}=\alpha$. Then, there exists a unique continuous linear map 
$\Recon\colon \F^{\gamma}_{ p, q} \to F^{\bar{\alpha}}_{p, q} $ such that
\begin{equation}\label{reconstruction_bound}
\bigg\|\Big\|\sup_{\eta \in \B^r} \frac{|\langle \Recon \bf - \Pi_x \bf (x), \eta^{\lambda}_{x} \rangle|}{\lambda^{\gamma}} \Big\|_{L^q_{\lambda}}\bigg\|_{L^p} \lesssim \vvvert\bf\vvvert_{\F^\gamma_{p,q}}\|\Pi\|( 1 + \|\Gamma\|),
\end{equation}
uniformly over all modelled distributions and all models.
Furthermore given a second model $(\bar{\Pi}, \bar{\Gamma})$ and denoting by $\bar{\Recon}$ the associated reconstruction operator, we have
\begin{equation}\label{lipschitz_reconstruction_bound}
\begin{aligned}
     \bigg\|&\Big\| \sup_{\eta \in \B^r} \frac{|\langle \Recon \bf - \bar{\Recon}\fbar - \Pi_x \bf(x) + \bar{\Pi}_x \fbar (x),
     \eta^{\lambda}_x  \rangle|}{\lambda^{\gamma}}\Big\|_{L^q_{\lambda}}\bigg\|_{L^p} \\
    & \lesssim \vvvert\bf; \fbar\vvvert \|\Pi\|(1 + \|\Gamma\|) + \vvvert\fbar\vvvert 
    \big(\|\Pi - \bar{\Pi}\|( 1 + \|\Gamma\|) + \|\bar{\Pi}\| \|\Gamma - \bar{\Gamma}\| \big).
\end{aligned}
\end{equation}
\end{theorem}

\begin{proof}
We will proceed as in \cite{Hairer:2017} and thus we will build a discrete approximation of the reconstruction
operator, namely we start with $\bf \in \F^{\gamma}_{p,q}$ and build $\fbar$ as in Proposition \ref{prop:forth}. 
Then, we define for $n \in \N$ and $x \in \Lambda_n$
$$ A^n_x = \langle \Pi_x \fbar^{(n)}(x), \varphi^n_x\rangle
$$
as well as $$ \Recon_n \bf = \sum_{x \in \Lambda_n} A^n_x\varphi^n_x.
$$
Now, we follow the original paper and divide the proof in three steps. For simplicity, we assume 
in the following that $q < \infty.$ As we have already discussed previously, the case $q = \infty$ 
follows essentially in the same way by replacing sums with suprema. The different choice of $\bar{\alpha}$ 
for $q = \infty $ is a direct consequence of the improved convergence result in Proposition 
\ref{proposition:characterisation_convergence}.

\textit{Step 1:} We prove convergence of $\Recon_n\bf \to \Recon \bf$ in $F^{\bar{\alpha}}_{p,q}$ for $\bar{\alpha} < \alpha \wedge 0.$ To see this, we only need to prove the assumptions of Proposition 
\ref{proposition:characterisation_convergence}. We start by estimating: 
$$ |A^n_x| \le \sum_{\gamma > \zeta}\norm{\Pi} 2^{- n (\frac{| \s |}{2}  + \zeta)}|\fbar^{(n)}(x)|_{\zeta}.$$
Hence, by applying the bound from Remark \ref{rem:sup_Lp} and the definition of $\alpha$ 
\begin{align*}
\bigg\| \sup_{n\geq 0} \sum_{y\in\Lambda_n}\frac{|A^n_y|}
{2^{-n(\alpha+\frac{|\s|}{2})p}}\bigg \|_{L^p} &\lesssim \sup_{\zeta} 
\bigg\| \sup_{n \ge 0} \sum_{y \in \Lambda_n}|\bf(y)|_{\zeta} \chi^n_y(x) 
\bigg\|_{L^p} \\ &\lesssim \nnorm{\fbar}_{\bar{\mathcal{F}}^{\gamma}_{p,q}} \lesssim 
\vvvert\bf\vvvert_{\F^\gamma_{p,q}},
\end{align*}
while for the second bound we have 
\begin{align*}
    |\delta A^n_x | = & \ \Big|\sum_{y \in \Lambda_{n+1}}\langle \Pi_y \fbar^{(n+1)}(y) - \Pi_x \fbar^{(n)}(x), \varphi^{n+1}_y \rangle \langle \varphi^{n+1}_y, \varphi^n_x\rangle\Big| \\
    \lesssim & \ \sup_{\zeta} \sum_{\substack{y \in \Lambda_{n+1} \\ \norm{y - x}_{\s} \le C2^{-n}}} | \fbar^{(n)}(x) - \Gamma_{x,y} \fbar^{(n+1)}(y)|_{\zeta}2^{-n(\frac{|\s|}{2} + \zeta )}
\end{align*}\\ 
so that we indeed obtain the second required bound of Proposition \ref{proposition:characterisation_convergence}:
\begin{align*}
     & \bigg\|\bigg(\sum_{n\geq 0}\Big|\sum_{y\in\Lambda_n} \frac{|\delta A^n_y|}{2^{-n(\gamma+\frac{|\s|}{2})}}\chi^n_y(x) \Big|^q\bigg)^\frac{1}{q}\bigg\|_{L^p} \\
     & \lesssim \bigg\|\bigg(\sum_{n\geq 0}\Big|\sum_{y\in\Lambda_n} \sum_{h \in \E^{C,0}_{n+1}} \frac{|\fbar^{(n)}(y) - \Gamma_{y,y+h} \fbar^{(n+1)}(y+h)|_{\zeta}}{2^{-n(\gamma- \zeta)}}\chi^n_y(x) \Big|^q\bigg)^\frac{1}{q}\bigg\|_{L^p}, 
\end{align*}
which is bounded by the norm of $\bf$ thanks to Remark \ref{rem:consistency1}. Hence, the sequence $\Recon_n \bf $ 
converges to some distribution $\Recon \bf$ in $\F^{\bar{\alpha}}_{p,q}$. From the reconstruction bound, 
which we will prove below, we then recover the convergence of $\Recon_n\bf$ in $\F^{\bar{\alpha}}_{p,q}$ 
for all $\bar{\alpha} < \alpha$, even if $\alpha>0$. The proof of this is the same as in the paper on the Besov scale 
by Hairer and Labb\'{e} \cite{Hairer:2017} and we omit it here.

\textit{Step 2:} We prove the reconstruction bound (\ref{reconstruction_bound}). To this end, 
for any $n \in \N $ we write $$\Recon \bf - \Pi_x \bf (x) = \Recon_{n} \bf - \mathcal{P}_{n} \Pi_x \bf (x)+ \sum_{m \ge n} \Recon_{m+1}\bf - \Recon_{m}
\bf - \mathcal{P}^{\bot}_{m}\Pi_x \bf(x), $$
where $\mathcal{P}_n$ is the projection onto $V_n$ and $\mathcal{P}^{\bot}_n$ is the projection 
onto the orthogonal complement $V_n^{\bot}$ of $V_n$ in $V_{n+1}.$ 
We will treat separately the terms of order equal to $\lambda$ and the terms of higher order, 
that is for any $\lambda$ we choose $n$ such that  $\lambda \in [2^{-n}, 2^{-n+1})$.
Then (cf.\ \cite[Rem 2.2]{Hairer:2017}), we find immediately that
\begin{align*}
	\bigg\|\Big\|\sup_{\eta \in \B^r} \frac{|\langle \Recon \bf - \Pi_x \bf (x), \eta^{\lambda}_{x} \rangle|}{\lambda^{\gamma}}\Big\|_{L^q_{\lambda}}\bigg\|_{L^p} \simeq 
	\bigg\|\Big\|\sup_{\eta \in \B^r} \frac{|\langle \Recon \bf - \Pi_x \bf (x), \eta^{2^{-n}}_{x} \rangle|}{2^{-n\gamma}}\Big\|_{\ell^q}\bigg\|_{L^p}
\end{align*}
as well as
$$ \Recon_{n} \bf - \mathcal{P}_{n} \Pi_x \bf (x) = \sum_{y \in \Lambda_{n}} (A^{n}_y - \langle \Pi_x \bf(x) , \varphi^{n}_y \rangle )\varphi^{n}_y.$$ 
For terms of order $n$, we get the following bound uniformly in $n$ and with 
the constant $C$ depending only on the support of $\phi$:
\begin{align*}
	|\langle \Recon_{n} \bf - \mathcal{P}_{n}  \Pi_x \bf (x) , \eta^{2^{-n}}_x \rangle | & \ = \Big|\sum_{y \in \Lambda_{n}} (A^{n}_y - \langle \Pi_x \bf(x) , \varphi^{n}_y \rangle ) \langle \varphi^{n}_y, \eta^{2^{-n}}_x \rangle \Big| \\ 
	&\hspace{-2cm}\lesssim  \sum_{ \substack{y \in \Lambda_{n} \\ \norm{y - x}_\s \le C2^{-n}} } 
	\mint{-}_{Q(y, 2^{-n})}
	|\langle \Pi_z (\bf(z) - \Gamma_{z, x} \bf(x)), 2^{n \frac{|\s|}{2}}\varphi^{n}_y \rangle |\,\mathrm{d}z  \\
	&\hspace{-2cm}\lesssim  \ \sup_{\zeta} \mint{-}_{Q(0, 2C2^{-n})} 
	\frac{|\bf(x{+}h) - \Gamma_{x+h, x} \bf(x)|_{\zeta}}{2^{n \zeta}}\,\mathrm{d}h,
\end{align*}
where in the last inequality we used that the sum is actually finite, uniformly in $n$. In this way, we can find 
a bound of required order for the quantity
\begin{align*}
	\bigg\|\Big\|\sup_{\eta \in \B^r}\frac{ | \langle \Recon_{n} \bf - \mathcal{P}_{n}  \Pi_x \bf (x) , \eta^{2^{-n}}_x \rangle |}{2^{-n \gamma}}\Big\|_{\ell^q}\bigg\|_{L^p}.
\end{align*}
Now, we pass to the terms of order greater than $n$. For this, we make use of the decomposition 
$ \Recon_{m+1}\bf - \Recon_{m} \bf = g_m + \delta f_m,$ where $g_m \in V_m$ and $\delta f_m \in V_m^{\bot}.$ 
In the sequel, we treat contributions from $\sum_{m \ge n} g_m$ and 
$\sum_{m \ge n} \big(\delta f_m - \mathcal{P}^{\bot}_m \Pi_x\bf(x)\big)$ 
differently. We start with the latter one, and following \cite{Hairer:2017} we find that
\begin{align*}
    |\langle \delta f_m & - \mathcal{P}^{\bot}_m \Pi_x\bf(x), \eta^{2^{-n}}_{x}\rangle| \\ 
    &\hspace{0.7cm}=  \Big|\sum_{\substack{y \in \Lambda_{m+ 1} \\ z \in \Lambda_{m}}} \! (A^{m+1}_y - \langle \Pi_x \bf(x),
    \varphi^{m+1}_{y}  \rangle) \langle \varphi^{m+1}_y , \psi^m_z \rangle \langle \psi^m_z , \eta^{2^{-n}}_x \rangle\Big| \\
    &\hspace{0.7cm}\lesssim \ \sup_{\zeta} \frac{ 2^{-m( r + \zeta)} }{2^{-nr}}
    \mint{-}_{Q(0, C'2^{-n})} |\bf(x{+}h) - \Gamma_{x+h, x} \bf(x)|_{\zeta}\,\mathrm{d}h.
\end{align*}
So, as before we can now easily bound the contribution from
\begin{align*}
	\bigg\|\big\|\sup_{\eta \in \B^r} \sum_{m \ge n}\frac{  |\langle \delta f_m - \mathcal{P}^{\bot}_m \Pi_x\bf(x), \eta^{2^{-n}}_{x}\rangle|}{2^{-n \gamma}}\Big\|_{\ell^q}\bigg\|_{L^p}
\end{align*}
in terms of the norm of $\bf$ and the norms of the model, as asserted. 
Finally we also estimate the contribution from $g_m$ as follows:
\begin{align*}
    |\langle g_m, \ \eta^{2^{-n}}_x \rangle |
    = & \ \Big|\sum_{ \substack{ y \in \Lambda_m \\ z \in \Lambda_{m+1} } } \langle \Pi_z \fbar^{(m+1)}(z) - \Pi_y \fbar^{(m)}(y), \varphi^{m+1}_z \rangle \langle \varphi^{m+1}_z, \varphi^{m}_y \rangle \langle \varphi^{m}_y, \eta^{2^{-n}}_{x} \ \Big| \\
   \lesssim & \sup_{\zeta} \! \sum_{y \in \Lambda^{C,x}_{m} } \sum_{ h \in \mathcal{E}^{C, 0}_{m+1}}  
   | \fbar^{(m)}(y{+}h) - \Gamma_{y+h,y} \fbar^{(m+1)}(y)|_{\zeta} 2^{-(m-n)|\s| -m \zeta}, 
\end{align*}
where we used the notation from the proof of Proposition \ref{prop:waveletBounds} with $n_0$ replaced by $n$. 
In fact, following the same line of argument of the proof of Proposition \ref{prop:waveletBounds} 
for the small scales, with $\eta \in (0,1)$ such that $\gamma + \frac{\eta - 1}{\eta}q |\s| > 0$ 
and with the resulting weight function $\theta(z) = 2^{-z(\gamma + \frac{\eta-1}{\eta} q|\s|)}$, 
where $z\geq 0$, we can now bound
\begin{align*}
	& \bnorm{ \Big\|\sup_{\eta \in \B^r} \sum_{m \ge n} \frac{ | \langle g_m, \eta^{2^{-n}}_x \rangle |}
	{2^{-n \gamma}}\Big\|_{\ell^q(n)}}_{L^p} \\
	& \lesssim \bigg\|\bigg(\sum_{n\geq 0}\Big|\sum_{y\in\Lambda_n}\sum_{h\in\E_n^M}
    \frac{|\fbar^{(n)}(y{+}h)-\Gamma_{y+h,y}\fbar^{(n)}(y)|_\zeta}{2^{-n(\gamma-\zeta)}}
    \chi^n_y(x)\Big|^q\bigg)^\frac{1}{q}\bigg\|_{L^p},
\end{align*}
which in turn is bounded by $\vvvert\fbar\vvvert_{\Fbar^\gamma_{p,q}}$ in view of Remark \ref{rem:consistency2}. Putting together these results we have by triangle inequality that 
\begin{align*}
	 \bigg\|\Big\|\sup_{\eta \in \B^r} &\frac{|\langle \Recon \bf - \Pi_x \bf (x), \eta^{2^{-n}}_{x} \rangle|}{2^{-n\gamma}}\Big\|_{\ell^q}\bigg\|_{L^p} \\
	 &\hspace{1cm}\leq\bigg\|\Big\|\sup_{\eta \in \B^r}\frac{ | \langle \Recon_{n} \bf - \mathcal{P}_{n}  
	 \Pi_x \bf (x) , \eta^{2^{-n}}_x \rangle |}{2^{-n \gamma}}\Big\|_{\ell^q}\bigg\|_{L^p} \\
	 &\hspace{2cm}+ \bigg\|\Big\|\sup_{\eta \in \B^r} \sum_{m \ge n}\frac{  |\langle \delta f_m - 
	 \mathcal{P}^{\bot}_m \Pi_x\bf(x), \eta^{2^{-n}}_{x}\rangle|}{2^{-n \gamma}}\Big\|_{\ell^q}\bigg\|_{L^p} \\
	 &\hspace{2cm}+ \bigg\|\Big\|\sup_{\eta \in \B^r} \sum_{m \ge n} 
	 \frac{ | \langle g_m, \eta^{2^{-n}}_x \rangle |}{2^{-n \gamma}}\Big\|_{\ell^q}\bigg\|_{L^p} \\
	 &\hspace{1cm}\lesssim \vvvert\bf\vvvert_{\F^\gamma_{p,q}}\|\Pi\|( 1 + \|\Gamma\|),
\end{align*}
uniformly over all models. If we take two different models, the reconstruction bound follows on the same line of argument
as presented in \cite{Hairer:2017}.

\textit{Step 3:} We prove the uniqueness of the reconstruction operator. For this purpose, 
we remind that for any Schwartz distribution $\psi$ and for any smooth test function $\rho$ supported in the unit ball, 
the convolution $\psi * \rho^{\delta} (x) = \langle \psi, \rho^{\delta}_x\rangle$ 
converges in distribution to $\psi$ as $\rho$ approximates a delta, i.e.\ if $\delta \to 0.$ 
Now, assume there are two distributions $\xi^1$ and $\xi^2$, which both satisfy the reconstruction bound
\eqref{reconstruction_bound}. Then, for any $\delta > 0 $ choose $n\in\N$ such that $\delta \in [2^{-n}, 2^{-n+1}).$ 
Now, we can bound
\begin{align*}
    \Big\|\frac{|\langle \xi^1 - \xi^2, \rho^{\delta}_x \rangle|}{\delta^{\gamma}}\Big\|_{L^p} &  
    \lesssim \bigg\|\int_{2^{-n}}^{2^{-n+1}}\sup_{\eta \in \B^r} 
    \frac{| \langle \xi^1 - \xi^2, \eta^{\lambda}_x \rangle |}{\lambda^{\gamma}}
		\,\frac{\mathrm{d}\lambda}{\lambda}\bigg\|_{L^p}.
\end{align*}
Furthermore, let us write $$f_n = \int_{2^{-n}}^{2^{-n+1}}
\sup_{\eta \in \B^r} \frac{| \langle \xi^1 - \xi^2, \eta^{\lambda}_x
\rangle |}{\lambda^{\gamma}}\,\frac{\mathrm{d}\lambda}{\lambda}.$$ Then, we know from the reconstruction bound that
$\norm{\norm{f_n}_{\ell^q}}_{L^p} < \infty.$ It is then an easy exercise to see that this implies 
$\norm{f_n}_{L^p} \to 0$ as $n \to \infty.$ Since taking $n \to \infty$ is equivalent to 
$\delta \to 0$, we thus find that $\langle \xi^1 - \xi^2, \rho^{\delta}_x \rangle \to 0$ in $L^p(\mathrm{d}x).$ 
Since $L^p$ convergence implies convergence in distribution, this tells us that $\xi^1 - \xi^2 = 0. $ 
Note that this argument does not rely on $\gamma,$ and it tells us that there is at most one distribution 
that can satisfy the reconstruction bound. This concludes the proof of the recontruction theorem.
\end{proof}

\begin{remark}\label{reconstruction_bound2}
For given $C>0$, we denote more generally by $\B^r_C(\R^d)$ the subspace of smooth functions on $\R^d$
which have a $\mathcal{C}^r$-norm bounded by 1 and which are supported in the cube $Q(0,C)$. A thorough investigation
of the proof of the reconstruction theorem then shows that we are also equipped with the bound
\begin{align}
    \bigg\|\Big\| \sup_{\eta \in \B^r_C(\R^d)} \frac{|\langle\Recon\bf-\Pi_x\bf(x), 
    \eta^{\lambda}_{x} \rangle|}{\lambda^{\gamma}}\Big\|_{L^q_\lambda}\bigg\|_{L^p}
    \lesssim_C \vvvert\bf\vvvert_{\F^\gamma_{p,q}}.
\end{align}
Of course, this is also true for the respective bound incorporating the case of two (in general different) models.
\end{remark}

\begin{remark}\label{reconstruction_bound3}
There is another version of the reconstruction bound which we would like to mention at this point. 
More precisely, it follows from the identity
\begin{align*}
    \langle\Recon\bf-\Pi_{x{+}h}&\bf(x{+}h),\eta^{2^{-n}}_{x{+}h}\rangle \\
    &\hspace{-1cm}= \langle\Recon\bf-\Pi_{x}\bf(x),\eta^{2^{-n}}_{x{+}h} \rangle
    +\langle\Pi_{x{+}h}(\bf(x{+}h)-\Gamma_{x+h,x}\bf(x)),\eta^{2^{-n}}_{x{+}h} \rangle
\end{align*}
and the bound contained in Remark \ref{reconstruction_bound2} that
\begin{align}
    \bigg\|\Big\| \mint{-}_{Q(0,C2^{-n})}\sup_{\eta \in \B^r(\R^d)}
    \frac{|\langle\Recon\bf-\Pi_{x{+}h}\bf(x{+}h),\eta^{2^{-n}}_{x{+}h} \rangle|}
    {2^{-n\gamma}}\,\mathrm{d}h\Big\|_{\ell^q}\bigg\|_{L^p}
    \lesssim_C \vvvert\bf\vvvert_{\F^\gamma_{p,q}},
\end{align}
again uniformly over all $\bf\in\F^\gamma_{p,q}$.
\end{remark}

The remainder of this section is devoted to proving a consequence of the reconstruction theorem,
namely we want to study the special case when the regularity structure at hand is only given by the
polynomial regularity structure $(\bar{\mathcal{A}}, \bar{\mathcal{T}}, \bar{\mathcal{G}})$,
together with the model $(\Pi, \Gamma)$ acting canonically on this structure. 
In this case, it turns out that modelled distributions correspond by means of the reconstruction operator
bijectively to Triebel--Lizorkin distributions over $\R^d.$ This is the content of the following

\begin{proposition}
    Take $\gamma \in \R_+ \setminus \N$, $1 \le p < \infty$ and $1 \le q \le \infty$. Then, 
    the reconstruction operator yields an isomorphism between the Banach spaces 
    $\F^{\gamma}_{p,q}(\bar{\mathcal{T}})$ and $F^{\gamma}_{p,q}(\R^d).$ 
\end{proposition}

We separate the proof of this proposition into the following two lemmata. In a first step, we will prove that the reconstruction operator is an injection between the two spaces above. Then, we will show that $\Recon$ is invertible.

\begin{lemma}
    Consider $\bf \in \F^{\gamma}_{p,q}(\bar{\mathcal{T}})$ and for $k \in \N^d$ with $|k|_{\s} < \gamma$, we define the 
    functions $\bf_k(x) = \mathcal{Q}_k \bf (x)$. Then, $k!\bf_k$ is the $k$-th derivative of $\Recon\bf$
    (in the sense of distributions) and $\Recon\bf = \bf_0 \in F^{\gamma}_{p,q}(\R^d)$.
\end{lemma}

\begin{proof}
Following the argument for the uniqueness of the reconstruction operator, we can see that 
there exists at most one distribution  $ \xi^{(k)} $  that satisfies the bound
\begin{align*}
    \bigg\|\Big\|\sup_{\eta \in \B^{r + |k|}}\frac{ | \langle \xi^{(k)} - \partial^k \Pi_x \bf(x) , \eta^{\lambda}_x \rangle |}{\lambda^{ \gamma - |k|}}\Big\|_{L^q_{\lambda}}\bigg\|_{L^p} < \infty.
\end{align*}
From the identity $(k!\bf_k - \Pi_x \bf(x))(y)  = \Q_k(\bf(y) - \Gamma_{y,x} \bf(x))$ 
and the definition of modelled distributions it is an easy computation to see 
that the above estimate is satisfied with $\xi^{(k)} = k! \bf_k.$ 
Moreover, also $\xi^{(k)} = \partial^k \Recon \bf$ satisfies such a bound. 
It therefore follows that $ \partial^k \Recon \bf= k!\bf_k$. 

It remains to prove that $\bf_0$ lies in $F^{\gamma}_{p,q}(\R^d)$. 
Note that the reconstruction theorem only gives us a regularity strictly 
smaller than $\gamma,$ at least for $ q < \infty.$ Since $\bf_0 \in L^p$ it 
is an easy computation to show that $$ \Big\|\sup_{\eta \in \B^r }
|\langle \bf_0, \eta_x^{\lambda} \rangle|\Big\|_{L^p} \lesssim \|\bf_0\|_{L^p}.$$
Furthermore, the second bound in the definition of the Triebel--Lizorkin space $F^\gamma_{p,q}$
is an immediate consequence of the reconstruction bound, since $\Pi_x \bf(x)$ is a
polynomial and the test functions from the definition annihilate polynomials
up to scaled degree $r$. 
\end{proof}

\begin{lemma}
    There exists a continuous injection $\iota\colon F^{\gamma}_{p,q} \to  \F^{\gamma}_{p,q}(\bar{T})$ 
    such that $ \Recon\iota\xi = \xi$ for any $\xi \in F^{\gamma}_{p,q}.$
\end{lemma}

\begin{proof}
The proof follows step by step the one in the original article. We start with proving 
that there is a continuous injection into the discrete space of modelled distributions 
$\bar{\F}^{\gamma}_{p,q}(\bar{T}).$ We fix some integer $q$ and $k \in \N^d$ with $|k|_{\s} \le q$ 
and a function $\eta$. To lighten the notation in the following, 
we will replace $|\cdot |_{\s}$ with $| \cdot |.$ Then, we define
\begin{equation*}
\begin{aligned}
P^{q}_{k,y} (\eta , u) = &
    \displaystyle \sum\limits_{| \ell + k | \le q} (-1)^{\ell} \partial_u^{\ell} 
    \left[ \eta (u {-} y) \frac{(y{-}u)^{\ell}}{\ell ! k !} \right].
\end{aligned}
\end{equation*}
For $\xi \in {F}^{\gamma}_{p,q}$  and $y \in \Lambda_n$, we also 
define $\fbar^{(n)}(y) \in \bar{\T}_{\gamma}^{-}$ by 
$$
\mathcal{Q}_k \fbar^{(n)}(y) = \langle \partial^k \xi (\cdot), P^{\lfloor \gamma \rfloor}_{k,y}  (\rho^n , \cdot)\rangle,
$$
where $\rho$ is any smooth symmetric function with compact support in the unit ball, 
and which integrates to $1$. In addition, $\rho^n$ is a shorthand for $\rho^{2^{-n}}$,
i.e.\ we use $L_1$ scaling for $\rho.$ This is in analogy of the definition 
given by Proposition \ref{prop:forth}, cf.\ also \cite{Hairer:2017}. Now, 
it is our task to show that $\fbar$ lies in $\bar{\F}^{\gamma}_{p,q}(\bar{T}).$ 

The idea is of course to transform the bounds in the definition of $F^{\gamma}_{p,q}$ 
into the bounds from the definition of $\bar{\F}^{\gamma}_{p,q}(\bar{T})$. 
For this purpose, we will also need to define the functions
\begin{align*}
    \Phi^{k,n}_{y,h} (\cdot) & =  P^{\lfloor \gamma \rfloor}_{k,y+h} (\rho^n , \cdot) - 
    \sum_{|\ell + k| \le \lfloor \gamma \rfloor} (-h)^{\ell}  \frac{(\ell + k)!}{ \ell! k !} \partial^{\ell}_u 
    P^{\lfloor \gamma \rfloor}_{k +\ell, y} (\rho^n , \cdot), \\
    \Psi^{k,n}_{y} (\cdot) & = P^{\lfloor \gamma \rfloor}_{k,y}(\rho^n , \cdot) - 
    P^{\lfloor \gamma \rfloor}_{k,y}(\rho^{n+1} , \cdot),
\end{align*}
where $h$ is an element of $\mathcal{E}_n$. These functions have been chosen appropriately, so that the following two equalities hold: 
\begin{align*}
    \langle \partial^k \xi, \Phi^{k,n}_{y,h}\rangle = & \mathcal{Q}_k (\fbar^{(n)}(y)  -  \Gamma_{y+h, y} \fbar^{(n)}(y) ),\\
    \langle \partial^k \xi, \Psi^{k,n}_{y}\rangle = & \mathcal{Q}_k (\fbar^{(n)}(y)  -   \fbar^{(n+1)}(y)).
\end{align*}
One easily recognizes that $P^{\lfloor \gamma \rfloor}_{k,y}(\rho^n , \cdot), 
\ \Phi^{k,n}_{y,h}$ and $ \Psi^{k,n}_{y},$ are smooth functions with compact 
support in the ball of radius $2^{-n}$ about $y$.
Thus, taking $n=0$ the first bound from Definition \ref{def:fbar} follows immediately:
\begin{align*}
	\sup_{\zeta} \bigg(  \sum_{y \in \Lambda_0} |\fbar^{(0)}(y)|^p_{\zeta}\bigg)^{1/p} 
	\lesssim \vvvert\xi\vvvert_{F^\gamma_{p,q}}.
\end{align*}
Now, we turn to the translation and consistency bounds. 
Assuming for the moment that we can prove that these functions 
annihilate polynomials of degree smaller than $ \gamma  - |k|,$ 
it follows that $\fbar \in \bar{\mathcal{F}}^{\gamma}_{p,q}.$ Indeed, 
since $\partial^k \xi \in F^{\gamma - |k|}_{p,q}$ we obtain bounds of the form
\begin{align*}
    \bigg\|\bigg(\sum_{n\geq 0}\Big| &\sum_{y\in\Lambda_n}\sum_{h\in\E_{n}}\frac{|\fbar^{(n)}(y{+}h)-\Gamma_{y{+}h,y}\fbar^{(n)}(y)|_\zeta}{2^{-n(\gamma-\zeta)}} \chi^n_y(x)\Big|^q\bigg)^\frac{1}{q}\bigg\|_{L^p} \\
    & \lesssim \bigg\|\bigg(\sum_{n\geq 0}\Big|\sum_{y\in\Lambda_n}\sum_{h\in\E_{n}}\sum_{|k| = \zeta} \frac{ | \langle \partial^k \xi, \Phi^{k,n}_{y,h}\rangle| }{2^{-n(\gamma-|k|)}} \chi^n_y(x)\Big|^q\bigg)^\frac{1}{q}\bigg\|_{L^p} \\ 
    & \lesssim \bigg\|\bigg(\sum_{n\geq0}\Big| \sup_{\eta \in \B^r_{\lfloor \gamma \rfloor - |k|}} \frac{ | \langle \partial^k \xi, \eta^n_x \rangle| }{2^{-n(\gamma-|k|)}}\Big|^q\bigg)^\frac{1}{q}\bigg\|_{L^p}  
    \lesssim \vvvert\xi\vvvert_{F^\gamma_{p,q}},
\end{align*}
and similarly for the consistency bound. Thus, all that is left to prove is 
that $\Phi$ and $\Psi$ annihilate polynomials of degree smaller than $\gamma - |k|.$ 
For, it is an easy calculation to see that
\begin{equation}\label{eqn:meanP}
\int P^{\lfloor \gamma \rfloor}_{k,y} (\rho^n , u)\,\mathrm{d}u = \frac{1}{k!},
\end{equation}
and from this it follows that $\Phi$ and $\Psi$ vanish when integrated 
against any constant, because up to some derivative terms they are made 
out of differences of the function $P.$ By the same arguments presented 
in \cite{Hairer:2017} it is possible to prove that
$$
    \langle P^{\lfloor \gamma \rfloor}_{k,y}(\rho^n , \cdot) , ( \cdot - y )^m \rangle = 0,
$$
for degrees $m$ such that $m \neq 0$ and $|m {+} k| \le \lfloor \gamma \rfloor .$ 
Together with the fact that all $P_k$'s have the same volume mean, this is sufficient 
in order to see that $\Psi$ vanishes when integrated against the requested polynomials. 
For completeness we treat the case for $\Phi,$ since it is left as an exercise in \cite{Hairer:2017}. 

Fix a degree $m$ such that $|m {+} k| \le \lfloor \gamma \rfloor.$ 
We want to prove that $$\langle \Phi^{k,n}_{y,h} , (\cdot - (y{+}h))^m \rangle = 0.$$ Indeed we have
\begin{align*}
	\langle \Phi^{k,n}_{y,h} , &  (\cdot - (y+h))^m \rangle = \\
	& =  - \sum_{|\ell + k| \le \lfloor \gamma \rfloor} (-h)^{\ell}  
	\frac{(\ell + k)!}{ \ell! k !} \int \partial^{\ell}_u P^{\lfloor \gamma \rfloor}_{ k +\ell, y}
	(\rho^n , u)( u - (y+h) )^m\,\mathrm{d}u\\
	& = - \sum_{\ell \le m} h^{\ell}  \frac{(\ell + k)!}{ \ell! k !} \int \frac{m!}{(m-\ell)!}  
	P^{\lfloor \gamma \rfloor}_{k +\ell, y} (\rho^n , u)(-h)^{m-\ell}\,\mathrm{d}u \\
	& = -\frac{h^m}{k!} \sum_{\ell \le m} (-1)^{m-\ell} \frac{m!}{l! (m-l)!}  = 0,
\end{align*}
where we have repeatedly used that fact that $P^{\lfloor \gamma \rfloor}_{k +\ell,z}$ 
vanishes against polynomials centred in $z$, and in the second line we applied integration by parts. 
In the last line we first used (\ref{eqn:meanP}) and then the binomial formula.
At this point we can conclude that $\fbar \in \bar{\F}^{\gamma}_{p,q}.$ 

What we finally need to prove is that $\Recon \iota \xi = \xi,$ where $\iota \xi$ is of course given by the modelled
distribution $\bf$ corresponding to $\fbar$ through the isomorphism of Proposition \ref{prop:forth}. 
We observe for this purpose that $\Recon i \xi = \mathcal{Q}_0 \bf $, and that it is 
part of the statement of Proposition \ref{prop:forth} that
$\Gamma_{x,x_n}\fbar^{(n)}(x_n)\to\mathcal{Q}_0 \bf (x)$
in $L^p$ as $n\to\infty$. It becomes clear that $\xi = \lim_n  \Gamma_{x,x_n}\fbar^{(n)}(x_n)$ just by embedding
$F^{\gamma}_{p,q} \subset L^p.$ Indeed, it is a well known result for convolutions on $L^p$ that 
$\xi*\rho^n \to \xi$ in $L^p$. Then, one can follow exactly the arguments of Hairer and Labb\'{e} 
\cite{Hairer:2017}, which finally concludes the proof.
\end{proof}

\section{Convolution against singular kernels} \label{sec:schauder}
In this section, we want to prove Schauder type estimates on the level of modelled distributions
of class $\F^\gamma_{p,q}$, i.e.\ given a kernel $K$ which improves regularity by some $\beta>0$
we want to construct a linear map $\F^\gamma_{p,q}\ni\bf\mapsto\K_\gamma\bf\in\F^{\gamma+\beta}_{p,q}$.
In addition, we want to ensure that the construction behaves suitably under the action 
of the reconstruction operator, i.e.\
\begin{align}\label{ReconConvolution}
    \Recon\K_\gamma\bf = K\ast\Recon\bf.
\end{align}
Of course, almost all non-trivial parts of this program were already carried out by Hairer in \cite[Sec. 5]{Hairer:2014}.
What is essentially left to us is to prove the announced Schauder estimate in the framework of the scale $\F^\gamma_{p,q}$
as well as to check the validity of \eqref{ReconConvolution} in this setting. But before we state and prove 
the precise results, let us recall the ingredients of the construction
of the linear map $\bf\mapsto\K_\gamma\bf$.

To this end, we still assume that the regularity structure at hand
contains the polynomial structure $(\Abar,\Tbar,\Gbar)$. Furthermore, let 
$K\colon Q(0,1)\to\R$ be a kernel which is smooth at every point in $Q(0,1)\setminus\{0\}$.
We also assume that the kernel is
$\beta$-regularizing in the sense of \cite[Assumptions 5.1, 5.4]{Hairer:2014} for some $\beta>0$.
More precisely, we assume that for every $n\geq 0$ there exists a smooth kernel $K_n\colon\R^d\to\R$ such that
the following properties are satisfied:
\begin{itemize}\itemsep3pt
    \item[i)]   The decomposition $K=\sum_{n\geq 0} K_n$ holds true on $Q(0,1)$.
    \item[ii)]  The kernel $K_0$ has support in $Q(0,1)$. Furthermore, we have
                for every $n\geq 0$ and every $x\in\R^d$ the scaling relation
                \begin{align}
                     K_n(x) = 2^{-n(\beta-|\s|)}K_0(2^{n\s}x).
                \end{align}
    \item[iii)] The kernel $K_0$ kills polynomials with scaled degree at most $r$.           
\end{itemize}
Here and for what follows in this section, we let $r>\max|\A_{\gamma+\beta}|$.

Next, we want to recall the assumption that the polynomial structure provides the only
integer homogeneities for the regularity structure under consideration. 
It was pointed out by Hairer (cf. \cite[Sec. 5]{Hairer:2014}) that in general it is necessary that $\K_\gamma\bf$ 
takes non-zero values in parts of the regularity structure which do not incorporate the polynomial structure. 
In order to encode the action of the kernel $K$ in these parts with non-integer homogeneities, we assume that
our regularity structure comes with an abstract integration map $\mathcal{I}\colon\T\to\T$ of order $\beta$, 
i.e.\ (cf.\ \cite[Def 5.7]{Hairer:2014})
\begin{itemize}\itemsep3pt
    \item[i)] $\mathcal{I}|_{\T_\zeta}\colon \T_\zeta \to \T_{\zeta+\beta}$,
    \item[ii)] $\mathcal{I}$ annihilates all elements in the structure $\Tbar$,
    \item[iii)] $\mathcal{I}\Gamma\bTau - \Gamma\mathcal{I}\bTau \in\Tbar$ for all $\Gamma\in\G$ and $\bTau\in\T$.
\end{itemize}
We refer again to \cite[Rem 5.8]{Hairer:2014} for a discussion of this notion.
Now, let $1\leq p<\infty$, $1\leq q\leq\infty$ and $\gamma>0$. 
Given a modelled distribution $\bf\in\F^\gamma_{p,q}$,
we then define as in \cite[Equ.\ (5.15)]{Hairer:2014} the operator
\begin{equation}
\begin{aligned}
\K_{\gamma}\bf(x) &= \mathcal{I}\big(\bf(x)\big) +\sum_{\zeta\in\A_\gamma}
\sum_{|k|_\s<\zeta{+}\beta}\sum_{n\geq 0}
\frac{\Xk}{k!}\langle\Pi_{x}\Q_\zeta\bf(x),D^k_1K_n(x,\cdot)\rangle \\
&\quad+\sum_{|k|_\s<\gamma+\beta}\sum_{n\geq 0}\frac{\Xk}{k!}
\langle\Recon\bf-\Pi_{x}\bf(x),D^k_1K_n(x,\cdot)\rangle.
\end{aligned}
\end{equation}
Here, we identified $K_n(x,z)= K_n(x{-}z)$. For convenience, let us also introduce for each $n\geq 0$
the operator
\begin{align*}
\K_{n,\gamma}\bf(x) &= 
\sum_{\zeta\in\A_\gamma}\sum_{|k|_\s<\zeta{+}\beta}\frac{\Xk}{k!}\langle\Pi_{x}\Q_\zeta\bf(x),D^k_1K_n(x,\cdot)\rangle \\
&\quad\quad\qquad+\sum_{|k|_\s<\gamma+\beta}\frac{\Xk}{k!}\langle\Recon\bf-\Pi_{x}\bf(x),D^k_1K_n(x,\cdot)\rangle.
\end{align*}
Note that due to the bounds provided by a model,
the reconstruction theorem as well as the scaling properties of the kernel $K_0$, 
the occurring sums over $n\geq 0$ converge absolutely in $L^p(\R^d)$. Of course,
we are not only interested in this local bound but actually in the improved version
concerning the whole norm for the scale of spaces $\F^\gamma_{p,q}$. 

Apart from that, we also want to ensure that \eqref{ReconConvolution} holds true.
Recall that this is of paramount importance when one wants to recast abstract solutions of fixed-point maps on the
level of modelled distributions as mild solutions to regularized versions of an SPDE under consideration.
Verifying the Schauder type estimates as well as \eqref{ReconConvolution} is in turn inseparably linked to the task of 
finding a class of models which act appropriately on the abstract integration map $\mathcal{I}$.

To this end, the notion of an admissible model was introduced by Hairer, cf.\ \cite[Def 5.9]{Hairer:2014}. For a model
to be admissible it is required that the following relation (appropriately interpreted) holds true:
\begin{align*}
    \Pi_x\mathcal{I}\bTau(y) = \langle\Pi_x\bTau,K(y,\cdot)\rangle
    - \sum_{|k|_\s<\zeta{+}\beta}\frac{\Xk}{k!}\langle\Pi_{x}\bTau,D^k_1K(x,\cdot)\rangle,
\end{align*}
where $\zeta\in\A_\gamma$ and $\bTau\in\T_\zeta$. It is indeed a non-trivial result that this definition 
really satisfies the required analytic bounds of a model. For this, and a discussion of the notion of admissible 
models, we again refer to \cite{Hairer:2014}. 

\begin{theorem}\label{theo:Schauder}
Let $(\A,\T,\G)$ be a regularity structure and $K\colon Q(0,1)\to\R$ be a kernel
such that all of the assumptions listed above are satisfied.
Let also $1\leq p < \infty$ and $1\leq q\leq \infty$.
In addition, we assume that $\gamma\in\R_+\setminus\N$ as well as $\zeta+\beta\notin\N$ 
for all $\zeta\in(\A_\gamma\cup\{\gamma\})\setminus\N$. Given an admissible model $(\Pi,\Gamma)$, 
the linear operator $\K_\gamma$ then maps $\F^\gamma_{p,q}$ continuously
into $\F^{\gamma+\beta}_{p,q}$, and the following bound
\begin{align}\label{eq:Convol}
\vvvert\K_{\gamma}\bf\vvvert_{\F^{\gamma+\beta}_{p,q}} 
\lesssim \|\Pi\|(1+\|\Gamma\|)\vvvert\bf\vvvert_{\F^{\gamma}_{p,q}}
\end{align}
holds true uniformly over all $\bf\in\F^\gamma_{p,q}$ and models $(\Pi,\Gamma)$. We also have
\begin{align*}
    \Recon\K_\gamma\bf = K\ast\Recon\bf.
\end{align*}
Furthermore, consider the situation incorporating a second model $(\bar{\Pi},\bar{\Gamma})$,
and denote by $\bar{\K}_\gamma$ the associated abstract convolution operator.
Then, the quantity $\vvvert\K_\gamma\bf;\bar{\K}_\gamma\fbar\vvvert_{\F^{\gamma+\beta}_{p,q}}$ is bounded by
\begin{align*}
\|\Pi\|(1+\|\Gamma\|)\vvvert\bf;\fbar\vvvert_{\F^\gamma_{p,q}}
+\|\Pi-\bar{\Pi}\|(1+\|\Gamma\|)\vvvert\fbar\vvvert_{\F^\gamma_{p,q}}
+\|\bar{\Pi}\|\|\Gamma-\bar{\Gamma}\|\vvvert\fbar\vvvert_{\F^\gamma_{p,q}},
\end{align*}
uniformly over all $\bf\in\F^\gamma_{p,q}(\Pi,\Gamma)$, 
all $\fbar\in\F^\gamma_{p,q}(\bar{\Pi},\bar{\Gamma})$ as well as all models
$(\Pi,\Gamma)$ and $(\bar{\Pi},\bar{\Gamma})$.
\end{theorem}

\begin{proof}
In the course of the proof, we will focus only on the bounds in the case of one model.
The asserted bound concerning two (in general different) models then follows from the 
same arguments which are already employed in \cite{Hairer:2017}.

We begin with the respective local bound. For, we first fix a non-integer homogeneity
$\zeta<\gamma+\beta$. In this case, $\Q_\zeta\K_\gamma\bf(x)=\Q_\zeta\mathcal{I}\big(\bf(x)\big)$
from which the required bound immediately follows due to the properties of the abstract
convolution map $\mathcal{I}$. Hence, let us move on with with contributions in the
canonical structure, i.e.\ consider $k\in\N_0^d$ such that $|k|_\s<\gamma+\beta$.
We obviously have
\begin{align*}
    k!\Q_{\Xk}&\K_{n,\gamma}\bf(x) \\ &= \langle\Recon\bf - \Pi_x\bf(x),D^k_1K_n(x,\cdot)\rangle
    +\sum_{\zeta>|k|_\s-\beta}\langle\Pi_x\Q_\zeta\bf(x),D^k_1K_n(x,\cdot)\rangle.
\end{align*}
Let us derive bounds for each of these two terms. With respect to the first one, we obtain
with the help of the reconstruction bound
\begin{align*}
    \sum_{n\geq 0}\big\||\langle&\Recon\bf - \Pi_x\bf(x),D^k_1K_n(x,\cdot)\rangle|\big\|_{L^p} \\
    &\lesssim \sum_{n\geq 0}2^{-n(\gamma+\beta-|k|_\s)}\bigg\|\sup_{\eta\in\B^r(\R^d)}
    \frac{|\langle\Recon\bf-\Pi_x\bf(x),\eta^{2^{-n}}_x\rangle|}{2^{-n\gamma}}\bigg\|_{L^p}
    \lesssim \vvvert\bf\vvvert_{\F^\gamma_{p,q}}.
\end{align*}
For the second term, we simply estimate
\begin{align*}
    \sum_{\zeta>|k|_\s-\beta}\sum_{n\geq 0}\big\||\langle&\Pi_x\Q_\zeta\bf(x),D^k_1K_n(x,\cdot)\rangle|\big\|_{L^p} \\
    &\lesssim \sum_{\zeta>|k|_\s-\beta}\sum_{n\geq 0}2^{-n(\zeta+\beta-|k|_\s)}\big\||\bf(x)|_\zeta\big\|_{L^p}
    \lesssim \vvvert\bf\vvvert_{\F^\gamma_{p,q}}.
\end{align*}
This concludes the discussion for the local bound.

We turn to the respective translation bound. To this end, let us start again with a non-integer
homogeneity $\zeta<\gamma+\beta$. We then have (cf.\ \cite{Hairer:2014}) $$\Q_\zeta\big(\K_\gamma\bf(x{+}h)- 
\Gamma_{x+h,x}\K_\gamma\bf(x)\big)=\Q_\zeta\mathcal{I}\big(\bf(x{+}h)-\Gamma_{x+h,x}\bf(x)\big)$$
for all $x\in\R^d$, all $n\geq 0$ and all $h\in Q(0,4\cdot 2^{-n})$. Hence, the desired
bound follows from the translation bound for $\bf$ and the properties of $\mathcal{I}$.
Fix now some $k\in\N_0^d$ such that $|k|_\s<\gamma+\beta$, i.e.\ we proceed with contributions
in the polynomial structure. For, we also fix $x\in\R^d$, $n\geq 0$ as well as $h\in Q(0,4\cdot 2^{-n})$.
In the following, we separate the decomposition $K=\sum_{m\geq 0} K_m$ into a sum over
$m<n$ and a sum over $m\geq n$, and will argue differently for each of these two contributions.

Let us treat the case of $m\geq n$. We make use of the identity (cf.\ \cite{Hairer:2014})
\begin{align*}
    k!\Q_{\Xk}&\big(\K_{m,\gamma}\bf(x{+}h)-\Gamma_{x+h,x}\K_{m,\gamma}\bf(x)\big) \\ 
		&\hspace{0.5cm}=\langle\Recon\bf-\Pi_{x+h}\bf(x{+}h),D^k_1K_m(x{+}h,\cdot)\rangle \\
    &\hspace{+1.5cm}+\sum_{\substack{l\in\N_0^d \\ |k+l|_\s<\gamma+\beta}}
    \frac{h^l}{l!}\langle\Recon\bf-\Pi_x\bf(x),D^{k+l}_1K_m(x,\cdot)\rangle \\
    &\hspace{+1.5cm}+\sum_{\zeta>|k|_\s-\beta}\langle\Pi_{x+h}\Q_\zeta
    \big(\bf(x{+}h)-\Gamma_{x+h,x}\bf(x)\big),D^k_1K_m(x{+}h,\cdot)\rangle.
\end{align*}
With respect to the first term in this decomposition, we have the bound
\begin{align*}
    \sum_{m\geq n}&\mint{-}_{Q(0,4\cdot 2^{-n})}
    \frac{|\langle\Recon\bf-\Pi_{x+h}\bf(x{+}h),D^k_1K_m(x{+}h,\cdot)\rangle|}
    {2^{-n(\gamma-|k|_\s)}}\,\mathrm{d}h \\ &\lesssim \sum_{m\geq n}2^{-m(\beta-|k|_\s)}
    \mint{-}_{Q(0,4\cdot 2^{-n})}\sup_{\eta\in\B^r(\R^d)}
    \frac{|\langle\Recon\bf-\Pi_{x+h}\bf(x{+}h),\eta^{2^{-m}}_{x+h}\rangle|}
    {2^{-n(\gamma-|k|_\s)}}\,\mathrm{d}h \\ &\lesssim \sum_{m\geq n}2^{(n-m)(\gamma+\beta-|k|_\s)}
    \mint{-}_{Q(x,4\cdot 2^{-n})}\sup_{\eta\in\B^r(\R^d)}
    \frac{|\langle\Recon\bf-\Pi_{z}\bf(z),\eta^{2^{-m}}_{z}\rangle|}{2^{-m\gamma}}\,\mathrm{d}z.
\end{align*}
We can proceed from here by estimating
\begin{align*}
    \mint{-}_{Q(x,4\cdot 2^{-n})}&\sup_{\eta\in\B^r(\R^d)}
    \frac{|\langle\Recon\bf-\Pi_{z}\bf(z),\eta^{2^{-m}}_{z}\rangle|}{2^{-m\gamma}}\,\mathrm{d}z \\
    &\hspace{-1.6cm}\lesssim \sum_{\substack{y\in\Lambda_m \\ \|y{-}x\|_\s\leq 4\cdot 2^{-n}}}
    2^{(n-m)|\s|}\mint{-}_{Q(y,4\cdot 2^{-m})}\sup_{\eta\in\B^r(\R^d)}
    \frac{|\langle\Recon\bf-\Pi_{z}\bf(z),\eta^{2^{-m}}_{z}\rangle|}{2^{-m\gamma}}\,\mathrm{d}z \\
    &\hspace{-1.6cm}\lesssim 2^{(n-m)\kappa'}\bigg\{
    \mathcal{M}\bigg(\bigg|\sum_{y\in\Lambda_m}\mint{-}_{Q(y,4\cdot 2^{-m})}
    \sup_{\eta\in\B^r}\frac{|\langle\Recon\bf-\Pi_{z}\bf(z),\eta^{2^{-m}}_{z}\rangle|}
    {2^{-m\gamma}}\chi^m_y\,\mathrm{d}z\bigg|^\kappa\bigg)(x)\bigg\}^\frac{1}{\kappa},
\end{align*}
which holds uniformly over all $0<\kappa<1$. Here, we also 
introduced the shorthand $\kappa'=|\s|(\kappa-1)/\kappa$.
Now, choosing $\kappa$ sufficiently close to one, we obtain with 
the vector-valued maximal inequality
\begin{align*}
    \bigg\|\Big\| \sum_{m\geq n}&\mint{-}_{Q(0,4\cdot 2^{-n})}
    \frac{|\langle\Recon\bf-\Pi_{x+h}\bf(x{+}h),D^k_1K_m(x{+}h,\cdot)\rangle|}
    {2^{-n(\gamma-|k|_\s)}}\,\mathrm{d}h\Big\|_{\ell^q}\bigg\|_{L^p} \\ &\lesssim
    \bigg\|\Big\| \sum_{y\in\Lambda_n}\mint{-}_{Q(y,4\cdot 2^{-n})}
    \sup_{\eta\in\B^r}\frac{|\langle\Recon\bf-\Pi_{z}\bf(z),\eta^{2^{-n}}_{z}\rangle|}
    {2^{-n\gamma}}\chi^n_y(x)\,\mathrm{d}z\Big\|_{\ell^q}\bigg\|_{L^p} \\ &\lesssim
    \bigg\|\Big\|\mint{-}_{Q(0,5\cdot 2^{-n})}\sup_{\eta\in\B^r}
    \frac{|\langle\Recon\bf-\Pi_{x{+}h}\bf(x{+}h),\eta^{2^{-n}}_{x{+}h}\rangle|}
    {2^{-n\gamma}}\,\mathrm{d}h\Big\|_{\ell^q}\bigg\|_{L^p}.
\end{align*}
Thus, thanks to Remark \ref{reconstruction_bound3} we eventually arrive at a bound of required order.
Concerning the second term, we derive the bound
\begin{align*}
    \sum_{m\geq n}&\sum_{\substack{l\in\N_0^d \\ |k+l|_\s<\gamma+\beta}}\mint{-}_{Q(0,4\cdot 2^{-n})}
    \frac{|h^l|}{l!}\frac{|\langle\Recon\bf-\Pi_x\bf(x),D^{k+l}_1K_m(x,\cdot)\rangle|}{2^{-n(\gamma-|k|_\s)}}\,\mathrm{d}h
    \\ &\lesssim \sum_{m\geq n}\sum_{\substack{l\in\N_0^d \\ |k+l|_\s<\gamma+\beta}}
    \mint{-}_{Q(0,4\cdot 2^{-n})}\frac{\|h\|_\s^{|l|_\s}}{2^{m(\beta-|k+l|_\s)}}
    \sup_{\eta\in\B^r(\R^d)}\frac{|\langle\Recon\bf-\Pi_{x}\bf(x),\eta^{2^{-m}}_{x}\rangle|}
    {2^{-n(\gamma-|k|_\s)}}\,\mathrm{d}h \\ &\lesssim \sum_{\substack{l\in\N_0^d \\ |k+l|_\s<\gamma+\beta}}\sum_{m\geq n}
    2^{(n-m)(\gamma+\beta-|k+l|_\s)}\sup_{\eta\in\B^r(\R^d)}\frac{|\langle\Recon\bf-\Pi_{x}\bf(x),\eta^{2^{-m}}_{x}\rangle|}
    {2^{-m\gamma}}.
\end{align*}
Thus, by virtue of Young's inequality and the reconstruction bound we again obtain
a bound of required order. It remains to derive a suitable bound for the third term.
In this case, we have
\begin{align*}
    \sum_{m\geq n}&\sum_{\zeta>|k|_\s-\beta}\mint{-}_{Q(0,4\cdot 2^{-n})}
    \frac{|\langle\Pi_{x+h}\Q_\zeta\big(\bf(x{+}h)-\Gamma_{x+h,x}\bf(x)\big),D^k_1K_m(x{+}h,\cdot)\rangle|}
    {2^{-n(\gamma-|k|_\s)}}\,\mathrm{d}h \\ &\lesssim \sum_{\zeta>|k|_\s-\beta}\sum_{m\geq n} 2^{-m(\zeta+\beta-|k|_\s)}
    \mint{-}_{Q(0,4\cdot 2^{-n})}\frac{|\bf(x{+}h)-\Gamma_{x+h,x}\bf(x)|_\zeta}{2^{-n(\gamma-|k|_\s)}}\,\mathrm{d}h
    \\ &\lesssim \sum_{\zeta>|k|_\s-\beta}\sum_{m\geq n}2^{(n-m)(\zeta+\beta-|k|_\s)}
    \mint{-}_{Q(0,4\cdot 2^{-n})}\frac{|\bf(x{+}h)-\Gamma_{x+h,x}\bf(x)|_\zeta}{2^{-n(\gamma-\zeta)}}\,\mathrm{d}h,
\end{align*}
i.e.\ this time the desired bound follows immediately. This concludes our discussion of the regime $m\geq n$.

We move on with the regime $m<n$. For, following Hairer \cite{Hairer:2014} we introduce the test functions
\begin{align*}
    K^{k,\gamma+\beta}_{n,x,y} = D^k_1K_n(y,\cdot) - \sum_{\substack{l\in\N_0^d \\ |k+l|_\s<\gamma+\beta}}
    \frac{(y-x)^l}{l!}D^{k+l}_1K_n(x,\cdot).
\end{align*}
We then make use of the identity (cf.\ \cite{Hairer:2014})
\begin{align*}
    k!\Q_{\Xk}\big(\K_{m,\gamma}\bf(x{+}h)-\Gamma_{x+h,x}\K_{m,\gamma}\bf(x)\big) &=
    \langle\Recon\bf-\Pi_x\bf(x),K^{k,\gamma+\beta}_{m,x,x+h}\rangle \\ &\hspace{-4.5cm}+
    \sum_{\zeta\leq |k|_\s-\beta}\langle\Pi_{x+h}\Q_\zeta\big(\bf(x{+}h)-\Gamma_{x+h,x}\bf(x)\big),
    D^k_1K_m(x{+}h,\cdot)\rangle.
\end{align*}
As above, we bound each of the two terms on the right hand side of this identity individually.
Let us begin with the second one. To this end, we first observe that terms related to
homogeneities with $\zeta+\beta=|k|_\s$ actually do not contribute. Indeed, our assumptions
imply in this case that $\zeta\in\N_0$. Thus, as $K_0$ annihilates polynomials of scaled
degree less than or equal to $r$, these terms then vanish due to the action of the model on the canonical structure.
Hence, we can restrict the sum to homogeneities $\zeta\in\A_\gamma$ such that $\zeta < |k|_\s-\beta$.
But then we get the bound
\begin{align*}
    \sum_{\zeta<|k|_\s-\beta}&\sum_{m < n}\mint{-}_{Q(0,4\cdot 2^{-n})}
    \frac{|\langle\Pi_{x+h}\Q_\zeta\big(\bf(x{+}h)-\Gamma_{x+h,x}\bf(x)\big),
    D^k_1K_m(x{+}h,\cdot)\rangle|}{2^{-n(\gamma-|k|_\s)}}\,\mathrm{d}h \\ &\lesssim 
    \sum_{\zeta<|k|_\s-\beta}\sum_{m < n}2^{-m(\zeta+\beta-|k|_\s)}
    \mint{-}_{Q(0,4\cdot 2^{-n})}\frac{|\bf(x{+}h)-\Gamma_{x+h,x}\bf(x)|_\zeta}
    {2^{-n(\gamma-|k|_\s)}}\,\mathrm{d}h \\ &\lesssim \sum_{\zeta<|k|_\s-\beta}\sum_{m < n}
    2^{(n-m)(\zeta+\beta-|k|_\s)}\mint{-}_{Q(0,4\cdot 2^{-n})}
    \frac{|\bf(x{+}h)-\Gamma_{x+h,x}\bf(x)|_\zeta}{2^{-n(\gamma-\zeta)}}\,\mathrm{d}h,
\end{align*}
which in turn immediately entails a bound of requested type. This leaves to derive a bound for
the term incorporating the test functions $K^{k,\gamma+\beta}_{n,x,y}$. Here, we first recall
the following scaled version of Taylor's theorem as presented in \cite[Prop A.1]{Hairer:2014}.
Let $e_i$ denote the $i$th standard unit vector of $\R^d$ and let, for $\gamma'>0$,
\begin{align*}
    \partial\gamma' = \{l\in\N^d_0\setminus\{0\}\colon |l|_\s\geq\gamma',\,|l-e_{\mathfrak{m}(l)}|_\s<\gamma'\},
\end{align*}
where $\mathfrak{m}(l)=\inf\{i\in\{1,\ldots,d\}\colon l_i\neq 0\}$. Then, it holds
\begin{align}\label{scaledTaylor}
    K^{k,\gamma+\beta}_{m,x,x+h} = \sum_{\substack{l\in\N^d_0\setminus\{0\} \\ k+l\in\partial(\gamma+\beta)}}
    \int_{\R^d}D^{k+l}_1K_m(x{+}z,\cdot)\,\mu^l(h,\mathrm{d}z),
\end{align}
where $\mu^l(h,\mathrm{d}z)$ is a signed measure on $\R^d$ with total mass $h^l/l!$
and support in $\{z\in\R^d\colon z_i\in [0,h_i]\}$. Since $\gamma+\beta\notin\N$ by assumption,
the set $\partial(\gamma+\beta)$ is actually identical to 
$\{l\in\N^d_0\setminus\{0\}\colon |l|_\s>\gamma+\beta,\,|l-e_{\mathfrak{m}(l)}|_\s<\gamma+\beta\}$.
Now, one can find an absolute constant $C>0$ such that
\begin{align*}
    \sum_{m<n}&\mint{-}_{Q(0,4\cdot 2^{-n})}
    \frac{|\langle\Recon\bf-\Pi_x\bf(x),K^{k,\gamma+\beta}_{m,x,x+h}\rangle|}{2^{-n(\gamma-|k|_\s)}}\,\mathrm{d}h \\
    &\lesssim \sum_{\substack{l\in\N^d_0\setminus\{0\} \\ k+l\in\partial(\gamma+\beta)}}\sum_{m<n}
    2^{-m(\beta-|k+l|_\s)}2^{-n|l|_\s}\sup_{\eta\in\B^r_C(\R^d)}
    \frac{|\langle\Recon\bf-\Pi_x\bf(x),\eta^{2^{-m}}_x\rangle|}{2^{-n(\gamma-|k|_\s)}} \\
    &\lesssim \sum_{\substack{l\in\N^d_0\setminus\{0\} \\ k+l\in\partial(\gamma+\beta)}}\sum_{m<n}
    2^{(n-m)(\gamma+\beta-|k+l|_\s)}\sup_{\eta\in\B^r_C(\R^d)}
    \frac{|\langle\Recon\bf-\Pi_x\bf(x),\eta^{2^{-m}}_x\rangle|}{2^{-m\gamma}}.
\end{align*}
By means of Young's inequality as well as Remark \ref{reconstruction_bound2}, this is again sufficient to obtain a bound of
desired order. In particular, the asserted bound with regard to the $\F^{\gamma+\beta}_{p,q}$-norm
of $\K_\gamma\bf$ eventually follows.

It remains to verify the asserted identity \eqref{ReconConvolution}. Of course, one could opt for a proof 
which adapts the arguments developed in \cite{Hairer:2017}
to the situation of the scale $\F^\gamma_{p,q}$. But we prefer at this point to make
use of the embeddings $\F^{\gamma+\beta}_{p,q}\subset\mathcal{B}^{\gamma+\beta}_{p,q\vee p}$
and $\F^{\gamma}_{p,q}\subset\mathcal{B}^{\gamma}_{p,q\vee p}$. Indeed, just note that
all involved constructions coincide on their respective spaces, i.e.\ the asserted identity
is satisfied as it in turn already holds true for the scale of spaces $\mathcal{B}^\gamma_{p,q}$.
\end{proof}

\section{Products of modelled distributions}\label{sec:products}
An essential tool in the theory of regularity structures is the possibility to build products between modelled distributions. In this section we address this issue in the framework of Triebel-Lizorkin spaces. We recall the definition of an abstract product in a regularity structure as in \cite{Hairer:2014}.

\begin{definition}[Sector]\label{def:sector}
Given a regularity structure $(\A,\T,\G)$ a set $V \subset\T$ is called a sector if $\Gamma \bTau \in V$ for any $\bTau \in V , \ \Gamma \in \G$ and we can write $V = \bigoplus_{\alpha \in \A} V_{\alpha}$ with $V_{\alpha} \subset \T_{\alpha}$.
\end{definition}

\begin{definition}[Product]\label{def:product}
Consider a regularity structure $(\T, \A, \G)$ with two sectors $V, \overline{V}$. 
A product between $V$ and $\overline{V}$ is a map:
$$
\star\colon V \times \overline{V} \to \T
$$
such that, for any $\alpha\in\A$ and $\beta\in\A$, its restriction to $V_\alpha\times\overline{V}_\beta$
is a continuous bilinear map $\star\colon V_\alpha\times\overline{V}_\beta\to\T_{\alpha+\beta}$. 
Furthermore, it is required that the identity
\begin{align*}
\Gamma (\bTau \star \Taubar) = \Gamma \bTau \star \Gamma \Taubar
\end{align*}
holds true for all $\Gamma\in\G$, all $\bTau \in V_{\alpha}$ and all $\Taubar \in \overline{V}_{\beta}$.
\end{definition}
Finally, for a given parameter $\alpha$ we say that $\bf \in \F^{\gamma, \alpha}_{p,q}$ 
if it lies in $\F^{\gamma}_{p,q}$ and takes values in the elements of the regularity structure 
that have homogeneity of at least $\alpha,$ namely $\bf \in \T_{\ge \alpha}.$ In the framework 
provided by the above definitions we can prove the following result.

\begin{proposition}\label{prop:product}
Consider a regularity structure endowed with a product. For any two functions 
$\bf$ in $\F^{\gamma_1, \alpha_1}_{p_1, q_1}$ and $\bg$ in 
$\F^{\gamma_2, \alpha_2}_{p_2, q_2}$ with $1 \le p_i < \infty$ and 
$1 \le q_i \le \infty$ the pointwise product $\bf \bg (x) := \bf(x) \star \bg(x)$ lies in 
$\F^{\gamma, \alpha}_{p, q}$ with
\begin{align*}
\alpha = \alpha_1 {+} \alpha_2, \ \gamma = (\gamma_1 {+} \alpha_2) 
\wedge (\gamma_2 {+} \alpha_1), \ p = \frac{p_1 p_2}{p_1 {+} p_2}, \ q = q_1\vee q_2,
\end{align*}
provided that $p\geq 1.$ In addition we have the bound on the norms:
\begin{align*}
\vvvert\bf \bg\vvvert_{\F^{\gamma}_{p,q}} \lesssim 
\vvvert\bf\vvvert_{\F^{\gamma_1}_{p_1, q_1}}\vvvert\bg\vvvert_{\F^{\gamma_2}_{p_2, q_2}}.
\end{align*}
\end{proposition}

\begin{proof}
We articulate the proof in the following way. First we consider a sequence of 
functions $\bpi^{(n)}: \Lambda_n \to \T^{-}_{\gamma}$ which is a discrete version of the product, namely:
\begin{align*}
\bpi^{(n)}(x) = \fbar^{(n)}(x) \star \gbar^{(n)}(x).
\end{align*}
Note that this definition may produce contributions in $\T_{\geq\gamma}$. Since the statement
of the proposition only requires us to encode the product up to homogeneities $<\gamma$,
we set all contributions in $\T_{\geq\gamma}$ to zero, i.e.\ we will actually study 
\begin{align*}
\bpi^{(n)}_{<\gamma}(x) := 
\sum_{k+l<\gamma}\Q_k\fbar^{(n)}(x)\star\Q_l\gbar^{(n)}(x). 
\end{align*}
Again, since the definitions of $\fbar^{(n)}$ and $\gbar^{(n)}$ in \eqref{average}
incorporate re-expansions due to the action of the elements $\Gamma\in G$, $\bpi^{(n)}_{<\gamma}(x)$
may have contributions in $\T_{<\alpha}$. Thus, we aim to prove that 
$\bpi^{(n)}_{<\gamma} \in \bar{\F}^{\gamma}_{p,q}$. Then, Proposition \ref{prop:back} 
implies that $\bpi^{(n)}_{<\gamma}$ corresponds to a function $\bpi_{<\gamma} \in \F^{\gamma}_{p,q}$,
which again may have contributions in $\T_{<\alpha}$.
Finally, it thus suffices to show that $\Q_{\geq\alpha}\bpi_{<\gamma} = \bf \bg$ in order
to conclude the proof. 

In the following calculations we will use 
that uniformly over $\zeta$ in a bounded set, there exist only a finite number of 
homogeneities $\zeta_1, \zeta_2 \in \A$ with $\zeta_i \ge \alpha_i$ such that 
$\zeta_1{+} \zeta_2 = \zeta$. The global bound follows directly from H\"older's inequality, 
so let us concentrate on the translation bound for $\bpi^{(n)}_{<\gamma}$. Fix $\zeta<\gamma$. 
Let us compute $\bpi^{(n)}_{<\gamma}(y{+}h)\, - \Gamma_{y{+}h, y}\bpi^{(n)}_{<\gamma}(y)$.
This is a purely algebraic manipulation:
\begin{align*}
&\bpi^{(n)}_{<\gamma}(y{+}h)\, - \Gamma_{y{+}h, y}\bpi^{(n)}_{<\gamma}(y) \\
&=\sum_{k+l<\gamma}\Big\{\Q_k\fbar^{(n)}(y{+}h)\star\Q_l\gbar^{(n)}(y{+}h)-
\Gamma_{y+h,y}\Q_k\fbar^{(n)}(y)\star\Gamma_{y+h,y}\Q_l\gbar^{(n)}(y)\Big\} \\
&=\sum_{k+l<\gamma}\Q_k\big(\fbar^{(n)}(y{+}h){-}\Gamma_{y+h,y}\fbar^{(n)}(y)\big)\star\Q_l\gbar^{(n)}(y{+}h) \\
&\hspace{1cm}+\sum_{k+l<\gamma}\Q_k\Gamma_{y+h,y}\fbar^{(n)}(y)\star
\Q_l\big(\gbar^{(n)}(y{+}h){-}\Gamma_{y+h,y}\gbar^{(n)}(y)\big) \\
&\hspace{1cm}+\mathrm{Res}_\gamma(\fbar^{(n)},\gbar^{(n)}), 
\end{align*}
where we introduced the ``error term''
\begin{align*}
\mathrm{Res}_\gamma(\fbar^{(n)},\gbar^{(n)}) &:= 
\sum_{k+l<\gamma}\Q_k\Gamma_{y+h,y}\fbar^{(n)}(y)\star\Q_l\Gamma_{y+h,y}\gbar^{(n)}(y) \\
&\hspace{1cm}-\sum_{k+l<\gamma}\Gamma_{y+h,y}\Q_k\fbar^{(n)}(y)\star\Gamma_{y+h,y}\Q_l\gbar^{(n)}(y) \\
&=\Q_{<\gamma}\bigg(\sum_{k+l\geq\gamma}\Gamma_{y+h,y}\Q_k\fbar^{(n)}(y)
\star\Gamma_{y+h,y}\Q_l\gbar^{(n)}(y) \bigg).
\end{align*}
To obtain the first identity, we made use of the property that the product commutes
with the action of the elements $\Gamma\in\G$ from the structure group.
Hence, we may bound
\begin{align*}
|\bpi^{(n)}_{<\gamma}&(y{+}h)\, - \Gamma_{y{+}h, y}\bpi^{(n)}_{<\gamma}(y)|_\zeta \\
&\leq\sum_{\zeta_1+\zeta_2=\zeta}|\fbar^{(n)}(y{+}h){-}\Gamma_{y+h,y}\fbar^{(n)}(y)|_{\zeta_1}
|\gbar^{(n)}(y{+}h)|_{\zeta_2} \\
&\hspace{1cm}+\sum_{\zeta_1+\zeta_2=\zeta}|\Gamma_{y+h,y}\fbar^{(n)}(y)|_{\zeta_1}
|\gbar^{(n)}(y{+}h){-}\Gamma_{y+h,y}\gbar^{(n)}(y)|_{\zeta_2} \\
&\hspace{1cm}+|\mathrm{Res}_\gamma(\fbar^{(n)},\gbar^{(n)})|_\zeta.
\end{align*}
We proceed by estimating each of these contributions. Let us begin with the derivation of an 
appropriate bound for the last term. To this end, note first that due to Assumption \ref{gamma}
the sum defining $\mathrm{Res}_\gamma(\fbar^{(n)},\gbar^{(n)})$ actually runs over all
homogeneities $k\in\A_{\gamma_1}$ and $l\in\A_{\gamma_2}$ such that $k+l>\gamma$.
Hence, we find $\varepsilon>0$ such that
\begin{align*}
|\mathrm{Res}_\gamma(\fbar^{(n)},\gbar^{(n)})|_\zeta &\lesssim
\sum_{k+l>\gamma}\sum_{\zeta_1+\zeta_2=\zeta}\|h\|_{\s	}^{k+l-\zeta_1-\zeta_2}
|\fbar^{(n)}(y)|_k|\gbar^{(n)}(y)|_l \\
&\lesssim 2^{-n\varepsilon}\sum_{k+l>\gamma}2^{-n(\gamma-\zeta)}|\fbar^{(n)}(y)|_k|\gbar^{(n)}(y)|_l,
\end{align*}
uniformly over all $n\geq 0$, all $y\in\Lambda_n$, all $h\in\mathcal{E}_n$ and all $\zeta\in\A_\gamma$.
From this we can deduce with the help of H\"{o}lder's inequality the bound
\begin{align*}
&\bigg\| \bigg( \sum_{ n\ge 0 } 
\Big\vert \sum_{y \in \Lambda_n} \sum_{h \in \E_n}\, 
\frac{|\mathrm{Res}_\gamma(\fbar^{(n)},\gbar^{(n)})|_\zeta}{2^{-n(\gamma - \zeta)}}
\chi^n_y(x) \Big\vert^q \bigg)^{\frac{1}{q}}\bigg\|_{L^p} \\
&\lesssim \bigg(\sum_{n\geq 0}2^{-n\varepsilon q}\bigg)^\frac{1}{q}\sup_{k+l>\gamma}
\bigg\|\sup_{n\geq 0}\sum_{y\in\Lambda_n}\sum_{h\in\mathcal{E}_n}
|\fbar^{(n)}(y)|_{k}\chi^n_y(x)\bigg\|_{L^{p_1}}
\bigg\|\sup_{n\geq 0}\sum_{y\in\Lambda_n}\sum_{h\in\mathcal{E}_n}
|\gbar^{(n)}(y)|_{l}\chi^n_y(x)\bigg\|_{L^{p_2}},
\end{align*}
which is of required order due to Remark \ref{rem:sup_Lp}.
Regarding the first term, we immediately obtain the bound
\begin{align*}
\bigg\| \bigg( \sum_{ n\ge 0 } 
\Big\vert \sum_{y \in \Lambda_n} \sum_{h \in \E_n}&\, \frac{|\fbar^{(n)} (y{+}h) - 
\Gamma_{y{+}h, y} \fbar^{(n)}|_{\zeta_1} }{2^{-n(\gamma - \zeta)}}|\gbar^{(n)}(y{+}h)|_{\zeta_2} 
\chi^n_y(x) \Big\vert^q \bigg)^{\frac{1}{q}}\bigg\|_{L^p} \\
&\leq\bigg\|\sup_{n\geq 0}\sum_{y\in\Lambda_n}\sum_{h\in\mathcal{E}_n}
|\gbar^{(n)}(y{+}h)|_{\zeta_2}\chi^n_y(x)\bigg\|_{L^{p_2}}
\vvvert\fbar\vvvert_{\bar{F}^{\gamma_1}_{p_1,q_1}},
\end{align*}
uniformly over $\zeta_1+\zeta_2=\zeta$; and which is therefore
of required order due to Remark \ref{rem:sup_Lp} and Proposition \ref{prop:forth}.
Analogously, one obtains a bound of requested order for the second term.
This establishes the translation bound.
Since the consistency bound follows in exactly the same way, we conclude that 
$\bpi^{(n)}_{<\gamma} \in \bar{\F}^{\gamma}_{p,q}.$ 
This implies in particular that there exists a function $\bpi:=\bpi_{<\gamma} \in \F^{\gamma}_{p,q}$ 
such that $$\lim_{n\to\infty} \Q_{\zeta} \bpi^{(n)}_{<\gamma} = \Q_{\zeta}\bpi \ \text{ in } L^p. $$ Now, 
since we know that $\bpi^{(n)}_{<\gamma} = \sum_{k+l<\gamma}\Q_l\bf^{(n)}\star\Q_k\bg^{(n)}$ and 
$$\Q_{\zeta} \bf^{(n)} \to \Q_{\zeta}\bf \ \text{ in } L^{p_1}, \ \ 
\Q_{\zeta} \bg^{(n)} \to \Q_{\zeta}\bg  \ \text{ in } L^{p_2},$$
it follows from H\"older's inequality that $\Q_{\zeta}\bpi^{(n)}_{<\gamma} \to \Q_{\zeta}\bf \bg$ in $L^p$
for all homogeneities $\zeta\in [\alpha,\gamma)\cap\A$. This completes the proof of the result.
\end{proof}

\begin{remark}
The condition that $p$ shall at least be $1$ in the previous theorem is not necessary. However, 
stating the theorem in the general case would require defining Triebel-Lizorkin spaces in the case $p, q \in (0,1).$
Since the article is already quite lengthy, we refrained from doing so.
\end{remark}

\begin{remark}
The previous result together with Theorem \ref{reconstruction_theorem} applied to 
the polynomial regularity structure allows to deduce that the product of two 
Triebel-Lizorkin distributions is a well-defined continuous bilinear map provided 
that  $\gamma > 0$; thus recovering a well-known result from harmonic analysis.
\end{remark}

\section{A note on the Besov scale} \label{sec:Besov}
In this section, we want to discuss a phenomenon which is well known from the classical theory
of Besov spaces, namely that one obtains the same space no matter whether one considers 
differences of a function or volume means of differences. We aim to show that this is
still the case in the framework of regularity structures.

\begin{definition}
Let $(\A,\T,\mathcal{G})$ be a regularity structure, and let $(\Pi,\Gamma)$ be a model.
Consider $1\leq p \leq\infty$, $1\leq q \leq \infty$ as well as $\gamma > 0$. 
Then, we let $\mathcal{D}_{p,q}^\gamma$ be the (Banach) space of all functions 
$\bf\colon\R^d\to\T_{\gamma}^-$ such that
\begin{itemize}
\item[i)] $\displaystyle\sup_{\zeta\in\A_{\gamma}}\big\||\bf(x)|_\zeta\big\|_{L^p} < \infty$,
\item[ii)] $\displaystyle\sup_{\zeta\in\A_{\gamma}}\bigg\|\Big\|
\mint{-}_{Q(0,4\lambda)}\frac{|\bf(x{+}h)-\Gamma_{x+h,x}\bf(x)|_\zeta}
{\lambda^{\gamma-\zeta}}\,\mathrm{d}h\Big\|_{L^p}\bigg\|_{L^q_\lambda}<\infty$.
\end{itemize}
The associated norm for $\bf\in\mathcal{D}_{p,q}^\gamma$ is denoted by 
$\vvvert\bf\vvvert_{\mathcal{D}_{p,q}^\gamma}$.
\end{definition}

For the special case $p<\infty$ and $q=\infty$, the space 
$\mathcal{D}^\gamma_{p,q}$ was already introduced and studied in the work 
of Hairer and Labb\'{e} \cite{Hairer:2015} on multiplicative stochastic heat equations.

\begin{definition}\label{def:discreteBesov}
Let $(\A,\T,\mathcal{G})$ be a regularity structure, and let $(\Pi,\Gamma)$ be a model.
Consider $1\leq p \leq\infty$, $1\leq q \leq \infty$ as well as $\gamma > 0$.
We denote by $\bar{\mathcal{B}}^\gamma_{p,q}$ the (Banach) space of all sequences of maps 
\begin{align*}
\fbar^{(n)}\colon\Lambda_n\to\T_{\gamma}^-,\quad n\geq 0,
\end{align*}
such that, uniformly over all $\zeta\in\A_{\gamma}$,
\begin{itemize}
\item[i)] $\displaystyle\Big(\sum_{y\in\Lambda_0}
\big|\fbar^{(0)}(y)\big|_\zeta^p\Big)^\frac{1}{p} < \infty$,
\item[ii)] $\displaystyle\bigg(\sum_{n\geq 0}\sum_{h\in\E_{n}}\bigg\|
\frac{|\fbar^{(n)}(y{+}h)-\Gamma_{y{+}h,y}\fbar^{(n)}(y)|_\zeta}
{2^{-n(\gamma-\zeta)}}\bigg\|_{\ell^p_n}^q\bigg)^\frac{1}{q}<\infty$,
\item[iii)] $\displaystyle\bigg(\sum_{n\geq 0}\bigg\|
\frac{|\fbar^{(n)}(y)-\fbar^{(n+1)}(y)|_\zeta}
{2^{-n(\gamma-\zeta)}}\bigg\|_{\ell^p_n}^q\bigg)^\frac{1}{q}<\infty$.
\end{itemize}
The associated norm for an element $\fbar\in\bar{\mathcal{B}}_{p,q}^\gamma$ will be denoted by 
$\vvvert\fbar\vvvert_{\bar{\mathcal{B}}^\gamma_{p,q}}$.
\end{definition}

Again, there is essentially no difference between the spaces $\D$ and $\bar{\mathcal{B}}$.
This is the content of the following result.

\begin{proposition}\label{prop:sequence}
Let $\gamma>0$ and $1\leq p,q\leq\infty$. Given $\bf\in\mathcal{D}^\gamma_{p,q}$
and $n\geq 0$, we define the maps $\fbar^{(n)}\colon\Lambda_n\to\T^-_\gamma$ via
\begin{align}\label{average2}
\fbar^{(n)}(y) = \mint{-}_{Q(y,2^{-n})}\Gamma_{y,z}\bf(z)\,\mathrm{d}z.
\end{align}
We then have $\fbar\in\bar{\mathcal{B}}^\gamma_{p,q}$ and 
$\vvvert\fbar\vvvert_{\bar{\mathcal{B}}^\gamma_{p,q}}\lesssim
\vvvert\bf\vvvert_{\mathcal{D}^\gamma_{p,q}}$. In addition, it holds
\begin{align*}
    \Gamma_{x,x_n}\fbar^{(n)}(x_n) \to \bf(x) \quad\text{in } L^p.
\end{align*}
On the other side, for any $\fbar\in\bar{\mathcal{B}}^\gamma_{p,q}$ there is 
$\bf\colon\R^d\to\T^-_\gamma$ such that
\begin{align*}
    \bf_n(x):=\Gamma_{x,x_n}\fbar^{(n)}(x_n) \to \bf(x) \quad\text{in } L^p,
\end{align*}
where actually $\bf\in\mathcal{D}^\gamma_{p,q}$ with $\vvvert\bf\vvvert_{\mathcal{D}^\gamma_{p,q}}
\lesssim\vvvert\fbar\vvvert_{\bar{\mathcal{B}}^\gamma_{p,q}}$.
\end{proposition}

\begin{proof}
    \textit{Step 1:} Let us prove the first assertion in the statement above, i.e.\
    we fix a modelled distribution $\bf\in\D^\gamma_{p,q}$ and we aim to show that
    $\fbar\in\bar{\mathcal{B}}^\gamma_{p,q}$ where $\fbar$ is defined as in \eqref{average2}.
    Since the respective local bound follows directly, we immediately turn to the
    respective translation bound. For this, we observe that
    \begin{align*}
        &\frac{|\fbar^{(n)}(y{+}h)-\Gamma_{y+h,y}\fbar^{(n)}(y)|_\zeta}{2^{-n(\gamma-\zeta)}}\\
        &\hspace{0.3cm}\lesssim \mint{-}_{Q(y,2^{-n})}\mint{-}_{Q(x,2^{-n})}
        \frac{|\Gamma_{y+h,z+h+y-x}(\bf(z{+}h{+}y{-}x)-\Gamma_{z+h+y-x,x}\bf(x))|_\zeta}{2^{-n(\gamma-\zeta)}}
				\,\mathrm{d}z\,\mathrm{d}x\\
        &\hspace{0.3cm}\lesssim \sum_{\beta\geq\zeta}
        \mint{-}_{Q(y,2^{-n})}\mint{-}_{Q(x,2\cdot 2^{-n})}
        \frac{|\bf(z{+}h)-\Gamma_{z+h,x}\bf(x)|_\beta}{2^{-n(\gamma-\beta)}}
				\,\mathrm{d}z\,\mathrm{d}x.
    \end{align*}
    But this already leads to a bound of desired order. The consistency bound can be derived
    analogously. The convergence assertion follows from the same argument as the corresponding one
    for the Triebel--Lizorkin scale $\F^\gamma_{p,q}$.
    
    \textit{Step 2:} 
    Let us turn to the second part of the statement, i.e.\ consider $\fbar\in\bar{\mathcal{B}}^\gamma_{p,q}$.
    Along the same lines as for the Triebel--Lizorkin scale, one obtains the bound
    \begin{align*}
        \sum_{n\geq n_0}\big\||\bf_{n+1}(x)-\bf_n(x)|_\zeta\big\|_{L^p}\lesssim 2^{-n_0(\gamma-\zeta)} 
        \vvvert\fbar\vvvert_{\bar{\mathcal{B}}^\gamma_{p,q}},
    \end{align*}
    uniformly over all $n_0\geq 0$. Denote by $\bf\colon\R^d\to\T^-_\gamma$ the associated
    limit in $L^p$. We aim to show that $\bf\in\mathcal{D}^\gamma_{p,q}$ with a corresponding
    bound for its norm. To this end, it is again convenient to bound
    \begin{align*}
        \bigg\|\Big\|\mint{-}_{Q(0,4\lambda)}&
        \frac{|\bf(x{+}h)-\Gamma_{x+h,x}\bf(x)|_\zeta}
        {\lambda^{\gamma-\zeta}}\,\mathrm{d}h\Big\|_{L^p}\bigg\|_{L^q_\lambda} \\
        &\lesssim \bigg(\sum_{n\geq 0}\bigg\|\mint{-}_{Q(0,4\cdot2^{-n})}
        \frac{|\bf(x{+}h)-\Gamma_{x+h,x}\bf(x)|_\zeta}
        {2^{-n(\gamma-\zeta)}}\,\mathrm{d}h\bigg\|^q_{L^p}\bigg)^\frac{1}{q}.
    \end{align*}
    Now, we make use of the decomposition from \eqref{decomp}. Contributions from the last two terms
    can be treated along the same lines as in the case of the Triebel--Lizorkin scale. Thus, let us
    only discuss the contribution due to the first term, i.e.\ $\bf(x{+}h)-\bf_n(x{+}h)$ with $x\in\R^d$,
    $n\geq 0$ and $h\in Q(0,4\cdot 2^{-n})$. Then, we distinguish between two cases. Let us first
    discuss the case where $q/p\geq 1$. In this case, we simply bound (with obvious modification
    if $p=\infty$ and/or $q=\infty$) 
    \begin{align*}
        \bigg\|&\mint{-}_{Q(0,4\cdot2^{-n})}
        \frac{|\bf(x{+}h)-\bf_n(x{+}h)|_\zeta}{2^{-n(\gamma-\zeta)}}\,\mathrm{d}h\bigg\|_{L^p}^q \\
        &\qquad\lesssim \mint{-}_{Q(0,4\cdot2^{-n})}
        \Big\|\frac{|\bf(x{+}h)-\bf_n(x{+}h)|_\zeta}{2^{-n(\gamma-\zeta)}}\Big\|^q_{L^p}\,\mathrm{d}h
        =\Big\|\frac{|\bf(x)-\bf_n(x)|_\zeta}{2^{-n(\gamma-\zeta)}}\Big\|^q_{L^p}.
    \end{align*}
    On the other side, if $q/p<1$ and therefore in particular $p>1$, one can employ the Hardy--Littlewood
    maximal inequality to obtain the same type of bound. Thus, no matter which case applies, we have
    reduced the argument to the situation of the third term on the right hand side of \eqref{decomp}
    which concludes the proof.
\end{proof}

Recall from \cite{Hairer:2017} that $\vvvert\bf\vvvert_{\mathcal{B}^\gamma_{p,q}}
\sim\vvvert\fbar\vvvert_{\bar{\mathcal{B}}^\gamma_{p,q}}$, where $\fbar$ is defined by the local averages
as in \eqref{average2}. We therefore obtain the announced result from the beginning of this section.

\begin{corollary}
Let $1\leq p,q\leq \infty$. Then $\D^\gamma_{p,q} = \mathcal{B}^\gamma_{p,q}$ in the
sense of equivalent norms.
\end{corollary}


\end{document}